\documentclass[11pt,a4paper, leqno]{article}
\usepackage{amsmath}
\usepackage{amsthm}
\usepackage{amsfonts}
\usepackage{graphics}

\makeatletter
\def\multilimits@{\bgroup
  \Let@
  \restore@math@cr
  \default@tag
 \baselineskip\fontdimen10 \scriptfont\tw@
 \advance\baselineskip\fontdimen12 \scriptfont\tw@
 \lineskip\thr@@\fontdimen8 \scriptfont\thr@@
 \lineskiplimit\lineskip
 \vbox\bgroup\ialign\bgroup\hfil$\m@th\scriptstyle{##}$\hfil\crcr}
\def\Sb{_\multilimits@}
\def\Sp{^\multilimits@}
\def\endSb{\crcr\egroup\egroup\egroup}

\def\smallmatrix{\null\,\vcenter\bgroup
 \Let@\restore@math@cr\default@tag
 \baselineskip6\ex@ \lineskip1.5\ex@ \lineskiplimit\lineskip
 \ialign\bgroup\hfil$\m@th\scriptstyle{##}$\hfil&&\thickspace\hfil
 $\m@th\scriptstyle{##}$\hfil\crcr}
\def\endsmallmatrix{\crcr\egroup\egroup\,}
\newcount\c@MaxMatrixCols
\makeatother
\theoremstyle{plain}
\newtheorem{theorem}{Theorem}[section]
\newtheorem{lemma}[theorem]{Lemma}

\theoremstyle{definition}
\newtheorem{definition}[theorem]{Definition}
\newtheorem{example}[theorem]{Example}

\theoremstyle{remark}
\newtheorem{remark}[theorem]{Remark}
\newtheorem{notation}[theorem]{Notation}

\numberwithin{equation}{section}
\def\grad{\mathop{\rm grad}\nolimits} 				      
\def\rank{\mathop{\rm rank}}
\def\tr{\mathop{\rm tr}}
\def\Hom{\mathop{\rm Hom}\nolimits}
\def\Rep{\mathop{\rm Rep}\nolimits}
\def\Ext{\mathop{\rm Ext}\nolimits}

\def\Ker{\mathop{\rm Ker}\nolimits}
\def\Coker{\mathop{\rm Coker}\nolimits}

\def\Z{\mathbb{Z}}
\def\Q{\mathbb{Q}}
\def\R{\mathbb{R}}
\def\C{\mathbb{C}}

\def\journame{\itshape}
\def\bookname{\itshape}
\def\bs{\boldsymbol}

\def\lra{\longrightarrow}
\def\lla{\longleftarrow}
\newcommand{\vt}[1]{\overset{#1}{\cdot}}
\newcommand{\ovt}[1]{\overset{#1}{\circ}}

\def\journame{\itshape}
\def\bookname{\itshape}

\def\address#1#2{\begingroup
\noindent\parbox[t]{8.5cm}{%
\small{\scshape\ignorespaces#1}\par\vskip1ex
\noindent\small{\itshape E-mail}%
\/: #2\par\vskip4ex}\hfill%
\endgroup}%

\begin{document}
\title{$b$-Functions associated with quivers of type $A$ \\
{\normalsize{\it Dedicated to the 60th birthday of Professor Tatsuo
Kimura}}}
\author{Kazunari Sugiyama}
\date{\today}
%\dedicatory{Dedicated to the 60th birthday of Professor Tatsuo Kimura}
\maketitle

\begin{abstract}
 We study the $b$-functions of relative invariants of
 the prehomogeneous vector spaces associated with quivers of
 type $A$. By applying the decomposition formula
 for $b$-functions, 
 we determine explicitly the $b$-functions of one variable
 for each irreducible relative invariant. 
 Moreover, we give a graphical algorithm to determine
 the $b$-functions of several variables. 
\end{abstract}

\section*{Introduction}

We start with a classical formula
\begin{equation}
  \label{form:Cayley'sFormula}
 \det\left(\frac{\partial}{\partial v}      
 \right)\det (v)^{s+1} = (s+1)(s+2)\cdots (s+n) \cdot \det (v)^{s}
 \qquad (v\in M(n)), 
\end{equation}
where $M(n)$ is the set of square matrices of degree $n$, and
$\det\left({\partial}/{\partial v}\right)$ is the differential
operator obtained by replacing each variable $v_{ij}$ in $\det(v)$
with $\partial/\partial v_{ij}$.
From a modern view point, 
the identity~\eqref{form:Cayley'sFormula} can be regarded
as an example of $b$-functions of prehomogeneous vector spaces.
In the present paper, we generalize \eqref{form:Cayley'sFormula} to
the relative invariants of 
the prehomogeneous vector spaces associated with
quivers of type $A$. For this purpose, we use the result of
\cite{SatoSugi}, which asserts that under certain conditions,
$b$-functions of reducible prehomogeneous vector spaces have
decompositions correlated to the decomposition of representations.
Moreover, in the latter half, we describe a graphical 
algorithm to calculate the $b$-functions of several variables.

Now what are relative invariants of 
prehomogeneous vector spaces associated with quivers of
type $A$? Let $Q$ be a quiver of type $A_{r}$.
In Introduction, {\it for simplicity}, we assume that $Q$
is of the following form:
\[
  Q: \quad \vt{1} \lra \vt{2}\lra \vt{3} \lra \cdots \lra \vt{r}
\]
Fix a dimension vector $\underline{n}=(n_{1},\dots, n_{r})\in
\Z_{>0}^{r}$ and put
\begin{align*}
 GL(\underline{n})&= GL(n_{1})\times GL(n_{2})\times\cdots \times
 GL(n_{r}), \\
 \Rep(Q, \underline{n}) &=
 M(n_{2}, n_{1})\oplus M(n_{3},n_{2})\oplus\cdots \oplus
 M(n_{r}, n_{r-1}).
\end{align*}
The action of 
$GL(\underline{n})$ on $\Rep(Q,\underline{n})$ is given as follows:
for $g = (g_{1},\dots, g_{r})\in GL(\underline{n})$ and
$v = \left(X_{2,1}, X_{3,2},\dots,
X_{r,r-1}\right)\in\Rep(Q,\underline{n})$, 
\[
 g\cdot v =
 \left(g_{2}X_{2,1}g_{1}^{-1}, \, g_{3}X_{3,2}g_{2}^{-1},\,
 \dots, g_{r}X_{r,r-1}g_{r-1}^{-1} \right). 
\]
Then $(GL(\underline{n}), \Rep(Q,\underline{n}))$ is a prehomogeneous
vector space, i.e.,  there exists an open dense $GL(\underline{n})$-orbit in
$\Rep(Q, \underline{n})$. 
Now we define the set $I_{\underline{n}}(Q)$ by
\begin{equation}
 \label{form:DefOfINQIntro}
 I_{\underline{n}}(Q) =
 \{(p,q)\, ;\, 1\leq p< q\leq r, \; n_{p}=n_{q} \;\;
 \text{and}\;\; n_{t}> n_{p} \;\;
 \text{for any} \;\; t=p+1,\dots, q-1\}.
\end{equation}
Then, for $(p,q)\in I_{\underline{n}}(Q)$,
\[
 f_{(p,q)}(v) = \det \left(X_{q,q-1}X_{q-1,q-2}\cdots X_{p+1,p}
 \right)
\]
is a non-zero relative invariant of
$(GL(\underline{n}), \Rep(Q, \underline{n}))$, and
$f_{(p,q)}(v)$ $((p,q)\in
I_{\underline{n}}(Q))$ are the fundamental relative invariants. 
These polynomials are known as ``determinantal semi-invariants''
(Schofield~\cite{Schofield}).
By the general theory of prehomogeneous vector spaces, there exists a
polynomial $b_{(p,q)}(s)\in \C[s]$ satisfying
\[
 f_{(p,q)}\left(\frac{\partial}{\partial v}\right)
 f_{(p,q)}(v)^{s+1} = b_{(p,q)}(s)\cdot
 f_{(p,q)}(v)^{s}.
\]
Our first result is the calculation of $b_{(p,q)}(s)$.
\begin{theorem}[Theorem{~\ref{thm:Bfun}}, equioriented case]
 \[
  b_{(p,q)}(s) = \prod_{t=p+1}^{q}
 \prod_{\lambda=1}^{n_{p}}(s+n_{t}-n_{p}+\lambda).
 \]
\end{theorem}
In Theorem{~\ref{thm:Bfun}}, the $b$-functions
$b_{(p,q)}(s)$ for quivers $Q$ of type $A$ with {\it arbitrary} orientation
are determined. This result is a simple application of the
decomposition formula found by  F.~Sato and the author~\cite{SatoSugi}. 

Our second result is on the $b$-functions of
several variables. See Lemma~\ref{prop:DefOfBfunSeverVari} for the
definition of $b$-functions of several variables.
In general, two different relative invariants may have common
variables, and this causes a serious combinatorial difficulty.
As is easily guessed from \eqref{form:DefOfINQIntro}, even in
the equioriented case, 
it is unbelievably tedious to enumerate all the patterns of
occurrence of relative invariants which have common variables. 
Instead of doing this, the author proposes an algorithm to use
diagrams, which are recently called ``lace diagrams''
(Buch-Rim\'{a}ny~\cite{BR},
Knutson-Miller-Shimozono~\cite{KMS}). To apply our algorithm, we draw
some diagram for each 
irreducible relative invariant, and then simply
{\it superpose} them. From the resultant diagram, the $b$-function
of several variables can be easily calculated. 
See Figures~\ref{fig:LDinFirstExample}, 
\ref{fig:SuperpositionInFirstExample}, 
\ref{figure:SecondExample}, and
\ref{fig:SuperpositionInSecondExample} in
Section~\ref{section:BfunctionsOfSeveralVariables}

We remark that in \cite{GS, Sevenheck}, the $b$-functions of linear free
divisors related to prehomogeneous vector spaces 
have been studied, and among them are the $b$-functions
arising from some quiver representations.
There are extensive studies on geometric structures of
$(GL(\underline{n}), \Rep(Q,\underline{n}))$, and
of further interest is to clarify the relation between
those studies and our results on $b$-functions. 
%For example,
%Riedtmann and Zwara~\cite{RG} showed that the common zero set
%of the relative invariants of $(GL(\underline{n}), \Rep(Q,
%\underline{n}))$ is a complete intersection for
%generic $\underline{n}$, and 
%Abeasis-Del Fra-Kraft~\cite{ADK} discussed the resolutions of
%singularities of the orbit closures in $(GL(\underline{n}), \Rep(Q,
%\underline{n}))$. The fact that $b$-functions can be described by
%using diagrams seems relevant to the result of \cite{RG}, and
%it is worth studying the relation between our results and
%the ``index of minimal degeneration'' defined in \cite{ADK}. 

The plan of the present paper is as follows. In
Section~\ref{section:preliminaries}, we recall some basic definitions
and results, and in Section~\ref{section:RelativeInvariants},
we recall the results on
the explicit construction of relative invariants.
In Section~\ref{section:BfunOneVar}, we calculate the $b$-functions
of one variable (Theorem{~\ref{thm:Bfun}}), and in
Section~\ref{section:RepThOfQuiver}, our lace diagrams are defined.
Our algorithm to calculate $b$-functions of several variables with
diagrams is explained in Section~\ref{section:BfunctionsOfSeveralVariables},
and the following sections 
(Sections~\ref{section:Afunctions} and \ref{section:Proof}) 
are devoted to its proof. 

\bigskip

\noindent
{\bf Acknowledgement}. I dedicate this paper for the 60th birthday
of my adviser, Professor Tatsuo Kimura. I have had the good
fortune and privilege to be one of his students.
Also thanks are due to the referee for several valuable
suggestions. In particular, he/she has pointed out that
Lemma~\ref{lemma:Restriction} in the original manuscript can be
generalized to arbitrary quivers.
This research was partially supported by the Grant-in-Aid for
Scientific Research (No.\ 19740003) from Japan Society for the Promotion
of Science.

\bigskip

\noindent
{\bf Notation}. 
As usual, $\Z$, $\Q$, $\R$, and $\C$
stand for the ring of rational integers, the field of rational numbers,
the field of real numbers, and the field of complex numbers,
respectively.
For positive integers $m, n$, we denote by $M(m,n)$ the totality of
$m\times n$ complex matrices, and by $0_{m, n}$ the $m\times n$ zero
matrix.
However, we write simply $M(m)$ and $0_{m}$ instead of $M(m,m)$
and $0_{m, m}$, respectively. 
Further, we denote by $E_{m}$ the identity matrix of size $m$. 
%For a matrix $A$, we denote by ${}^{t}\! A$
%the transposed matrix.

\section{Preliminaries}
\label{section:preliminaries}

A {\it quiver} is an oriented graph $Q = (Q_{0}, Q_{1})$ consisting
of a set $Q_{0}$ of vertices and a set $Q_{1}$ of arrows.
Each arrow $a\in Q_{1}$ has a tail $t(a)\in Q_{0}$ and
a head $h(a)\in Q_{0}$. Throughout the present paper,
except in Section~\ref{section:Proof}, 
we consider
quivers of type $A$. Let $Q$ be a quiver of type $A_{r}$, i.e.,
a chain of $r$-vertices and arrows between them:
\[
 Q: \quad \vt{1} \lla \vt{2}\lla \vt{3} \lra \cdots \lla \vt{r}
\]
We identify the vertex and arrow sets with integral intervals, as
\[
 Q_{0} =\{1,\dots, r\}, \qquad Q_{1}=\{1,2,\dots, r-1\}
\]
such that $\{t(a), h(a)\}= \{a, a+1\}$ for
each $a\in Q_{1}$. We also set $\delta(a) =
h(a)-t(a)$, which is equal to $-1$ for a {\it leftward}
arrow and $+1$ for a {\it rightward} arrow.

Fix a dimension vector $\underline{n} =(n_{1},\dots, n_{r})\in
\Z_{>0}^{r}$ and let $L_{i}$ be a vector space of dimension
$n_{i}$. The set of quiver representations with dimension vector
$\underline{n}$ forms the affine space
\[
 \Rep(Q,\underline{n}) =
 \Hom(L_{t(1)}, L_{h(1)})\oplus
 \Hom(L_{t(2)}, L_{h(2)})\oplus \cdots
 \oplus \Hom(L_{t(r-1)}, L_{h(r-1)}),
\]
on which the group $GL(\underline{n}) =
GL(L_{1})\times \cdots \times GL(L_{r})$ has a natural action. 
We choose a basis of $L_{i}$ and identify
\[
 \begin{array}{ll}
 GL(L_{i})\cong
GL(n_{i}) \qquad  &(i=1,\dots, r), \\
 \Hom(L_{t(a)}, L_{h(a)})\cong M(n_{h(a)}, n_{t(a)})\qquad 
&(a=1,\dots, r-1).  
\end{array}
\]
Then, for $g=(g_{i})_{1\leq i\leq r}\in GL(\underline{n})$ and
$v = (X_{h(a), t(a)})_{1\leq a\leq r-1}\in \Rep(Q,\underline{n})$,
 the action is given explicitly by
\[
 g\cdot v = \left( g_{h(a)} X_{h(a), t(a)} \, g_{t(a)}^{-1}
 \right)_{1\leq a\leq r-1}.
\]

\begin{example}
 \label{exmp:FirstExample}
 Let us consider an {\it equioriented} quiver of type $A_{5}$.
 \[
  Q:\quad  \vt{1}\lra \vt{2}\lra \vt{3} \lra \vt{4}\lra \vt{5} 
 \]
 For $\underline{n} = (n_{1}, \dots, n_{5})\in
 \Z_{>0}^{5}$,
 $GL(\underline{n})$ and $\Rep(Q,\underline{n})$ are given by
 \begin{align*}
  GL(\underline{n}) &= GL(n_{1})\times GL(n_{2})\times GL(n_{3})\times
 GL(n_{4})\times GL(n_{5}) , \\
  \Rep(Q,\underline{n}) &= M(n_{2}, n_{1})\oplus M(n_{3}, n_{2})\oplus
  M(n_{4}, n_{3})\oplus M(n_{5}, n_{4}),
 \end{align*}
 and for 
 $g = (g_{1},\dots, g_{5})\in GL(\underline{n})$ and
 $v = (X_{2, 1}, \, X_{3,2}, \, X_{4, 3}, \, X_{5, 4})\in
 \Rep(Q,\underline{n})$, we have
 \[
  g\cdot v =
 (g_{2} X_{2, 1} g_{1}^{-1}, \, g_{3} X_{3, 2} g_{2}^{-1}, \,
 g_{4} X_{4, 3} g_{3}^{-1}, \, g_{5} X_{5, 4} g_{4}^{-1}).
 \]
\end{example}

\begin{remark}
 When $Q$ is equioriented, 
 $(GL(\underline{n}), \Rep(Q, \underline{n}))$ can be regarded
 as a prehomogeneous vector space of parabolic type arising from
 a special linear Lie algebra $\mathfrak{sl}(N)$ with
 $N=n_{1}+\cdots+ n_{r}$ (Rubenthaler~\cite{Rubenthaler},
 Mortajine~\cite{Mortajine}). 
\end{remark}

\begin{example}
 \label{exmp:SecondExample}
 Let us consider an alternating quiver of type $A_{5}$.
 \[
   Q:\quad  \vt{1}\lra \vt{2}\lla \vt{3} \lra \vt{4}\lla \vt{5} 
 \]
 For $\underline{n} = (n_{1}, \dots, n_{5})\in \Z_{>0}^{5}$, 
 $GL(\underline{n})$ and $\Rep(Q,\underline{n})$ are given by
 \begin{align*}
  GL(\underline{n}) &= GL(n_{1})\times GL(n_{2})\times GL(n_{3})\times
 GL(n_{4})\times GL(n_{5}) , \\
  \Rep(Q,\underline{n}) &= M(n_{2}, n_{1})\oplus M(n_{2}, n_{3})\oplus
  M(n_{4}, n_{3})\oplus M(n_{4}, n_{5}),
 \end{align*}
 and for 
 $g = (g_{1}, \dots, g_{5})\in GL(\underline{n})$
 and
 $v = (X_{2, 1}, \, X_{2,3}, \, X_{4, 3}, \, X_{4, 5})\in
 \Rep(Q,\underline{n})$, we have
 \[
  g\cdot v =
 (g_{2} X_{2, 1} g_{1}^{-1}, \, g_{2} X_{2, 3} g_{3}^{-1}, \,
 g_{4} X_{4, 3} g_{3}^{-1}, \, g_{4} X_{4, 5} g_{5}^{-1}).
 \]
\end{example}

Two elements $v, v'\in \Rep(Q,\underline{n})$ belong to the
same $GL(\underline{n})$-orbit if and only if $v$ and $v'$ are
isomorphic as representations of the quiver $Q$, and the
isomorphism classes of representations of $Q$ can be
classified by the  {\it indecomposable decompositions}.
(For basic terminology on quiver representations, we refer to
\cite{Brion, DerksenWeyman}.) 
It is well known that for the Dynkin quivers, the isomorphism classes
of the indecomposable
representations correspond to the positive roots of the corresponding
root systems. 
For the quiver of type $A_{r}$, there is an indecomposable
representation $I_{ij}$ for each pair of integers $(i, j)$
with $1\leq i\leq j\leq r$. 
The dimension vector of $I_{ij}$ assigns the dimension $1$ to
all vertices $k\in Q_{0}$ if $i\leq k\leq j$, and
the dimension $0$ otherwise. For each arrow $a\in Q_{1}$
with $i\leq a< j$, the map $I_{ij}^{(a)}:\C\to \C$ is the
identity. Then, as a representation of the quiver $Q$, 
any $v\in \Rep(Q,\underline{n})$ can be decomposed
into the direct sum of indecomposable representations as
\begin{equation}
 \label{form:IndecomposableDecomposition}
 v\cong \bigoplus_{1\leq i\leq j\leq r}
m_{ij} I_{ij},
\end{equation}
and the multiplicities $m_{ij}$ are uniquely determined. 
Hence $\Rep(Q,\underline{n})$ decomposes into a finite number of
$GL(\underline{n})$-orbits. In particular,
$(GL(\underline{n}), \Rep(Q,\underline{n}))$ is a prehomogeneous
vector space.

Now we give a brief review on basic properties of
prehomogeneous vector spaces. We refer to \cite{Nonreg},
\cite[Chapter 2]{PVBook}, \cite{KashiwaraBook} for the details. 

Let $G$ be a connected algebraic group and  $\rho:G \to GL(V)$  a
rational representation of $G$ on a finite dimensional vector space
$V$.
The triplet $(G,\rho,V)$ is called a \textit{prehomogeneous vector
space} if $V$ has an open dense $G$-orbit, say $O_{0} = \rho(G)v_{0}$. Let
$f$ be a non-zero rational function on $V$ and $\chi\in
\Hom(G,\mathbb{C}^{\times})$. Then we call $f$ a
\textit{relative invariant} with character $\chi$ if
$f(\rho(g)v) = \chi(g)f(v)$ for all $g\in G$ and $v\in O_0$.
If $f_{1}$ and $f_{2}$ are relative invariants which
correspond to the same character, then $f_{2}$ is a
constant multiple of $f_{1}$.

We denote by $S_{1},\dots,S_{l}$ the irreducible components of $V
\backslash O_{0}$ with codimension one, and let $f_{i}$ $(1\leq i\leq
l)$ be an irreducible polynomial defining $S_i$.
Then $f_{1},\dots,f_{l}$ are algebraically independent relative
invariants. Furthermore, every relative invariant $f$ is of the form $f
= c f_{1}^{m_{1}} 
 \cdots f_{l}^{m_{l}} \; (c\in \mathbb{C}^{\times}, m_{i}\in \mathbb{Z})$. 
We call  $f_{1},\dots,f_{l}$ the \textit{fundamental relative invariants}.

A prehomogeneous vector space $(G,\rho, V)$ is called {\it reductive}
if $G$ is a reductive algebraic group.
Now we assume that $(G,\rho, V)$ is 
a reductive prehomogeneous vector space which has 
a relatively invariant polynomial $f$ 
with character $\chi$.  
Let $d= \deg f$, $n= \dim V$. 
We denote by $V^{*}$ be the dual space of $V$, and by
$\rho^{*}:G\to GL(V^{*})$ the contragredient
representation of $\rho$. 
Then the dual triplet $(G,\rho^{*},V^{*})$ is a prehomogeneous vector
space and has a relatively invariant
polynomial $f^{*}$ of degree $d$ with character $\chi^{-1}$. 

Fix a basis $\{e_{1},\dots, e_{n}\}$ of $V$ and let
$v = (v_{1}, \dots, v_{n})$ be the coordinate system of $V$
with respect to this basis. We identify $V$ with $\C^{n}$.
Let $\{e_{1}^{*},\dots, e_{n}^{*}\}$ be the dual basis of
$\{e_{1}, \dots, e_{n}\}$ and $v^{*} =
(v_{1}^{*}, \dots, v_{n}^{*})$ be the coordinate system
of $V^{*}$ with respect to the dual basis. 
We also identify $V^{*}$ with $\C^{n}$.

\begin{lemma}
 \label{prop:b-funOf1var}
 %{\rm \cite[Lemmas 1.6, 1.7]{Nonreg}}
 There exists a polynomial $b_{f}(s) = b_{0}s^{d} + b_{1}s^{d-1} +
 \cdots + b_{d}\in \mathbb{C}[s]$ with $b_{0} \neq 0$ such that
 \[
  f^{*}(\grad_{v})f(v)^{s+1} = b_{f}(s)f(v)^{s}, 
 \]
 where $\grad_{v}$ is given by
 \[
  \grad_{v}= \left(\frac{\partial}{\partial v_{1}}, \cdots,
 \frac{\partial}{\partial v_{n}}\right).
 \]
\end{lemma}
We call $b_{f}(s)$ the {\it $b$-function} of $f$.
By Kashiwara~\cite[Theorem~6.9]{KashiwaraBook}, 
every root of the $b$-function $b_{f}(s)$ is a
negative rational number .

Finally, we recall the definition of $b$-functions of several variables
(cf. M.~Sato~\cite[Proposition~14]{EnAyumi}).
Let $(G, \rho, V)$ be a reductive prehomogeneous vector space
and $f_{1},\dots, f_{l}$ the fundamental relative invariants.
Let $f_{1}^{*},\dots, f_{l}^{*}$ the fundamental relative invariants
of the dual prehomogeneous vector space $(G,\rho^{*}, V^{*})$
such that the characters of $f_{i}$ and $f_{i}^{*}$ are the
inverse of each other. We put
$\underline{f} = (f_{1},\dots, f_{l})$ and
$\underline{f}^{*} = (f_{1}^{*},\dots, f_{l}^{*})$. 
%$V_{f_{i}} = \{v\in V\, ;\, f_{i}(v)\neq 0\}$ and
%$V_{\underline{f}} = \cap_{i=1}^{l} V_{f_{i}}$. 
For a multi-variable $\underline{s} = (s_{1},\dots, s_{l})$,
we define their powers  by
$\underline{f}^{\underline{s}} = \prod_{i=1}^{l} f_{i}^{s_{i}}$
and
$\underline{f}^{*\underline{s}} =\prod_{i=1}^{l} f_{i}^{* s_{i}}$.
\begin{lemma}
 \label{prop:DefOfBfunSeverVari}
 For any $l$-tuple $\underline{m} = (m_{1},\dots, m_{l})\in
 \Z_{\geq 0}^{l}$, there exists a non-zero polynomial
 $b_{\underline{m}}(\underline{s})$ of $s_{1},\dots, s_{l}$
 satisfying
 \[
  \underline{f}^{*\underline{m}}(\grad_{v})
 \underline{f}^{\underline{s} +\underline{m}}(v) =
 b_{\underline{m}}(\underline{s})\, \underline{f}^{\underline{s}}(v).
 \]
 Here $b_{\underline{m}}(\underline{s})$
 is is independent of $v$.
\end{lemma}
We call $b_{\underline{m}}(\underline{s})$ the {\it $b$-function}
of $\underline{f} = (f_{1},\dots, f_{l})$.
We easily see that $b_{\underline{f}^{\underline{m}}}(s) =
b_{\underline{m}}(\underline{m} s)$, and thus
$b_{f_{i}}(s) \; (i=1,\dots, l)$ is a specialization of
$b_{\underline{m}}(\underline{s})$.

\section{Relative invariants}
\label{section:RelativeInvariants}

In this section, we describe a condition for
$(GL(\underline{n}), \Rep(Q,\underline{n}))$
with $Q$ being of type $A$ to have
relative invariants and give their explicit construction.
For the details, see Abeasis~\cite{Abeasis},
Koike~\cite{Koike}. %(Mortajine~\cite{Mortajine})
Note that 
Schofield~\cite{Schofield} constructed relative invariants
for general quivers $Q$.
 (see Section~\ref{section:Proof}.)

Let $Q$ be a quiver of type $A_{r}$ with arbitrary orientation.
The orientation of $Q$ is determined by the sequence
\[
 \{1=\nu(0)< \nu(1)< \nu(2)<\cdots <\nu(h)
< \nu(h+1)= r\}
\]
which consists of the sinks and sources of $Q$. 
Note that
if $Q^{*}$ is the quiver obtained from $Q$ by reversing all the arrows,
then $(GL(\underline{n}), \Rep(Q^{*}, \underline{n}))$ is the
dual prehomogeneous vector space
$(GL(\underline{n}), \Rep(Q, \underline{n})^{*})$. 

Now we fix a dimension vector $\underline{n}=(n_{1},\dots, n_{r})$
and consider the fundamental relative invariants of
$(GL(\underline{n}), \Rep(Q, \underline{n}))$.
First, for a given pair $(p, q)$ with $1\leq p< q\leq r$, we define
indices $\alpha=\alpha(p,q)$ and $\beta = \beta (p,q)$ by
the conditions
\[
 \nu(\alpha-1)\leq p< \nu(\alpha), \qquad
 \nu(\beta)<q\leq \nu(\beta+1).
\]
When $p, q$ are clear from the context, we just write
$\alpha, \beta$ instead of $\alpha(p,q), \beta(p,q)$.
Then $I_{\underline{n}}(Q)$ is defined to be  the totality 
of pairs $(p, q)$ with $1\leq p<q\leq r$ which satisfy the following
conditions
(I1)$\sim$(I4):
\begin{enumerate}
 \item[(I1)] For $t$ with $p<t\leq \nu(\alpha)$, we have
	     $n_{t}> n_{p}$,
 \item[(I2)] For $\kappa=0,1,\dots, \beta-\alpha-1$ and $t$ with
	     $\nu(\alpha+\kappa)< t\leq \nu(\alpha+\kappa+1)$,
	     we have
	     \[
	      n_{t}> n_{\nu(\alpha+\kappa)}-
	     n_{\nu(\alpha+\kappa-1)} +\cdots+
	     (-1)^{\kappa}n_{\nu(\alpha)} + (-1)^{\kappa+1}
	     n_{p},
	     \]
 \item[(I3)] For $t$ with $\nu(\beta)<t<q$, we have
	     \[
	      n_{t}> n_{\nu(\beta)} -
	     n_{\nu(\beta-1)} + \cdots + (-1)^{\beta-\alpha}
	     n_{\nu(\alpha)} + (-1)^{\beta-\alpha+1}n_{p},
	     \]
 \item[(I4)] $n_{q} = n_{\nu(\beta)} -
	     n_{\nu(\beta-1)} + \cdots + (-1)^{\beta-\alpha}
	     n_{\nu(\alpha)} + (-1)^{\beta-\alpha+1}n_{p}.$
\end{enumerate}

By Abeasis~\cite{Abeasis}, we have the following lemma. 

\begin{lemma}%[Abeasis~\cite{Abeasis}]
 \label{thm:Abeasis}
 There exists a one-to-one correspondence between
 $I_{\underline{n}}(Q)$ and the set of
 $GL(\underline{n})$-orbits in
 $\Rep(Q, \underline{n})$ of codimension one.
 In particular, there exists a one-to-one correspondence
 between  $I_{\underline{n}}(Q)$ and 
 the fundamental  relative invariants of $(GL(\underline{n}), \Rep
 (Q,\underline{n}))$.
\end{lemma}

The explicit construction of an irreducible relative invariant
corresponding to $(p,q)\in I_{\underline{n}}(Q)$ is given as follows.
When there exist no sink and source between two vertices
$\mu, \nu\; (\mu< \nu)$ of $Q$, either the following (a) or (b) holds:
\begin{eqnarray*}
 &{\rm (a)}& \qquad \vt{\mu}\lra \vt{\mu+1} \lra \cdots \lra
  \vt{\nu-1}\lra \vt{\nu} \\
 &{\rm (b)}& \qquad \vt{\mu}\lla \vt{\mu+1} \lla \cdots \lla
  \vt{\nu-1}\lla \vt{\nu} 
\end{eqnarray*}
In the case of (a), we put
\[
 X_{\nu,\mu} = X_{\nu, \nu-1} X_{\nu-1,\nu-2}\cdots X_{\mu+1, \mu},
\]
and in the case of (b), we put
\[
 X_{\mu,\nu}=X_{\mu, \mu+1}X_{\mu+1, \mu+2}\cdots X_{\nu-1, \nu}.
\]
Now suppose that $p$ is a source and $q$ is a sink.
\[
 \vt{p}\lra \cdots\lra \vt{\nu(\alpha)} \lla \vt{\nu(\alpha)+1}\lla
 \cdots \lla \vt{\nu(\alpha+1)}\lra\cdots \lla
 \vt{\nu(\beta)} \lra \cdots\lra \vt{q}
\]
In this case, for $v\in \Rep(Q,\underline{n})$, we define a matrix
$Y_{(p,q)}(v)$ by
\[
 Y_{(p,q)}(v) =
 \begin{pmatrix}
  X_{\nu(\alpha), p}& X_{\nu(\alpha), \nu(\alpha+1)} &
  O & \cdots & O & O \\
  O & X_{\nu(\alpha+2), \nu(\alpha+1)} &
  X_{\nu(\alpha+2), \nu(\alpha+3)} & \ddots & \vdots & \vdots \\
  \vdots & \vdots & \vdots & \ddots & \vdots & \vdots \\
  O & O & O & \cdots & X_{\nu(\beta-1), \nu(\beta-2)} &
  X_{\nu(\beta-1), \nu(\beta)} \\
  O & O & O & \cdots & O & X_{q, \nu(\beta)}
 \end{pmatrix},
\]
and put $f_{(p,q)}(v) =\det Y_{(p,q)}(v)$. Then it is easy to see that
$f_{(p,q)}(v)$ is a relative invariant of
$(GL(\underline{n}), \Rep(Q,\underline{n}))$.

Next we consider the case where
both $p$ and $q$ are sources.
\[
 \vt{p}\lra \cdots\lra \vt{\nu(\alpha)} \lla \vt{\nu(\alpha)+1}\lla
 \cdots \lla \vt{\nu(\alpha+1)}\lra\cdots \lra
 \vt{\nu(\beta)} \lla \cdots\lla \vt{q}
\]
Then we define a matrix $Y_{(p,q)}(v)$ by
\[
 Y_{(p,q)}(v) =
 \begin{pmatrix}
  X_{\nu(\alpha), p}& X_{\nu(\alpha), \nu(\alpha+1)} &
  O & \cdots & O & O \\
  O & X_{\nu(\alpha+2), \nu(\alpha+1)} &
  X_{\nu(\alpha+2), \nu(\alpha+3)} & \ddots & \vdots & \vdots \\
  \vdots & \vdots & \vdots & \ddots & \vdots & \vdots \\
  O & O & O & \cdots & X_{\nu(\beta-2), \nu(\beta-1)} &
  O \\
  O & O & O & \cdots & X_{\nu(\beta), \nu(\beta-1)}
  & X_{\nu(\beta),q}
 \end{pmatrix}
\]
and put $f_{(p,q)}(v) =\det Y_{(p,q)}(v)$. Then it is easy to see that
$f_{(p,q)}(v)$ is a relative invariant of
$(GL(\underline{n}), \Rep(Q,\underline{n}))$.
One can easily find similar expressions of $Y_{(p,q)}(v)$ for 
the other cases, i.e., where ``$p$ is a sink and $q$ is a source'' or
``both $p$ and $q$ are sinks''.

\begin{example}
 \label{exmp:FirstExampleII}
 In Example~\ref{exmp:FirstExample}, assume that 
 $n_{1}<n_{2}< n_{3}=n_{4}, \,
 n_{5}=n_{1}$. Then we have
 $I_{\underline{n}}(Q)=\{(1,5), (3,4)\}$.
 The fundamental relative invariants are given explicitly by
 \[
  f_{(3,4)}(v) = \det X_{4, 3}, \qquad
 f_{(1, 5)}(v) = \det X_{5,1} = \det (X_{5, 4} X_{4, 3}
 X_{3, 2} X_{2, 1}).
 \]
\end{example}

\begin{example}
 \label{exmp:SecondExampleII}
 In Example~\ref{exmp:SecondExample}, assume that
 $n_{1}+n_{3}=n_{2}+n_{4}, \, n_{1}< n_{2}< n_{3}, \, n_{5}=n_{1}$.
 Then we have $I_{\underline{n}}(Q) = \{
 (1,4), (2,5)\}$. The fundamental relative invariants are given
 explicitly by
 \[
  f_{(1,4)}(v) = \det
 \begin{pmatrix}
  X_{2, 1}& X_{2, 3} \\
  O & X_{4, 3}
 \end{pmatrix}, \qquad
 f_{(2, 5)}(v) =\det
 \begin{pmatrix}
  X_{2, 3}& O \\
  X_{4, 3} & X_{4, 5}
 \end{pmatrix}.
 \]
\end{example}

\begin{example}
 \label{exmp:FourthExample}
 Let $Q$ be the following quiver of type $A_{8}$:
 \[
  Q:\quad \vt{1}\lra \vt{2}\lra \vt{3}\lla \vt{4}\lra
 \vt{5}\lra \vt{6}\lla \vt{7}\lla \vt{8}
 \]
 In this case, $\nu(0) = 1, \, \nu(1) = 3, \,
 \nu(2) = 4, \, \nu(3) = 6, \, \nu(4)= 8$. 
 If the dimension vector $\underline{n}$ satisfies 
 \begin{eqnarray*}
 & & n_{t}> n_{1} \; (t=2,3), \qquad
 n_{t}> n_{3}-n_{1}\; (t=4), \qquad
 n_{t}> n_{4}-n_{3}+n_{1}\; (t=5,6),
 \\
 & &  
 n_{t}> n_{6}-n_{4}+n_{3}-n_{1}\; (t=7), \qquad
 n_{8} = n_{6}-n_{4}+n_{3}-n_{1}, 
 \end{eqnarray*}
 then $(1,8)\in I_{\underline{n}}(Q)$ and the corresponding relative
 invariant is given by
 \[
  f_{(1,8)}(v) = \det
 \begin{pmatrix}
  X_{3,1}& X_{3,4} & O \\
  O & X_{6,4} & X_{6,8}
 \end{pmatrix}=
 \det
 \begin{pmatrix}
  X_{3,2}X_{2,1}& X_{3,4} & O \\
  O & X_{6,5}X_{5,4} & X_{6,7}X_{7,8}
 \end{pmatrix}.
 \]
\end{example}

\section{$b$-Functions of one variable}
\label{section:BfunOneVar}

We identify the dual space $\Rep(Q,\underline{n})^{*}$ of
$\Rep(Q,\underline{n})$ with $\Rep(Q, \underline{n})$ itself
via a non-degenerate bilinear form
\[
 \left\langle v, w\right\rangle =
 \sum_{a=1}^{r-1} \tr \left( {}^{t} Y_{h(a), t(a)}
 X_{h(a), t(a)}
 \right)
\]
for $v = \left( X_{h(a), t(a)}
\right)_{a}$ and $w = \left( Y_{h(a), t(a)}
\right)_{a}\in\Rep(Q,\underline{n})$.
Then, by Lemma~\ref{prop:b-funOf1var}, 
there exists a polynomial
$b_{(p,q)}(s)\in \C[s]$ satisfying 
\[
 f_{(p,q)}\left(\grad_{v}\right)
 f_{(p,q)}(v)^{s+1} = b_{(p,q)}(s) \cdot f_{(p,q)}(v)^{s}.
\]
By using the decomposition formula for $b$-functions 
proved in F.~Sato and Sugiyama~\cite{SatoSugi}, we determine
$b_{(p,q)}(s)$.

\begin{notation}
 \label{notation:RIasDetSymbol}
 As we have seen, in the case of quivers of type $A$,
 our relative invariants $f_{(p,q)}(s)$ are uniquely determined
 if we specify
 the subquiver together with dimension vector
 $\underline{n}$.
 In the following, we use an informal notation
 \[
 f_{(p,q)}(v)= \det \left(\ovt{n_{p}}\longrightarrow \ovt{n_{p+1}}
 \longrightarrow \cdots \longleftarrow \ovt{n_{q}}
 \right)
 \]
 to denote the relative invariant $f_{(p,q)}(v)$.
 Moreover, we use the following informal notation to
 denote the $b$-function $b_{(p,q)}(s)$ of $f_{(p,q)}(s)$ :
 \[
  b_{(p,q)}(s) = b
 \left(\ovt{n_{p}}\longrightarrow \ovt{n_{p+1}}
 \longrightarrow \cdots \longleftarrow \ovt{n_{q}}
 \right).
 \]
\end{notation}
 
\begin{example}
 \label{exmp:EquiorientedCase}
 Let us consider an equioriented quiver of type $A_{r}$
 as below:
 \[
  Q: \vt{1}\longrightarrow \vt{2}\longrightarrow \vt{3}
 \longrightarrow \cdots \longrightarrow \vt{r}
 \]
 Then, for $\underline{n}=(n_{1},\dots, n_{r})\in
 \Z_{>0}^{r}$, we have
 \begin{align*}
  GL(\underline{n}) &= GL(n_{1})\times GL(n_{2})\times GL(n_{3})
  \times \cdots \times GL(n_{r}), \\
 \Rep(Q, \underline{n}) &= M(n_{2},n_{1})\oplus M(n_{3},n_{2})
  \oplus \cdots \oplus M(n_{r}, n_{r-1}),
 \end{align*}
 and the action is given by
 \[
  g\cdot v = \left(g_{2}X_{2,1}g_{1}^{-1}, \,
 g_{3}X_{3,2} g_{2}^{-1},\, \dots,
 g_{r}X_{r,r-1}g_{r-1}^{-1}\right)
 \]
 for $g= (g_{1}, g_{2}, g_{3},\dots, g_{r})\in GL(\underline{n})$
 and
 $v= (X_{2, 1}, X_{3, 2}, \dots, X_{r,r-1})\in \Rep(Q, \underline{n})$.
 Now we assume that 
 $n_{1}=n_{r}< n_{2}, n_{3},\dots, n_{r-1}$. Then
 \[
 %  \label{form:RIofEquiorientedAr}
 f(v) = \det (X_{r,1})= \det \left(X_{r,r-1}\cdots X_{3,2}X_{2,1}\right)  
 \]
 is an irreducible relative invariant corresponding to the character 
 $\chi(g) = \det g_{r} \det g_{1}^{-1}$.
 Note that this polynomial can be expressed as
 \[
 f(v)= \det\left(\ovt{n_{1}}\longrightarrow \ovt{n_{2}}\longrightarrow
 \ovt{n_{3}}\longrightarrow
 \cdots \longrightarrow \ovt{n_{r}}
  \right)
 \]
 if we employ Notation~\ref{notation:RIasDetSymbol}.  
 We shall calculate the $b$-function $b_{f}(s)$ of $f$ by using the
 result of \cite{SatoSugi}. We put
  \begin{eqnarray*}
  & & G'=  GL(n_{3})\times \cdots \times GL(n_{r}), \\
  & & E = M(n_{3}, n_{2})\oplus \cdots M(n_{r},n_{r-1}), \\
  & & F = M(n_{2}, n_{1}),  \\
  & & GL(m) = GL(n_{2}), \qquad GL(n)= GL(n_{1})
 \end{eqnarray*}
 and regard $(GL(\underline{n}), \Rep(Q,\underline{n}))$ as
 a prehomogeneous vector space of the form \cite[(2.2)]{SatoSugi}. 
 Then we have $l=0, d=1$ in the notation of
 \cite[Section 2]{SatoSugi} and thus we can apply
 \cite[Theorem~2.6]{SatoSugi} in order to obtain
 the decomposition
 \[
  b_{f}(s) = b_{1}(s) b_{2}(s).
 \]
 Moreover, by \cite[Theorem~3.3]{SatoSugi}, we have
 \[
  b_{2}(s) = \prod_{\lambda=1}^{n_{1}}(s+ n_{2}-n_{1}+\lambda).
 \]
 Note that $m=n_{2}, n=n_{1}, d=1$ in the notation of
 \cite[Theorem~3.3]{SatoSugi}.
 The last step is to calculate $b_{1}(s)$.
 Let $X_{2,1}^{0}= {}^{t}(E_{n_{1}}|0_{n_{1}, n_{2}-n_{1}})\in F$ and
 put this into $f(v)$. Then we have
 \[
  f(X_{2,1}^{0}, \, X_{3,2},\, \dots,
 X_{r,r-1}) = \det \left(X_{r,r-1}\cdots X_{3, 2}'\right),
 \]
 where $X_{3, 2}'$ is a part of the following block decomposition:
 \[
  X_{3,2} = \left(X_{3,2}'\, \big| \, X_{3, 2}''
 \right)\in M(n_{3}, n_{2}), \qquad
 X_{3, 2}'\in M(n_{3}, n_{1}), \qquad
 X_{3,2}''\in M(n_{3}, n_{2}-n_{1}).
 \]
 Note that
 $f(X_{2,1}^{0}, \, X_{3,2},\, \dots,
 X_{r,r-1})$ 
  can be regarded as  the relative invariant
 \[
  \det \left(\ovt{n_{1}}\longrightarrow \ovt{n_{3}}
 \longrightarrow \ovt{n_{4}}\longrightarrow \cdots
 \longrightarrow \ovt{n_{r}}\right)
 \]
 of a prehomogeneous vector space arising from an
 equioriented quiver of type $A_{r-1}$.
 By using Notation~\ref{notation:RIasDetSymbol}, we can
 summarize the above argument into 
 a reduction formula
 \begin{align}
  \nonumber
  b_{f}(s) &=b\left(\ovt{n_{1}}\longrightarrow \ovt{n_{2}}
 \longrightarrow \ovt{n_{3}}\longrightarrow \cdots
 \longrightarrow \ovt{n_{r}}\right) \\
  \label{form:CutOffInEquiorientedCase}
 &=
 b\left(\ovt{n_{1}}\longrightarrow \ovt{n_{3}}
 \longrightarrow \ovt{n_{4}}\longrightarrow \cdots
 \longrightarrow \ovt{n_{r}}\right)
 \times
 \prod_{\lambda=1}^{n_{1}}(s+ n_{2}-n_{1}+\lambda).
 \end{align} 
 By repeating this cut-off operation, we get to the
 $b$-function of $\det \left(\ovt{n_{1}}\longrightarrow
 \ovt{n_{r}}\right)$, which is nothing but the
 formula~\eqref{form:Cayley'sFormula}.
 We therefore obtain 
 \[
  b_{f}(s) = \prod_{t=2}^{r}\prod_{\lambda=1}^{n_{1}}
  (s+n_{t}-n_{1}+ \lambda).
 \]
\end{example}

The reduction formula~\eqref{form:CutOffInEquiorientedCase}
can be generalized as the following lemma.

\begin{lemma}
 \begin{align}
  \nonumber
   & b\left(\ovt{n_{p}}\longrightarrow \ovt{n_{p+1}}
     \longrightarrow \ovt{n_{p+2}}\cdots\cdots
     \ovt{n_{q}}
    \right)  \\
  \label{form:ReductionI}
  &=
  \prod_{\lambda=1}^{n_{p}} (s+n_{p+1} -n_{p}+\lambda)\times 
    b\left(\ovt{n_{p}}\longrightarrow \ovt{n_{p+2}}\cdots\cdots
     \ovt{n_{q}}
    \right) \\ \nonumber
  &=
  \prod_{\lambda=1}^{n_{p}} (s+n_{p+1} -n_{p}+\lambda)\times
  \prod_{\lambda=1}^{n_{p}} (s+n_{p+2} -n_{p}+\lambda)\times
  b\bigl(\cdots\cdots
  \bigr),
 \end{align}
 \begin{align}
  \nonumber
  &  b
    \left(\ovt{n_{p}}\longrightarrow \ovt{n_{p+1}}
     \longleftarrow \ovt{n_{p+2}}\cdots\cdots
     \ovt{n_{q}}
    \right) \\
  \label{form:ReductionII}
  &= 
  \prod_{\lambda=1}^{n_{p}} (s+n_{p+1} -n_{p}+\lambda)\times
   b\left(\ovt{n_{p+1} -n_{p}}\longleftarrow \ovt{n_{p+2}}\cdots\cdots
     \ovt{n_{q}}
    \right)  \\   \nonumber
  &=
    \prod_{\lambda=1}^{n_{p}} (s+n_{p+1} -n_{p}+\lambda)\times
    \prod_{\lambda=1}^{n_{p+1}-n_{p}} (s+n_{p+2} -n_{p+1} + n_{p}
  +\lambda)\times
    b\bigl(\cdots\cdots
  \bigr).
 \end{align} 
\end{lemma}

\begin{proof}
 We first prove \eqref{form:ReductionII} in the case of
 $\delta(p+2)=1$, in which case our quiver is of the form
 \[
  \ovt{n_{p}}\longrightarrow \ovt{n_{p+1}}
     \longleftarrow \ovt{n_{p+2}}\longrightarrow \ovt{n_{p+3}}
 \cdots\cdots
 \]
 and we have $\nu(\alpha) = n_{p+1}$ and
 $\nu(\alpha+1) = n_{p+2}$. The relative invariant $f_{(p,q)}(v)$
 is given by $f_{(p,q)}(v) = \det Y_{(p,q)}(v)$ with
 \[
  Y_{(p,q)}(v) =
 \begin{pmatrix}
  X_{p+1, p}& X_{p+1, p+2} & O \\
  O & X_{\nu(\alpha+2), p+2} & \cdots \\
  \vdots & \vdots & \ddots
 \end{pmatrix},
 \]
 and the corresponding character is $\chi(g) = \det g_{p}^{-1} \det
 g_{p+1} \det g_{p+2}^{-1} \cdots$.
 We put
 \begin{align*}
  G'&= GL(n_{p+2})\times GL(n_{p+3})\times \cdots \times GL(n_{q}), \\
  E &= M(n_{p+1}, n_{p+2})\oplus M(n_{p+3}, n_{p+2})\oplus
  \cdots, & F &=M(n_{p+1},n_{p}), \\
  GL(m) &= GL(n_{p+1}), & GL(n)&= GL(n_{p}),
 \end{align*}
 and regard $(GL(\underline{n}), \Rep(Q,\underline{n}))$
 as a triplet of the form \cite[(2.2)]{SatoSugi}.
 Then we have $l= d=1$ in the notation of \cite[Section 2]{SatoSugi}
 and thus we can apply \cite[Theorem~2.5]{SatoSugi} so that
 we have the decomposition
 \[
  b_{(p,q)}(s) = b_{1}(s) b_{2}(s).
 \]
 Moreover, by \cite[Theorem~3.3]{SatoSugi}, we have
 \[
  b_{2}(s) = \prod_{\lambda=1}^{n_{p}}(s+ n_{p+1}-n_{p}+\lambda).
 \]
 Note that $m=n_{p+1}, n=n_{p}, d=1$ in the notation of
 \cite[Theorem~3.3]{SatoSugi}. To calculate $b_{1}(s)$, we let
 $X_{p+1,p}^{0}= {}^{t} (E_{n_{p}}\, |\, 0_{n_{p}, n_{p+1}-n_{p}})
 \in F$ and put this into $f_{(p,q)}(v)$. Then we have
 \[
  f_{(p,q)}\left(X_{p+1,p}^{0}, \, X_{p+1,p+2}, X_{p+3,p+2}, \dots
 \right)=
 \det Y_{(p,q)}'(v), 
 \]
 where
 \[
  Y_{(p,q)}'(v) =
 \begin{pmatrix}
  X_{p+1,p+2}''& O & \cdots \\
  X_{\nu(\alpha+2), p+2} & \cdots & \cdots \\
  \vdots & \vdots & \ddots
 \end{pmatrix},
 \]
 and $X_{p+1,p+2}''$ is a part of the following block decomposition:
 \begin{align*}
  X_{p+1,p+2} &=
 \left(
 \begin{array}{c}
  X_{p+1,p+2}'\\ \hline
  X_{p+1,p+2}''
 \end{array}
 \right)\in M(n_{p+1}, n_{p+2}),  \\
  X_{p+1,p+2}' &\in M(n_{p},n_{p+2}), \qquad
  X_{p+1, p+2}'' \in M(n_{p+1}-n_{p}, n_{p+2}).
 \end{align*}
 The polynomial $f_{(p,q)}\left(X_{p+1,p}^{0}, \, X_{p+1,p+2},
 X_{p+3,p+2}, \dots 
 \right)$ can be regarded as the relative invariant
 \[
  \det
 \left(\ovt{n_{p+1}-n_{p}}\longleftarrow \ovt{n_{p+2}}\longrightarrow
 \ovt{n_{p+3}}\cdots \cdots \right)
 \]
 and hence we obtain a reduction formula
 \begin{align*}
  b_{(p,q)}(s)&= b\left(\ovt{n_{p}}\longrightarrow \ovt{n_{p+1}}
     \longleftarrow \ovt{n_{p+2}}\longrightarrow \ovt{n_{p+3}}
 \cdots\cdots
  \right) \\
  &= b\left(\ovt{n_{p+1}-n_{p}}
     \longleftarrow \ovt{n_{p+2}}\longrightarrow \ovt{n_{p+3}}
 \cdots\cdots
  \right) \times
  \prod_{\lambda=1}^{n_{p}}(s+ n_{p+1}-n_{p}+\lambda).
 \end{align*}
 By repeating this cut-off operation with \cite[Theorem~2.5]{SatoSugi},
 we have
 \begin{align*}
 b\left(\ovt{n_{p+1}-n_{p}}
     \longleftarrow \ovt{n_{p+2}}\longrightarrow \ovt{n_{p+3}}
 \cdots\cdots
  \right) 
  &=
 b\left(\ovt{n_{p+2}-n_{p+1}+n_{p}}\longrightarrow \ovt{n_{p+3}}
 \cdots\cdots
 \right)\\
  & \quad\times
  \prod_{\lambda=1}^{n_{p+1}-n_{p}}
  (s+n_{p+2}-n_{p+1}+n_{p}+\lambda).
 \end{align*}
 Combining the above two formulas, we obtain
 \begin{align}
  \nonumber
  b_{(p,q)}(s)&=b\left(\ovt{n_{p}}\longrightarrow \ovt{n_{p+1}}
     \longleftarrow \ovt{n_{p+2}}\longrightarrow \ovt{n_{p+3}}
 \cdots\cdots
  \right) \\
  \label{form:WhenDeltaP+2Is1}
  &=\prod_{\lambda=1}^{n_{p}}(s+ n_{p+1}-n_{p}+\lambda)\times
  \prod_{\lambda=1}^{n_{p+1}-n_{p}}
  (s+n_{p+2}-n_{p+1}+n_{p}+\lambda)\\
  \nonumber
  &\quad \times
  b\left(\ovt{n_{p+2}-n_{p+1}+n_{p}}\longrightarrow \ovt{n_{p+3}}
 \cdots\cdots
 \right).
 \end{align}
 On the other hand, when $\delta(p+2)=-1$, we carry out the
 cut-off operations as follows:
 \begin{align}
  \nonumber
  &b\left(\ovt{n_{p}}\longrightarrow \ovt{n_{p+1}}
     \longleftarrow \ovt{n_{p+2}}\longleftarrow \ovt{n_{p+3}}
 \cdots\cdots
  \right) \\  \nonumber
  &=\prod_{\lambda=1}^{n_{p}}(s+ n_{p+1}-n_{p}+\lambda)\times
  b\left(\ovt{n_{p+1}-n_{p}}
     \longleftarrow \ovt{n_{p+2}}\longleftarrow \ovt{n_{p+3}}
 \cdots\cdots
  \right) \\
  \label{form:WhenDeltaP+2Is-1}  
  &=
  \prod_{\lambda=1}^{n_{p}}(s+ n_{p+1}-n_{p}+\lambda)\times  
  \prod_{\lambda=1}^{n_{p+1}-n_{p}}
  (s+n_{p+2}-n_{p+1}+n_{p}+\lambda)\\ \nonumber
  &\quad \times
  b\left(\ovt{n_{p+1}-n_{p}}\longleftarrow \ovt{n_{p+3}}
 \cdots\cdots
 \right).
 \end{align}
 In the first equality, we use \cite[Theorem~2.5]{SatoSugi} and
 in the second equality, we use \cite[Theorem~2.6]{SatoSugi}.
 The two formulas \eqref{form:WhenDeltaP+2Is1} and
 \eqref{form:WhenDeltaP+2Is-1} can be summarized into
 one formula \eqref{form:ReductionII} in an abbreviated form. 
 We omit the proof of \eqref{form:ReductionI}, since it is quite the
 same as above. 
\end{proof}

Now we give a general formula for the $b$-function
$b_{(p,q)}(s)$. For
$\kappa=0,1,\dots, \beta-\alpha$, we put 
\begin{align*}
 \overline{n}_{\nu(\alpha+\kappa)} &=
 \sum_{\tau=0}^{\kappa} (-1)^{\tau} n_{\nu(\alpha+\kappa-\tau)}
 + (-1)^{\kappa+1} n_{p} \\
 &= n_{\nu(\alpha+\kappa)} -n_{\nu(\alpha+\kappa-1)} +
 \cdots + (-1)^{\kappa}n_{\nu(\alpha)} + (-1)^{\kappa+1}
 n_{p}.
\end{align*}
Then, as
an immediate consequence of
 \eqref{form:ReductionI} and
 \eqref{form:ReductionII}, we obtain the following theorem.
\begin{theorem}
 \label{thm:Bfun}
 \begin{align*}
  b_{(p,q)}(s) &=
  \prod_{t=p+1}^{\nu(\alpha)}\prod_{\lambda=1}^{n_{p}}
  (s+ n_{t} -n_{p}+ \lambda) \\
  & \quad \; \times \prod_{\kappa=0}^{\beta-\alpha-1}
  \prod_{t=\nu(\alpha+\kappa) + 1}^{\nu(\alpha+\kappa+1)}
  \prod_{\lambda=1}^{\overline{n}_{\nu(\alpha+\kappa)}}
  (s+ n_{t} -\overline{n}_{\nu(\alpha+\kappa)} + \lambda) \\
  &\quad \; \times \prod_{t=\nu(\beta)+1}^{q}
  \prod_{\lambda=1}^{\overline{n}_{\nu(\beta)}}
  (s+n_{t} -\overline{n}_{\nu(\beta)} + \lambda).
 \end{align*}
\end{theorem}

\begin{remark}
 Recently, Wachi~\cite{Wachi} has studied the above theorem
 from the view point of Capelli identities. In particular,
 he gave another proof of the above theorem for the
 equioriented case by using his generalized Capelli identity
 (\cite[Theorem~5.1]{Wachi}). 
\end{remark}

\begin{example}
 \label{exmp:FirstExampleIII}
 For the relative invariants in
  Example~\ref{exmp:FirstExampleII}, we have
 \begin{align*}
  b_{(3,4)}(s)&=(s+1)(s+2)\cdots (s+n_{3}), \\
  b_{(1,5)}(s)&=
  \prod_{t=2}^{5}\prod_{\lambda=1}^{n_{1}}
  (s+n_{t}-n_{1}+\lambda) \\
  &= (s+1)\cdots (s+n_{1})\times (s+n_{2}-n_{1}+1)\cdots (s+n_{2}) \\
  &\quad \times (s+n_{3}-n_{1}+1)^2\cdots (s+n_{3})^2.
 \end{align*}
\end{example}

\begin{example}
 \label{exmp:SecondExampleIII}
 For the relative invariants in Example~\ref{exmp:SecondExampleII}, we
 have 
 \begin{align*}  
  b_{(1,4)}(s)
  &= (s+1)\cdots (s+n_{3})\times (s+n_{2}-n_{1}+1)\cdots (s+n_{2}), \\
  b_{(2,5)}(s)&= (s+1)\cdots (s+ n_{4})\times
 (s+n_{3}-n_{2}+1)\cdots (s+n_{3}).
 \end{align*}
% See also Example~\ref{exmp:AlternatingQuiver}. 
% \begin{align*}  
%  b_{(1,4)}(s)&=(s+n_{2}-n_{1}+1)\cdots (s+n_{2})\times
%  (s+n_{3}-n_{2}+n_{1}+1)\cdots (s+n_{3}) \\
%  &\quad \times (s+n_{4}-n_{3}+n_{2}-n_{1}+1)\cdots (s+n_{4})  \\
%  &= (s+1)\cdots (s+n_{3})\times (s+n_{2}-n_{1}+1)\cdots (s+n_{2}).
% \end{align*} 
\end{example}

\begin{example}
\label{exmp:FourthExampleII}
 The $b$-function $b_{(1,8)}(s)$ of the relative invariant
 $f_{(1,8)}(v)$ in Example~\ref{exmp:FourthExample} is given by
 \begin{align*}
  b_{(1,8)}(s)&=
  \prod_{t=2}^{3} (s+n_{t}-n_{1}+1)\cdots (s+n_{t}) \\
  &\quad\times
  \prod_{t=4}^{4}(s+n_{t}-n_{3}+n_{1}+1)\cdots (s+n_{t})\\
  &\quad \times
  \prod_{t=5}^{6}(s+n_{t}-n_{4}+n_{3}-n_{1}+1)\cdots (s+n_{t}) \\[4pt]
  & \quad\times (s+1)\cdots (s+n_{7}).
 \end{align*}
\end{example} 

\section{Rank parameters and lace diagrams}
\label{section:RepThOfQuiver}

To describe a combinatorial method to compute the $b$-functions of
several variables, we need to introduce the notion of
rank parameters and lace diagrams. 

Assume that $Q$ is a quiver of type $A_{r}$ with arbitrary orientation. 
%\[
% Q: \vt{1}\lra \vt{2}\lla \cdots \lla \vt{r-2}\lra
% \vt{r-1}\lra \vt{r}
%\]
For any pair $(i, j)$ of integers with $i< j$, we denote by
$Q^{(i,j)}$ the subquiver of $Q$ starting at $i$ and terminating
at $j$. Here $Q^{(i,j)}$ includes the vertices $i$ and $j$, and thus
$i$ is either a source or a sink of $Q^{(i, j)}$, and so is $j$.  
Let $A=\left(A_{h(a), t(a)}
\right)_{1\leq a\leq r-1}\in \Rep(Q,\underline{n})$ be a given
representation of $Q$
with dimension vector $\underline{n}$ ; recall that $A_{h(a), t(a)}$
is a linear map from $L_{t(a)}$ to $L_{h(a)}$. 
For $1\leq i< j\leq r$,
we define a map
\[
 Y_{(i,j)}(A):\bigoplus_{\tau} L_{\tau}
 \longrightarrow 
 \bigoplus_{\sigma} L_{\sigma} \qquad\;
 \left(\begin{array}{l}
   \text{$\tau$ runs over all the sources of $Q^{(i,j)}$}\\
   \text{$\sigma$ runs over all the sinks of $Q^{(i,j)}$}
 \end{array}\right)
\]
to be the collection of linear maps $\varphi_{\kappa}^{A}$ defined
by
\begin{equation}
 \label{form:DifferenceMap}
  \varphi_{\kappa}^{A}:
 L_{\mu(\kappa-1)}\oplus L_{\mu(\kappa+1)} \to
 L_{\mu(\kappa)}\quad  ; \quad
 (z, w) \mapsto
 (A_{\mu(\kappa), \mu(\kappa-1)} z +
 A_{\mu(\kappa), \mu(\kappa+1)} w),
\end{equation}
where $\mu(\kappa-1), \mu(\kappa+1)$ are sources of $Q^{(i,j)}$
and $\mu(\kappa)$ is a sink of $Q^{(i,j)}$.

\begin{example}
 Let us consider the following quiver of type $A_{5}$:
 \[
  \vt{1}\lra \vt{2}\lla \vt{3} \lra \vt{4}\lra \vt{5} 
 \]
 Then we have
 \[
 \begin{array}{lcl}
   Y_{(1,4)}(A):L_{1}\oplus L_{3} \to
   L_{2}\oplus L_{4}  & \quad ;\quad &
   (z_{1},z_{3}) \longmapsto
   \left( A_{2,1}z_{1} +A_{2, 3}z_{3},\, A_{4, 3}z_{3} 
   \right) \\[5pt]
  Y_{(2,5)}(A): L_{3} \to
   L_{2}\oplus L_{5} & \quad ;\quad & 
   z_{3} \longmapsto
   \left(A_{2, 3}z_{3},\, A_{5, 4}A_{4, 3}z_{3} 
   \right) \\[5pt]
  Y_{(1,5)}(A): L_{1}\oplus L_{3} \to
   L_{2}\oplus L_{5} & \quad ;\quad&
   (z_{1}, z_{3}) \longmapsto
   \left(A_{2,1}z_{1}+A_{2, 3}z_{3},\, A_{5, 4}A_{4, 3}z_{3} 
   \right).   
 \end{array}
 \]
\end{example}
For $A\in \Rep(Q,\underline{n})$, we define $N_{ij}^{A}$ by
\[
 N_{ij}^{A}:=
 \begin{cases}
  \rank Y_{(i,j)}(A) & (i< j) \\
  \dim L_{i} = n_{i} & (i=j)
 \end{cases}.
\]
We call $N_{A}:=\{N_{ij}^{A}\}_{1\leq i\leq j\leq r}$
the {\it rank parameter} of $A\in \Rep(Q,\underline{n})$.
As the following lemma shows,  the rank parameter 
is an invariant which characterizes the $GL(\underline{n})$-orbit.

\begin{lemma}
 For $A\in \Rep(Q,\underline{n})$, we denote by
 $\mathcal{O}_{A}$ the $GL(\underline{n})$-orbit through $A$.
 Then for $A, B\in \Rep(Q, \underline{n})$, we have
 $\mathcal{O}_{A}\subset \overline{\mathcal{O}}_{B}$ if and only	
 if $N_{ij}^{A}\leq N_{ij}^{B}$ for $1\leq i\leq j\leq r$.
 That is, the partial ordering on the rank parameters coincides
 with the closure ordering on $GL(\underline{n})$-orbits.
 In particular, $\mathcal{O}_{A}=\mathcal{O}_{B}$ if and only if
 $N_{ij}^{A}=N_{ij}^{B}$ for $1\leq i \leq j\leq r$.
\end{lemma}

\begin{proof}
 See Abeasis and Del Fra~\cite[Theorem~5.2]{AbeasisDelFra}.
\end{proof}

Let us return to $b$-functions. Let $(p,q)\in I_{\underline{n}}(Q)$.
Note that $f_{(p,q)}(A) = \det Y_{(p,q)}(A)$ and thus
$A\in \Rep(Q,\underline{n})$
satisfies $f_{(p,q)}(A)\neq 0$ if and only if $Y_{(p,q)}(A)$ is
an isomorphism. We denote by $N^{(p,q)}$ the rank parameter which
is minimal (with respect to the above-mentioned partial ordering)
among the rank parameters $N_{A}$ such that
$Y_{(p,q)}(A)$ is an isomorphism.
The orbit $\mathcal{O}^{(p,q)}$ corresponding to $N^{(p, q)}$
is the closed $GL(\underline{n})$-orbit
in $\{A\in \Rep(Q,\underline{n})\, ; \; f_{(p,q)}(A)\neq 0\}$, and
thus it is unique (cf.\  Gyoja~\cite[Lemma~1.4]{Nonreg}).
We note that Shmelkin~\cite{Shmelkin} called
$A^{(p,q)}\in \mathcal{O}^{(p,q)}$ a {\it locally semi-simple
representation} of $Q$.

For $N^{(p,q)}=\{N_{ij}^{(p,q)}\}_{1\leq i\leq j\leq r}$, we put
\begin{align}
 \nonumber
 \mathcal{F}^{(p,q)} &:= \left\{
 \left\{N_{22}^{(p,q)}-N_{12}^{(p,q)}+1,\dots,
 N_{22}^{(p,q)}\,(=n_{2})\right\}\, ;\, \right.\\
 \label{form:DefOfFpq}
 &\qquad\quad
 \left\{N_{33}^{(p,q)}-N_{23}^{(p,q)}+1,\dots,
 N_{33}^{(p,q)}\,(=n_{3})\right\}\, ; \, \dots 
 \\ \nonumber
 &\qquad\quad
 \left. \dots \,;\,\left\{N_{rr}^{(p,q)}-N_{r-1,r}^{(p,q)}+1,\dots,
 N_{rr}^{(p,q)}\,(=n_{r})\right\}\right\}.
%  s+ \mathcal{F}^{(p,q)} &:= \left\{
% \left\{s+ N_{22}^{(p,q)}-N_{12}^{(p,q)}+1,\dots,
% s+ N_{22}^{(p,q)}\,(=s+ n_{2})\right\}\, ;\, \right.\\
% &\qquad\quad
% \left\{s+N_{33}^{(p,q)}-N_{23}^{(p,q)}+1,\dots,
% s+N_{33}^{(p,q)}\,(=s+ n_{3})\right\}\, ; \, \dots 
% \\
% &\qquad\quad
% \left. \dots \,;\,\left\{s+ N_{rr}^{(p,q)}-N_{r-1,r}^{(p,q)}+1,\dots,
% s+N_{rr}^{(p,q)}\,(=s+n_{r})\right\}\right\}
\end{align}
Note that $\mathcal{F}^{(p,q)}$ is a set consisting of $r-1$ sets, and
each set consists of consecutive natural numbers. 
However, if $N_{k-1,k}^{(p,q)}=0$, then we regard
$\{N_{kk}^{(p,q)}-N_{k-1,k}^{(p,q)}+1,\dots,
N_{kk}^{(p,q)}\,(=n_{k})\}$ as the empty set $\emptyset$.
Moreover, we define a  ``set'' of linear forms 
$s+ \mathcal{F}^{(p,q)}$ by
\begin{align*}
  s+ \mathcal{F}^{(p,q)} &:= \left\{
 \left\{s+ N_{22}^{(p,q)}-N_{12}^{(p,q)}+1,\dots,
 s+ N_{22}^{(p,q)}\,(=s+ n_{2})\right\}\, ;\, \right.\\
 &\qquad\quad
 \left\{s+N_{33}^{(p,q)}-N_{23}^{(p,q)}+1,\dots,
 s+N_{33}^{(p,q)}\,(=s+ n_{3})\right\}\, ; \, \dots 
 \\
 &\qquad\quad
 \left. \dots \,;\,\left\{s+ N_{rr}^{(p,q)}-N_{r-1,r}^{(p,q)}+1,\dots,
 s+N_{rr}^{(p,q)}\,(=s+n_{r})\right\}\right\}.
\end{align*}
Then we have the following lemma.
\begin{lemma}
 \label{lem:BfunAndRankArrays}
 Let $(p,q)\in I_{\underline{n}}(Q)$. If we multiply all the
 linear forms contained in $s+\mathcal{F}^{(p,q)}$ together,
 then we obtain the $b$-function $b_{(p,q)}(s)$.
\end{lemma}

The rest of this section is devoted to the proof of the above
lemma. 
To calculate the rank parameters, we need to construct the
lace diagrams corresponding to the locally closed orbits; let us recall
the definition of the lace diagrams (cf.~\cite{BR},\cite{KMS}).

\begin{definition}[lace diagrams]
 \label{def:LaceDiagrams}
 Consider a $GL(\underline{n})$-orbit $\mathcal{O}$ in
 $\Rep(Q,\underline{n})$.
 \begin{enumerate}
  \item  A {\it lace diagram} for the orbit $\mathcal{O}$ is
 a sequence of $r$-columns of dots, with $n_{i}$ dots in the
 $i$-th column, together with line segments connecting dots of
 consecutive columns. Each dot may be connected to at most one dot in
 the column to the left of it, and to at most one dot in the column
 to the right of it.
 \item 	Take an element $A\in \mathcal{O}$ and decompose $A$
 into the direct sum of indecomposable representations as
 \[
  A\cong \bigoplus_{1\leq i\leq j\leq r} m_{ij} I_{ij}.
 \]
 Here $I_{ij}$ is the indecomposable representation corresponding
 to the interval $[i,j]$ and $m_{ij}$ is the multiplicity. 
 See \eqref{form:IndecomposableDecomposition}. 
 \item Draw arrows between dots in such a way that there exist
 exactly $m_{ij}$ line segments starting at $i$ and terminating
 at $j$. In other words, if we take a suitably chosen basis of
 $L_{i}$ and identify them with the dots of the $i$-th column, 
 then each linear map 
 $A_{h(a), t(a)}: L_{t(a)} \to L_{h(a)}$ is given according to the
 connections between dots.  Namely, if dot $j$ of
 column $t(a)$ is connected to dot $k$ of column $h(a)$, then
 $A_{t(a), h(a)}$ maps the $j$-th basis element of $L_{t(a)}$
 to the $k$-th basis element of $L_{h(a)}$; and if dot $j$ of
 column $t(a)$ is not connected to any dot in column $h(a)$,
 then the corresponding basis element of $L_{t(a)}$ is mapped to
 zero.
 \item	Two consecutive columns connected with a rightward
 arrow (resp.\ leftward arrow) are bottom-aligned
 (resp.\ top-aligned). This 
 convention comes from \cite[p.~467]{Abeasis}.
 \end{enumerate} 
\end{definition}

\begin{example}
 \label{exmp:FirstExampleIV}
 Let $\underline{n}=(2, 5, 6,6,2)$ in
 Example~\ref{exmp:FirstExampleII} (see also
 Example~\ref{exmp:FirstExampleIII}). 
 Then we have
 \begin{align*}
  b_{(3,4)}(s)&= (s+1)(s+2)(s+3)(s+4)(s+5)(s+6), \\
  b_{(1,5)}(s)&= (s+1)(s+2)(s+4)(s+5)^3 (s+6)^2.
 \end{align*}

 Now the lace diagrams corresponding to the locally closed orbits
 $\mathcal{O}^{(3, 4)}$ and $\mathcal{O}^{(1, 5)}$ are given as
 Figure~\ref{fig:FirstExample}.
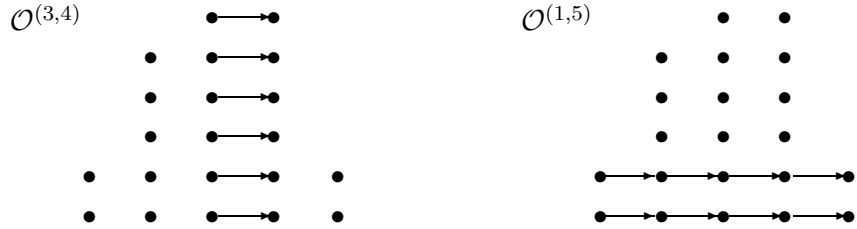
\begin{figure}[htbp]
  \begin{center}
  \begin{picture}(100,90)(0,10)
   \put(-20,80){$\mathcal{O}^{(3,4)}$}
   \put(7,7){$\bullet$}
   \put(7,22){$\bullet$}
   \put(30,7){$\bullet$}
   \put(30,22){$\bullet$}
   \put(30,37){$\bullet$}
   \put(30,52){$\bullet$}
   \put(30,67){$\bullet$}
   \put(53,7){$\bullet$}
   \put(53,22){$\bullet$}
   \put(53,37){$\bullet$}
   \put(53,52){$\bullet$}
   \put(53,67){$\bullet$}
   \put(53,82){$\bullet$}
   \put(58,10){\vector(1,0){20}}
   \put(58,25){\vector(1,0){20}}
   \put(58,40){\vector(1,0){20}}
   \put(58,55){\vector(1,0){20}}
   \put(58,70){\vector(1,0){20}}
   \put(58,85){\vector(1,0){20}}            
   \put(76,7){$\bullet$}
   \put(76,22){$\bullet$}
   \put(76,37){$\bullet$}
   \put(76,52){$\bullet$}
   \put(76,67){$\bullet$}
   \put(76,82){$\bullet$}
   \put(100,7){$\bullet$}
   \put(100,22){$\bullet$}   
  \end{picture}
   \qquad\qquad\qquad\qquad
  \begin{picture}(100,90)(0,10)
   \put(-20,80){$\mathcal{O}^{(1,5)}$}   
   \put(7,7){$\bullet$}
   \put(7,22){$\bullet$}
   \put(10,10){\vector(1,0){20}}
   \put(10,25){\vector(1,0){20}}
   \put(30,7){$\bullet$}
   \put(30,22){$\bullet$}
   \put(30,37){$\bullet$}
   \put(30,52){$\bullet$}
   \put(30,67){$\bullet$}
   \put(34,10){\vector(1,0){20}}
   \put(34,25){\vector(1,0){20}}
   \put(53,7){$\bullet$}
   \put(53,22){$\bullet$}
   \put(53,37){$\bullet$}
   \put(53,52){$\bullet$}
   \put(53,67){$\bullet$}
   \put(53,82){$\bullet$}
   \put(58,10){\vector(1,0){20}}
   \put(58,25){\vector(1,0){20}}
   \put(76,7){$\bullet$}
   \put(76,22){$\bullet$}
   \put(76,37){$\bullet$}
   \put(76,52){$\bullet$}
   \put(76,67){$\bullet$}
   \put(76,82){$\bullet$}
   \put(82,10){\vector(1,0){20}}
   \put(82,25){\vector(1,0){20}}
   \put(100,7){$\bullet$}
   \put(100,22){$\bullet$}   
  \end{picture}      
  \end{center}
    \caption{Lace diagrams corresponding to $\mathcal{O}^{(3,4)}$
 and $\mathcal{O}^{(1,5)}$}
    \label{fig:FirstExample} 
\end{figure}
Note that if any array in the diagram  is erased, then the condition
$f_{(p,q)}(A)\neq 0$ is not satisfied, and conversely if some extra
array is added, then the minimality condition is not satisfied.
Thus we see that  
 the rank parameters $N^{(3,4)}$ and $N^{(1,5)}$ are given
 by \eqref{form:RankParametersI}.
 \begin{equation}
  \label{form:RankParametersI}
  N^{(3, 4)}: \quad
 \begin{array}[t]{ccccc}
  2& 0& 0& 0& 0\\
  & 5& 0& 0& 0\\
  & & 6& 6& 0\\
  & & & 6& 0\\
  & & & & 2
 \end{array}\qquad\qquad
  N^{(1, 5)}: \quad
 \begin{array}[t]{ccccc}
  2& 2& 2& 2& 2\\
  & 5& 2& 2& 2\\
  & & 6& 2& 2\\
  & & & 6& 2\\
  & & & & 2
 \end{array}
 \end{equation}
 By \eqref{form:DefOfFpq}, the rank parameters read
 $\mathcal{F}^{(3, 4)}, \,
 s+ \mathcal{F}^{(3,4)}$ and $\mathcal{F}^{(1, 5)}, \,
 s+\mathcal{F}^{(1,5)}$  as
 \begin{align*}
  \mathcal{F}^{(3, 4)}&=\left\{
  \emptyset\; ;\; \emptyset\; ; \;
  \{1,2,3,4,5,6\}\; ; \; \emptyset
  \right\}, \\  
  s+\mathcal{F}^{(3, 4)}&=\left\{
  \emptyset\; ;\; \emptyset\; ; \;
  \{s+1,s+2,s+3,s+4,s+5,s+6\}\; ; \; \emptyset
  \right\}, \\
  \mathcal{F}^{(1,5)}&=\left\{
  \{4, 5\}\; ; \; \{5, 6\}\; ; \; \{5, 6\}\; ; \;
  \{1, 2\}
  \right\},  \\
  s+\mathcal{F}^{(1,5)}&=\left\{
  \{s+4, s+5\}\; ; \; \{s+5, s+6\}\; ; \; \{s+5, s+6\}\; ; \;
  \{s+1, s+2\}
  \right\}.
 \end{align*}
 Note that if we multiply all the linear forms contained in
 $s+\mathcal{F}^{(3,4)}$
 (resp.\ $s+\mathcal{F}^{(1,5)}$), then we obtain
  the $b$-function 
 $b_{(3,4)}(s)$ (resp.\ $b_{(1,5)}(s)$).
\end{example}

\begin{example}
 \label{exmp:SecondExampleIV}
 Let $\underline{n}=(2, 5, 7,4,2)$ in
 Example~\ref{exmp:SecondExampleII}
 (see also Example~\ref{exmp:SecondExampleIII}). 
 Then we have
 \begin{align*}
  b_{(1,4)}(s)&= (s+1)(s+2)(s+3)(s+4)^2 (s+5)^2 (s+6)(s+7), \\
  b_{(2,5)}(s)&=
  (s+1)(s+2)(s+3)^2(s+4)^2 (s+5)(s+6)(s+7).
 \end{align*}
 The lace diagrams corresponding to $\mathcal{O}^{(1, 4)}$ and
 $\mathcal{O}^{(2, 5)}$ are given as in
 Figure~\ref{fig:SecondExample}.
 \begin{figure}[htbp]
  \begin{center}
  \begin{picture}(100,105)(0,10)
   \put(-30,90){$\mathcal{O}^{(1,4)}$}
   \put(7,37){$\bullet$}
   \put(7,52){$\bullet$}
   \put(10,40){\vector(1,0){20}}
   \put(10,55){\vector(1,0){20}}   
   \put(30,37){$\bullet$}
   \put(30,52){$\bullet$}
   \put(30,67){$\bullet$}
   \put(30,82){$\bullet$}
   \put(30,97){$\bullet$}
   \put(53,70){\vector(-1,0){20}}      
   \put(53,85){\vector(-1,0){20}}   
   \put(53,100){\vector(-1,0){20}}
   \put(53,7){$\bullet$}
   \put(53,22){$\bullet$}
   \put(53,37){$\bullet$}
   \put(53,52){$\bullet$}
   \put(53,67){$\bullet$}
   \put(53,82){$\bullet$}
   \put(53,97){$\bullet$}
   \put(56,10){\vector(1,0){20}}
   \put(56,25){\vector(1,0){20}}
   \put(56,40){\vector(1,0){20}}
   \put(56,55){\vector(1,0){20}}            
   \put(76,7){$\bullet$}
   \put(76,22){$\bullet$}
   \put(76,37){$\bullet$}
   \put(76,52){$\bullet$}
  \put(100,37){$\bullet$}
  \put(100,52){$\bullet$}
  \end{picture}
   \qquad\qquad\qquad\qquad
  \begin{picture}(100,105)(0,10)
   \put(-30,90){$\mathcal{O}^{(2,5)}$}
   \put(7,37){$\bullet$}
   \put(7,52){$\bullet$}
   \put(30,37){$\bullet$}
   \put(30,52){$\bullet$}
   \put(30,67){$\bullet$}
   \put(30,82){$\bullet$}
   \put(30,97){$\bullet$}
   \put(53,40){\vector(-1,0){20}}            
   \put(53,55){\vector(-1,0){20}}         
   \put(53,70){\vector(-1,0){20}}      
   \put(53,85){\vector(-1,0){20}}   
   \put(53,100){\vector(-1,0){20}}
   \put(53,7){$\bullet$}
   \put(53,22){$\bullet$}
   \put(53,37){$\bullet$}
   \put(53,52){$\bullet$}
   \put(53,67){$\bullet$}
   \put(53,82){$\bullet$}
   \put(53,97){$\bullet$}
   \put(56,10){\vector(1,0){20}}
   \put(56,25){\vector(1,0){20}}
   \put(76,7){$\bullet$}
   \put(76,22){$\bullet$}
   \put(76,37){$\bullet$}
   \put(76,52){$\bullet$}
  \put(100,37){$\bullet$}
  \put(100,52){$\bullet$}
  \put(100,40){\vector(-1,0){20}}
  \put(100,55){\vector(-1,0){20}}                  
  \end{picture}         
  \end{center}
    \caption{Lace diagrams corresponding to
 $\mathcal{O}^{(1,4)}$ and $\mathcal{O}^{(2,5)}$}
    \label{fig:SecondExample} 
 \end{figure}
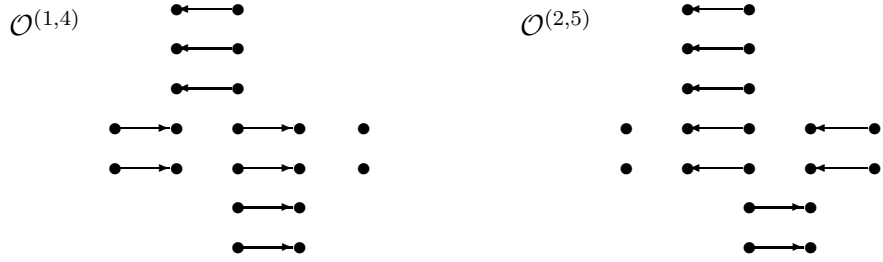 
 Here the reader is refereed to the condition~4 of
 Definition~\ref{def:LaceDiagrams}. Now we see that 
 the rank parameters $N^{(1,4)}$ and $N^{(2,5)}$ are given by
 \eqref{form:RankParametersII}.
 \begin{equation}
  \label{form:RankParametersII}
  N^{(1, 4)}: \quad
 \begin{array}[t]{ccccc}
  2& 2& 5& 9& 9\\
  & 5& 3& 7& 7\\
  & & 7& 4& 4\\
  & & & 4& 0\\
  & & & & 2
 \end{array}\qquad\qquad
  N^{(2, 5)}: \quad
 \begin{array}[t]{ccccc}
  2& 0& 5& 7& 9\\
  & 5& 5& 7& 9\\
  & & 7& 2& 4\\
  & & & 4& 2\\
  & & & & 2
 \end{array}
 \end{equation}
 Hence we observe that $\mathcal{F}^{(1,4)}, s+ \mathcal{F}^{(1,4)}$
 and $\mathcal{F}^{(2,5)}, s+\mathcal{F}^{(2,5)}$ are
 \begin{align*}
  \mathcal{F}^{(1, 4)}&=\left\{
  \{4, 5\}\; ;\; \{5, 6,7\}\; ; \;
  \{1,2,3,4\}\; ; \; \emptyset
  \right\}, \\
  s+\mathcal{F}^{(1, 4)}&=\left\{
  \{s+4, s+5\}\; ;\; \{s+5, s+6,s+7\}\; ; \;
  \{s+1,s+2,s+3,s+4\}\; ; \; \emptyset
  \right\}, \\  
  \mathcal{F}^{(2,5)}&=\left\{
  \emptyset\; ; \; \{3,4,5, 6,7\}\; ; \; \{3, 4\}\; ; \;
  \{1, 2\}
  \right\},\\
  s+\mathcal{F}^{(2,5)}&=\left\{
  \emptyset\; ; \; \{s+3,s+4,s+5,s+6,s+7\}\; ; \;
  \{s+3,s+4\}\; ; \;\{s+1, s+2\}
  \right\},  
 \end{align*}
 and that if we multiply all the linear forms contained in
  $s+\mathcal{F}^{(1,4)}$ (resp.\ $s+\mathcal{F}^{(2,5)}$), then
 we obtain  the $b$-function
 $b_{(1,4)}(s)$ (resp.\ $b_{(2,5)}(s)$). 
\end{example}

Now we accurately describe how to construct
the lace diagrams such as Figures~\ref{fig:FirstExample}
and~\ref{fig:SecondExample}.
 
\begin{definition}[exact lace diagrams]
 Let $(p,q)\in I_{\underline{n}}(Q)$ and $\mathcal{O}^{(p,q)}$
 the closed orbit in
 $\{A\in \Rep(Q,\underline{n})\, ;\, f_{(p,q)}(A)\neq 0\}$.
 Then we construct the lace diagram corresponding to
 $\mathcal{O}^{(p,q)}$ according to the following convention.
 \begin{enumerate}
  \item For $\nu=p,\dots, q$, 
        let $e_{1}^{(\nu)},\dots, e_{n_{\nu}}^{(\nu)}$ be a basis of $L_{\nu}$
	and we identify these vectors with the dots in
	column $\nu$ in the lace diagram.
  \item First we consider the case of $\delta(p)=1$, i.e., we 
	assume that	
	the arrow $p \in Q_{1}$ is rightward. Then, from each
	$e_{j}^{(p)}$ for $j=1,\dots, n_{p}$, we draw a {\it horizontal}
	arrow to a	
	dot in column $(p+1)$.
	Choose a suitable order of the basis of $L_{p+1}$ 
	so that these arrows give a
	one-to-one correspondence between
	$e_{1}^{(p)},\dots, e_{n_{p}}^{(p)}$ and
	$e_{1}^{(p+1)},\dots, e_{n_{p}}^{(p+1)}$. Stop here
	if $p+1=q$. 
  \item If $p+1< q$ and $\delta(p+1) = 1$,
	then we draw a horizontal arrow
	from each $e_{j}^{(p+1)}$ for $j=1,\dots, n_{p}$ 
	to a dot in column $(p+2)$.
	If $p+1< q$ and $\delta(p+1) = -1$, then we draw a horizontal
	arrow {\it to} each $e_{j}^{(p+1)}$ for $j= n_{p}+1,\dots,
	n_{p+1}$ {\it from} some dot in column $(p+2)$.
	Continue to draw arrows until
	we reach the column $q$. 
 \item  In the case of $\delta(p)= -1$, we reverse the orientations in the 
	above 2., 3. 
 \end{enumerate}
 We remark that in general, two different lace diagrams may correspond
 to the same orbit. However, under the above convention,
 the locally closed orbit $\mathcal{O}^{(p, q)}$ uniquely determines the
 lace diagram. We call it {\it the exact lace diagram} corresponding to
 $\mathcal{O}^{(p,q)}$.
\end{definition}

\begin{proof}[Proof of Lemma~\ref{lem:BfunAndRankArrays}]
 We compare the construction of the exact lace diagrams with
 the reduction formulas \eqref{form:ReductionI} and
 \eqref{form:ReductionII}. 
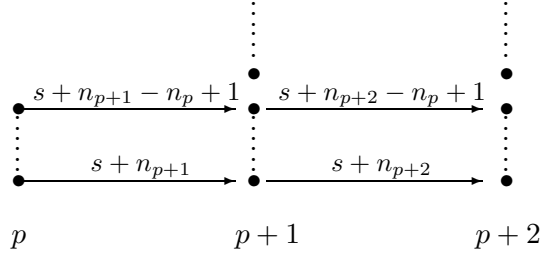
\begin{figure}
 \begin{center}
  \begin{picture}(220,100)(0,10)
   \setlength{\unitlength}{0.9pt}
   \put(7,7){$p$}
   \put(100,7){$p+1$}
   \put(200,7){$p+2$}   
   \put(7,30){$\bullet$}
   \put(8,37){$\vdots$}
   \put(8,50){$\vdots$}   
   \put(105,30){$\bullet$}
   \put(106,37){$\vdots$}
   \put(106,50){$\vdots$}   
   \put(210,30){$\bullet$}
   \put(211,37){$\vdots$}
   \put(211,50){$\vdots$}      
   \put(7,60){$\bullet$}
   \put(105,60){$\bullet$}
   \put(105,75){$\bullet$}
   \put(106,85){$\vdots$}
   \put(106,98){$\vdots$}      
   \put(210,60){$\bullet$}
   \put(210,75){$\bullet$}
   \put(211,85){$\vdots$}
   \put(211,98){$\vdots$}         
   \put(10,33){\vector(1,0){90}}
   \put(10,63){\vector(1,0){90}}
   \put(113,33){\vector(1,0){90}}
   \put(113,63){\vector(1,0){90}}   
   \put(40,37){\small $s+n_{p+1}$}   
   \put(16,67){\small $s+n_{p+1}-n_{p}+1$}
   \put(140,37){\small $s+n_{p+2}$}   
   \put(118,67){\small $s+n_{p+2}-n_{p}+1$}   
  \end{picture}
 \end{center}
 \caption{Reduction formula \eqref{form:ReductionI} and the 
 exact lace diagram}
 \label{fig:ArrowsAndFactorsI}
\end{figure}
 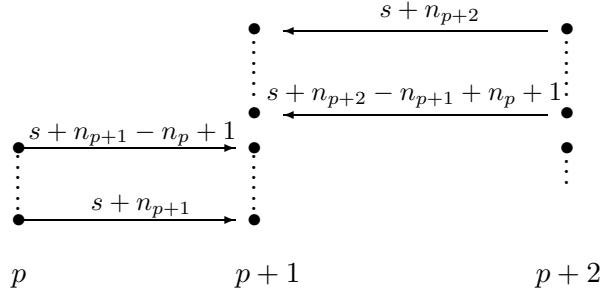
\begin{figure}
 \begin{center}
  \begin{picture}(240,120)(0,10)
      \setlength{\unitlength}{0.9pt}
   \put(7,7){$p$}
   \put(100,7){$p+1$}
   \put(225,7){$p+2$}   
   \put(7,30){$\bullet$}
   \put(8,37){$\vdots$}
   \put(8,50){$\vdots$}   
   \put(105,30){$\bullet$}
   \put(106,37){$\vdots$}
   \put(106,50){$\vdots$}   
   \put(236,48){$\vdots$}      
   \put(7,60){$\bullet$}
   \put(105,60){$\bullet$}
   \put(105,75){$\bullet$}
   \put(106,85){$\vdots$}
   \put(106,98){$\vdots$}
   \put(105,110){$\bullet$}   
   \put(235,60){$\bullet$}
   \put(235,75){$\bullet$}
   \put(236,85){$\vdots$}
   \put(236,98){$\vdots$}
   \put(235,110){$\bullet$}
   \put(10,33){\vector(1,0){90}}
   \put(10,63){\vector(1,0){90}}
   \put(230,77){\vector(-1,0){110}}
   \put(230,112){\vector(-1,0){110}}      
   \put(40,37){\small $s+n_{p+1}$}   
   \put(14,67){\small $s+n_{p+1}-n_{p}+1$}
   \put(113,83){\small $s+n_{p+2}-n_{p+1}+n_{p}+1$}
   \put(160,118){\small $s+n_{p+2}$}      
  \end{picture}
 \end{center}
 \caption{Reduction formula \eqref{form:ReductionII} and the
 exact lace diagram}
 \label{fig:ArrowsAndFactorsII}  
\end{figure} 
 Then we see that there exists a one-to-one correspondence between
 the arrows in the exact lace diagrams and the factors of
 $b$-functions. 
 See Figures~\ref{fig:ArrowsAndFactorsI} and~\ref{fig:ArrowsAndFactorsII}. 
 Thus the lemma is proved. 
\end{proof} 

\section{$b$-Functions of several variables}
\label{section:BfunctionsOfSeveralVariables}

In this section, we describe an algorithm to calculate the
multi-variate $b$-function
$b_{\underline{m}}(\underline{s})$ of $(GL(\underline{n}),
\Rep(Q,\underline{n}))$, and give some examples. 
We take an arbitrary numbering on
$I_{\underline{n}}(Q)$ and let
\[
 I_{\underline{n}}(Q) = \left\{(p_{1},q_{1}),(p_{2},q_{2}),\dots,
 (p_{l}, q_{l})
 \right\}.
\]
In the following, we write as
$f_{(p_{1},q_{1})}= f_{1},\, N^{(p_{2}, q_{2})} =
N^{(2)},\, \mathcal{F}^{(p_{3},q_{3})}=\mathcal{F}^{(3)},\dots$.

\bigskip

\noindent
{\large $1^{\circ}.$}  \ First, for given linear forms
$s_{i_{1}} + \alpha,\,
s_{i_{2}} + \alpha, \, \cdots, s_{i_{t}} + \alpha$ with the
{\it same constant term}, we define the 
{\it superposition operation} as follows:
\begin{equation}
 s_{i_{1}} + \alpha,\;
s_{i_{2}} + \alpha, \cdots, s_{i_{t}} + \alpha \quad
\longmapsto \quad
s_{i_{1}}+s_{i_{2}}+\cdots + s_{i_{t}} + \alpha
\end{equation}
Carry out this operation on the $(k-1)$-th components
\[
\begin{array}{c}
 \left\{ s_{1}+N_{k,k-1}^{(1)}- N_{kk}^{(1)}+1,\dots,
      s_{1} + N_{kk}^{(1)}\, (=s_{1} + n_{k})
      \right\}, \\[8pt]
 \left\{ s_{2}+N_{k,k-1}^{(2)}- N_{kk}^{(2)}+1,\dots,
      s_{2} + N_{kk}^{(2)}\, (=s_{2} + n_{k})
      \right\}, \\[8pt]
 \multicolumn{1}{c}{
 \dots\dots\dots\dots\dots\dots}  \\[8pt]
 \left\{ s_{l}+N_{k,k-1}^{(l)}- N_{kk}^{(l)}+1,\dots,
      s_{l} + N_{kk}^{(l)}\, (=s_{l} + n_{k})
      \right\}
\end{array}
\]
of 
$s_{1}+\mathcal{F}^{(1)}, s_{2}+
\mathcal{F}^{(2)},\dots, s_{l}+ \mathcal{F}^{(l)}$.
Here we ignore the empty set $\emptyset$.
\begin{example}
 From
 \[
  \begin{array}{r}
   \{s_{1}+3, \, s_{1}+4, \, s_{1}+5\},\\
   \{s_{2}+4,\, s_{2}+5\}, \\
   \emptyset,\\   
   \{s_{4}+1,\, s_{4}+2, \, s_{4}+3, \, s_{4}+4,\,
    s_{4}+5\},
  \end{array}
 \]
 we obtain the following new linear forms
 \[
  s_{4}+1, \; s_{4}+ 2,\;
 s_{1}+s_{4}+3, \; s_{1}+s_{2}+s_{4}+4,\;
 s_{1}+s_{2}+s_{4}+5.
 \]
\end{example}

\medskip

\noindent
{\large $2^{\circ}.$}  \ Carry out the operation {$1^{\circ}$}
for all $k=2,\dots, r$.

\medskip

\noindent
{\large $3^{\circ}.$} \ Substitute the linear forms
obtained in {$2^{\circ}$} by the rule
\[
 s_{i_{1}} + s_{i_{2}}+\cdots +s_{i_{t}} + \alpha\quad  \longmapsto
 \quad
 [s_{i_{1}} + s_{i_{2}}+\cdots +s_{i_{t}} + \alpha]_{
 m_{i_{1}} + m_{i_{2}}+\cdots +m_{i_{t}}},
\]
and then multiply all of them together.
Here the square parentheses symbol stands for
\[
 [A]_{m}:= A(A+1)(A+2)\cdots (A+m-1).
\]

\begin{theorem}
 \label{theorem:MainTheorem}
 The output of the operation $3^{\circ}$ is the 
$b$-function $b_{\underline{m}}(\underline{s})$ of
$\underline{f}=(f_{1},\dots, f_{l})$.
\end{theorem}

The proof of the theorem will be given in the following sections. 
In the remainder of this section, we calculate the $b$-functions
$b_{\underline{m}}(\underline{s})$ for the two examples
discussed in the previous sections.

\begin{example}
  \label{exmp:FirstExampleV}
 In Example~\ref{exmp:FirstExampleIV}, put $f_{1} :=f_{(3,4)}$,
 $f_{2}:= f_{(1,5)}$. For
  \begin{align*}
  s_{1}+ \mathcal{F}^{(1)}&=
   s_{1}+ \mathcal{F}^{(3, 4)}=\left\{
  \emptyset\; ;\; \emptyset\; ; \;
  \{s_{1}+ 1,
   s_{1}+ 2,s_{1}+ 3,s_{1}+ 4,s_{1}+ 5,
   s_{1}+ 6\}\; ; \; \emptyset
  \right\}, \\
  s_{2}+ \mathcal{F}^{(2)}&=
   s_{2}+ \mathcal{F}^{(1,5)}\\
  &=\left\{
  \{s_{2}+4, s_{2}+5\}\; ; \; \{s_{2}+5, s_{2}+6\}\; ; \;
   \{s_{2}+5, s_{2}+6\}\; ; \;
  \{s_{2}+1, s_{2}+2\}
  \right\},
 \end{align*}
 we carry out the operation $1^{\circ}$. Since the $1$-st, $2$-nd,
 $4$-th components of $\mathcal{F}^{(3, 4)}$ are the empty sets,
 we obtain $\{s_{2}+4, s_{2}+5\}\, \, \{s_{2}+5, s_{2}+6\}, \,
 \{s_{2}+1, s_{2}+2\}$ from the $1$-st, $2$-nd,
 $4$-th components, and at the $3$-rd component, we superpose
 the linear forms as follows:
 \[
  \begin{array}{r}
   \{s_{1}+ 1,
   s_{1}+ 2,s_{1}+ 3,s_{1}+ 4,s_{1}+ 5,
   s_{1}+ 6\}\\
   \{s_{2}+5, s_{2}+6\}
  \end{array}\; \longmapsto \;
 \begin{array}{l}
  s_{1}+1, \, s_{1}+2,\, s_{1}+3, \, s_{1}+4, \\
 s_{1}+s_{2}+ 5, \, s_{1}+s_{2}+ 6
 \end{array}.
 \]
 All the linear forms are aligned as
 \[
  \begin{array}{l}
   s_{1}+1, \, s_{1}+2,\, s_{1}+3, \, s_{1}+4, \\
   s_{2}+1, \, s_{2}+2,\, s_{2}+4, \, (s_{2}+5)^{\times 2}, \,
    s_{2}+ 6, \\
    s_{1}+s_{2}+ 5, \, s_{1}+s_{2}+ 6
  \end{array}
 \]
 and by multiplying these factors according to $3^{\circ}$, 
 we obtain the $b$-function
 $b_{\underline{m}}(\underline{s})$. That is, 
 \begin{align}
  \nonumber
  b_{\underline{m}}(\underline{s})&=
  [s_{1}+1]_{m_{1}} [s_{1}+2]_{m_{1}} [s_{1}+3]_{m_{1}}
  [s_{1}+4]_{m_{1}} \\
  \label{form:BfunOfFirstExample}
  &\quad \times
    [s_{2}+1]_{m_{2}}[s_{2}+2]_{m_{2}} [s_{2}+4]_{m_{2}}
  [s_{2}+5]_{m_{2}}^{2} [s_{2}+ 6]_{m_{2}}\\
  \nonumber
  &\quad \times
  [s_{1}+s_{2}+ 5]_{m_{1}+m_{2}} [s_{1}+s_{2}+ 6]_{m_{1}+m_{2}}.
 \end{align}
 The aspect of the superposition can be visualized as follows:
 First, as in Figure~\ref{fig:LDinFirstExample}, 
 we attach the linear forms in $s_{1}+ \mathcal{F}^{(3,4)}$
 (resp.\ $s_{2}+\mathcal{F}^{(1,5)}$) to the arrows in
 the lace diagram corresponding to $\mathcal{O}^{(3, 4)}$
 (resp.\ $\mathcal{O}^{(1, 5)}$).
 See Figure~\ref{fig:ArrowsAndFactorsI} for the way of attaching
 the factors. 
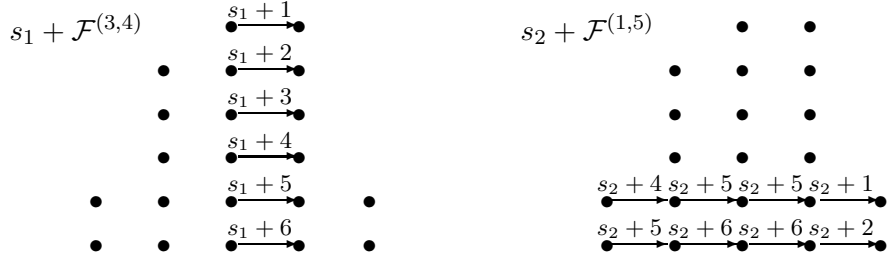
\begin{figure}[htbp]
  \begin{center}
  \begin{picture}(100,90)(0,15)
   \setlength{\unitlength}{1.1pt}
   \put(-20,80){$s_{1}+ \mathcal{F}^{(3,4)}$}
   \put(7,7){$\bullet$}
   \put(7,22){$\bullet$}
   \put(30,7){$\bullet$}
   \put(30,22){$\bullet$}
   \put(30,37){$\bullet$}
   \put(30,52){$\bullet$}
   \put(30,67){$\bullet$}
   \put(53,7){$\bullet$}
   \put(53,22){$\bullet$}
   \put(53,37){$\bullet$}
   \put(53,52){$\bullet$}
   \put(53,67){$\bullet$}
   \put(53,82){$\bullet$}
   \put(58,10){\vector(1,0){20}}
   \put(58,25){\vector(1,0){20}}
   \put(58,40){\vector(1,0){20}}
   \put(58,55){\vector(1,0){20}}
   \put(58,70){\vector(1,0){20}}
   \put(58,85){\vector(1,0){20}}
   \put(54,88){{\footnotesize$s_{1}+1$}}
   \put(54,73){{\footnotesize$s_{1}+2$}}
   \put(54,58){{\footnotesize$s_{1}+3$}}
   \put(54,43){{\footnotesize$s_{1}+4$}}
   \put(54,28){{\footnotesize$s_{1}+5$}}
   \put(54,13){{\footnotesize$s_{1}+6$}}         
   \put(76,7){$\bullet$}
   \put(76,22){$\bullet$}
   \put(76,37){$\bullet$}
   \put(76,52){$\bullet$}
   \put(76,67){$\bullet$}
   \put(76,82){$\bullet$}
   \put(100,7){$\bullet$}
   \put(100,22){$\bullet$}   
  \end{picture}
   \qquad\qquad\qquad\qquad
  \begin{picture}(100,90)(0,15)
      \setlength{\unitlength}{1.1pt}
   \put(-20,80){$s_{2}+\mathcal{F}^{(1,5)}$}   
   \put(7,7){$\bullet$}
   \put(7,22){$\bullet$}
   \put(10,10){\vector(1,0){20}}
   \put(10,25){\vector(1,0){20}}
   \put(6,28){{\footnotesize$s_{2}+4$}}
   \put(6,13){{\footnotesize$s_{2}+5$}}
   \put(30,7){$\bullet$}
   \put(30,22){$\bullet$}
   \put(30,37){$\bullet$}
   \put(30,52){$\bullet$}
   \put(30,67){$\bullet$}
   \put(34,10){\vector(1,0){20}}
   \put(34,25){\vector(1,0){20}}
   \put(30,28){{\footnotesize$s_{2}+5$}}
   \put(30,13){{\footnotesize$s_{2}+6$}}   
   \put(53,7){$\bullet$}
   \put(53,22){$\bullet$}
   \put(53,37){$\bullet$}
   \put(53,52){$\bullet$}
   \put(53,67){$\bullet$}
   \put(53,82){$\bullet$}
   \put(58,10){\vector(1,0){20}}
   \put(58,25){\vector(1,0){20}}
   \put(54,28){{\footnotesize$s_{2}+5$}}
   \put(54,13){{\footnotesize$s_{2}+6$}}
   \put(76,7){$\bullet$}
   \put(76,22){$\bullet$}
   \put(76,37){$\bullet$}
   \put(76,52){$\bullet$}
   \put(76,67){$\bullet$}
   \put(76,82){$\bullet$}
   \put(82,10){\vector(1,0){20}}
   \put(82,25){\vector(1,0){20}}
   \put(78,28){{\footnotesize$s_{2}+1$}}
   \put(78,13){{\footnotesize$s_{2}+2$}}   
   \put(100,7){$\bullet$}
   \put(100,22){$\bullet$}   
  \end{picture}      
  \end{center}
    \caption{Lace diagrams corresponding to $s_{1}+\mathcal{F}^{(3,4)}$
 and $s_{2}+  \mathcal{F}^{(1,5)}$}
 \label{fig:LDinFirstExample}
\end{figure}
 Then we superpose these two diagrams. If two linear forms are
 attached to the same arrow, those two linear forms are also
 superposed as in
 Figure~\ref{fig:SuperpositionInFirstExample}.
\begin{figure}[htbp] 
\begin{center}
   \begin{picture}(100,90)(0,15)
    \setlength{\unitlength}{1.1pt}
%   \put(-20,80){$s_{2}+\mathcal{F}^{(1,5)}$}   
   \put(7,7){$\bullet$}
   \put(7,22){$\bullet$}
   \put(10,10){\vector(1,0){20}}
   \put(10,25){\vector(1,0){20}}
   \put(6,28){{\footnotesize$s_{2}+4$}}
   \put(6,13){{\footnotesize$s_{2}+5$}}
   \put(30,7){$\bullet$}
   \put(30,22){$\bullet$}
   \put(30,37){$\bullet$}
   \put(30,52){$\bullet$}
   \put(30,67){$\bullet$}
   \put(34,10){\vector(1,0){20}}
   \put(34,25){\vector(1,0){20}}
   \put(30,28){{\footnotesize$s_{2}+5$}}
   \put(30,13){{\footnotesize$s_{2}+6$}}   
   \put(53,7){$\bullet$}
   \put(53,22){$\bullet$}
   \put(53,37){$\bullet$}
   \put(53,52){$\bullet$}
   \put(53,67){$\bullet$}
   \put(53,82){$\bullet$}
   \put(58,10){\vector(1,0){43}}
   \put(58,25){\vector(1,0){43}}
   \put(58,40){\vector(1,0){43}}
   \put(58,55){\vector(1,0){43}}
   \put(58,70){\vector(1,0){43}}
   \put(58,85){\vector(1,0){43}}
   \put(65,88){{\footnotesize$s_{1}+1$}}        
   \put(65,73){{\footnotesize$s_{1}+2$}}    
   \put(65,58){{\footnotesize$s_{1}+3$}}            
   \put(65,43){{\footnotesize$s_{1}+4$}}        
   \put(59,28){{\footnotesize$s_{1}+s_{2}+5$}}
   \put(59,13){{\footnotesize$s_{1}+s_{2}+6$}}
   \put(100,7){$\bullet$}
   \put(100,22){$\bullet$}
   \put(100,37){$\bullet$}
   \put(100,52){$\bullet$}
   \put(100,67){$\bullet$}
   \put(100,82){$\bullet$}
   \put(105,10){\vector(1,0){20}}
   \put(105,25){\vector(1,0){20}}
   \put(105,28){{\footnotesize$s_{2}+1$}}
   \put(105,13){{\footnotesize$s_{2}+2$}}   
   \put(123,7){$\bullet$}
   \put(123,22){$\bullet$}   
  \end{picture}      
\end{center}
 \caption{Superposition of the lace diagrams in
 Example~\ref{exmp:FirstExampleV}}
 \label{fig:SuperpositionInFirstExample}
 \end{figure}
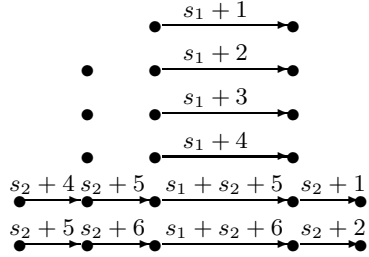
\end{example}

\begin{example}
 \label{exmp:SecondExampleV}
 In Example~\ref{exmp:SecondExampleIV}, we put 
 $f_{1}:= f_{(1,4)}$, $f_{2}:=f_{(2, 5)}$. For 
 \begin{align*}
  s_{1}+ \mathcal{F}^{(1,4)}&=\left\{
  \{s_{1}+ 4, s_{1}+ 5\}\; ;\; \{s_{1}+ 5,
  s_{1}+ 6,s_{1}+ 7\}\; ; \;
  \{s_{1}+ 1,s_{1}+ 2,s_{1}+ 3,s_{1}+ 4\}\; ; \; \emptyset
  \right\}, \\
  s_{2}+ \mathcal{F}^{(2,5)}
  &=\left\{
  \emptyset\; ; \; \{
  s_{2}+ 3,s_{2}+ 4,s_{2}+ 5, s_{2}+ 6,s_{2}+ 7\}\; ; \;
  \{s_{2}+ 3, s_{2}+ 4\}\; ; \;
  \{s_{2}+ 1, s_{2}+ 2\}
  \right\}, 
 \end{align*}
 we perform the operations
 $1^{\circ}, 2^{\circ}, 3^{\circ}$, and obtain
 \begin{align}
  \nonumber
  b_{\underline{m}}(\underline{s})&=
  [s_{1}+1]_{m_{1}}[s_{1}+2]_{m_{1}}[s_{1}+4]_{m_{1}}
  [s_{1}+5]_{m_{1}} \\
  \label{form:BfunOfSecondExample}
  &\quad\times
  [s_{2}+1]_{m_{2}}[s_{2}+2]_{m_{2}}[s_{2}+3]_{m_{2}}
  [s_{2}+4]_{m_{2}} \\
  \nonumber
  &\quad \times
  [s_{1}+s_{2}+ 3]_{m_{1}+m_{2}}
  [s_{1}+s_{2}+ 4]_{m_{1}+m_{2}}
  [s_{1}+s_{2}+ 5]_{m_{1}+m_{2}}\\
  \nonumber
  &\quad\times
  [s_{1}+s_{2}+ 6]_{m_{1}+m_{2}}
  [s_{1}+s_{2}+ 7]_{m_{1}+m_{2}}.  
 \end{align}
 Also in this case, the aspect of the superposition can be interpreted
 visually. 
 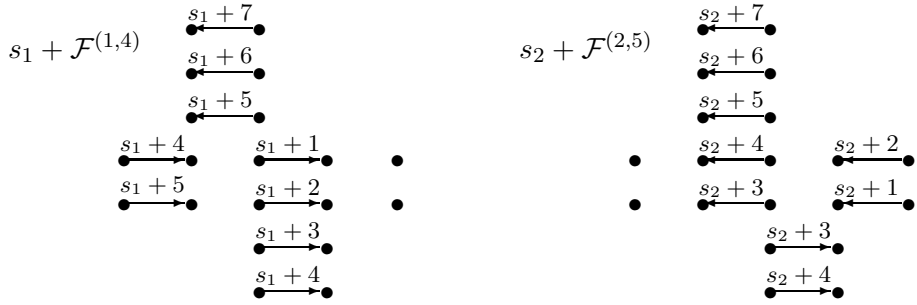
\begin{figure}[htbp]
  \begin{center}
  \begin{picture}(100,105)(0,15)
    \setlength{\unitlength}{1.1pt}
   \put(-30,90){$s_{1}+\mathcal{F}^{(1,4)}$}
   \put(7,37){$\bullet$}
   \put(7,52){$\bullet$}
   \put(10,40){\vector(1,0){20}}
   \put(10,55){\vector(1,0){20}}
   \put(8,44){{\footnotesize$s_{1}+5$}}
   \put(8,58){{\footnotesize$s_{1}+4$}}   
   \put(30,37){$\bullet$}
   \put(30,52){$\bullet$}
   \put(30,67){$\bullet$}
   \put(30,82){$\bullet$}
   \put(30,97){$\bullet$}
   \put(53,70){\vector(-1,0){20}}      
   \put(53,85){\vector(-1,0){20}}   
   \put(53,100){\vector(-1,0){20}}
   \put(31,73){{\footnotesize$s_{1}+5$}}
   \put(31,88){{\footnotesize$s_{1}+6$}}
   \put(31,103){{\footnotesize$s_{1}+7$}}         
   \put(53,7){$\bullet$}
   \put(53,22){$\bullet$}
   \put(53,37){$\bullet$}
   \put(53,52){$\bullet$}
   \put(53,67){$\bullet$}
   \put(53,82){$\bullet$}
   \put(53,97){$\bullet$}
   \put(56,10){\vector(1,0){20}}
   \put(56,25){\vector(1,0){20}}
   \put(56,40){\vector(1,0){20}}
   \put(56,55){\vector(1,0){20}}
   \put(54,13){{\footnotesize$s_{1}+4$}}
   \put(54,28){{\footnotesize$s_{1}+3$}}
   \put(54,43){{\footnotesize$s_{1}+2$}}
   \put(54,58){{\footnotesize$s_{1}+1$}}            
   \put(76,7){$\bullet$}
   \put(76,22){$\bullet$}
   \put(76,37){$\bullet$}
   \put(76,52){$\bullet$}
  \put(100,37){$\bullet$}
  \put(100,52){$\bullet$}
  \end{picture}
   \qquad\qquad\qquad\qquad
  \begin{picture}(100,105)(0,15)
    \setlength{\unitlength}{1.1pt}
   \put(-30,90){$s_{2}+\mathcal{F}^{(2,5)}$}
   \put(7,37){$\bullet$}
   \put(7,52){$\bullet$}
   \put(30,37){$\bullet$}
   \put(30,52){$\bullet$}
   \put(30,67){$\bullet$}
   \put(30,82){$\bullet$}
   \put(30,97){$\bullet$}
   \put(53,40){\vector(-1,0){20}}            
   \put(53,55){\vector(-1,0){20}}         
   \put(53,70){\vector(-1,0){20}}      
   \put(53,85){\vector(-1,0){20}}   
   \put(53,100){\vector(-1,0){20}}
   \put(31,43){{\footnotesize$s_{2}+3$}}      
   \put(31,58){{\footnotesize$s_{2}+4$}}   
   \put(31,73){{\footnotesize$s_{2}+5$}}
   \put(31,88){{\footnotesize$s_{2}+6$}}
   \put(31,103){{\footnotesize$s_{2}+7$}}            
   \put(53,7){$\bullet$}
   \put(53,22){$\bullet$}
   \put(53,37){$\bullet$}
   \put(53,52){$\bullet$}
   \put(53,67){$\bullet$}
   \put(53,82){$\bullet$}
   \put(53,97){$\bullet$}
   \put(56,10){\vector(1,0){20}}
   \put(56,25){\vector(1,0){20}}
   \put(54,13){{\footnotesize$s_{2}+4$}}
   \put(54,28){{\footnotesize$s_{2}+3$}}            
   \put(76,7){$\bullet$}
   \put(76,22){$\bullet$}
   \put(76,37){$\bullet$}
   \put(76,52){$\bullet$}
   \put(77,43){{\footnotesize$s_{2}+1$}}
   \put(77,58){{\footnotesize$s_{2}+2$}}      
  \put(100,37){$\bullet$}
  \put(100,52){$\bullet$}
  \put(100,40){\vector(-1,0){20}}
  \put(100,55){\vector(-1,0){20}}                  
  \end{picture}         
  \end{center}
    \caption{Lace diagrams corresponding to
  $s_{1}+\mathcal{F}^{(1,4)}$ and $s_{2}+
 \mathcal{F}^{(2,5)}$}
  \label{figure:SecondExample}
  \end{figure}
 First, as in Figure~\ref{figure:SecondExample}, we attach the
 linear forms in $s_{1}+\mathcal{F}^{(1,4)}$
 (resp.~$s_{2}+\mathcal{F}^{(2,5)}$) to the arrows in the lace
 diagram corresponding to $\mathcal{O}^{(1, 4)}$
 (resp.~$\mathcal{O}^{(2,5)}$). Here 
 the linear forms in each column are attached upside down 
 according as the arrow is
 leftward or rightward.
 See Figure~\ref{fig:ArrowsAndFactorsII} for the way of
 attaching the factors. 
 As before,
 we superpose two diagrams and if 
 two linear forms are
 attached to the same arrow, those two linear forms are also
 superposed as in
 Figure~\ref{fig:SuperpositionInSecondExample}. 
 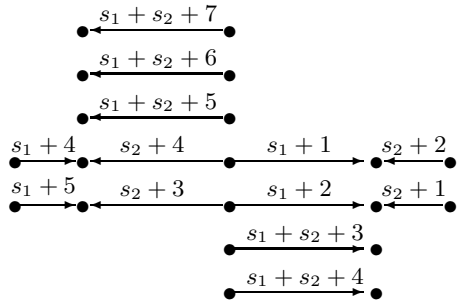
\begin{figure}[htbp]
 \begin{center}
  \begin{picture}(170,125)(0,15)
   \setlength{\unitlength}{1.1pt}
   \put(7,37){$\bullet$}
   \put(7,52){$\bullet$}
   \put(10,40){\vector(1,0){20}}
   \put(10,55){\vector(1,0){20}}
   \put(8,44){{\footnotesize$s_{1}+5$}}
   \put(8,58){{\footnotesize$s_{1}+4$}}   
   \put(30,37){$\bullet$}
   \put(30,52){$\bullet$}
   \put(30,67){$\bullet$}
   \put(30,82){$\bullet$}
   \put(30,97){$\bullet$}
   \put(80,40){\vector(-1,0){45}}            
   \put(80,55){\vector(-1,0){45}}         
   \put(80,70){\vector(-1,0){45}}      
   \put(80,85){\vector(-1,0){45}}   
   \put(80,100){\vector(-1,0){45}}
   \put(45,43){{\footnotesize$s_{2}+3$}}      
   \put(45,58){{\footnotesize$s_{2}+4$}}   
   \put(38,73){{\footnotesize$s_{1}+s_{2}+5$}}
   \put(38,88){{\footnotesize$s_{1}+s_{2}+6$}}
   \put(38,103){{\footnotesize$s_{1}+s_{2}+7$}}         
   \put(80,7){$\bullet$}
   \put(80,22){$\bullet$}
   \put(80,37){$\bullet$}
   \put(80,52){$\bullet$}
   \put(80,67){$\bullet$}
   \put(80,82){$\bullet$}
   \put(80,97){$\bullet$}
   \put(83,10){\vector(1,0){45}}
   \put(83,25){\vector(1,0){45}}
   \put(83,40){\vector(1,0){45}}
   \put(83,55){\vector(1,0){45}}
   \put(88,13){{\footnotesize$s_{1}+s_{2}+4$}}
   \put(88,28){{\footnotesize$s_{1}+s_{2}+3$}}
   \put(95,43){{\footnotesize$s_{1}+2$}}
   \put(95,58){{\footnotesize$s_{1}+1$}}            
   \put(130,7){$\bullet$}
   \put(130,22){$\bullet$}
   \put(130,37){$\bullet$}
   \put(130,52){$\bullet$}
   \put(134,43){{\footnotesize$s_{2}+1$}}
   \put(134,58){{\footnotesize$s_{2}+2$}}         
  \put(155,37){$\bullet$}
  \put(155,52){$\bullet$}
  \put(155,40){\vector(-1,0){20}}
  \put(155,55){\vector(-1,0){20}}                     
  \end{picture}
  \end{center}
 \caption{Superposition of the lace diagrams in
 Example~\ref{exmp:SecondExampleV}}
  \label{fig:SuperpositionInSecondExample}
 \end{figure}
\end{example}
In the following sections, we investigate these two examples in
more detail, while
the validity of the algorithm will be proved in general.

\section{Calculation of the $a$-function}
\label{section:Afunctions}

In this section, we discuss the method to calculate the $a$-functions.
In general, if we have the explicit form of the $a$-function
of a prehomogeneous vector space, we can determine the
$b$-function of the space to a certain extent, by using the
structure theorems on $a$-functions and $b$-functions of
several variables. We refer to M.\ Sato~\cite{EnAyumi} for the
details of the structure theorems. A convenient summary of
\cite{EnAyumi} is given in Ukai~\cite[\S\S 1.3]{Dynkin}.

We begin by recalling a general definition of $a$-functions. We keep
the notation of Section~\ref{section:preliminaries}.
Let $(G, \rho, V)$ be a reductive prehomogeneous vector space, 
$f$ a relative invariant, and
$\Omega(f) = V\setminus f^{-1}(0)$.
Then we define the map $\grad\log f:\Omega(f)\to V^{*}$ by
\begin{equation}
 \label{form:GradLogF}
 \grad\log f (v) = \sum_{i=1}^{n} \frac{1}{f(v)} \cdot
 \frac{\partial f}{\partial v_{i}}(v) \cdot e_{i}^{*}=
 \left(\frac{1}{f(v)}\frac{\partial f}{\partial v_{1}}(v),\dots,
  \frac{1}{f(v)}\frac{\partial f}{\partial v_{n}}(v)
 \right).
\end{equation}
Let $f_{1},\dots, f_{l}$ be the fundamental relative invariants
of $(G, \rho, V)$ and $f_{1}^{*},\dots, f_{l}^{*}$ the
fundamental relative invariants of $(G, \rho^{*}, V^{*})$
such that the characters of $f_{i}$ and $f_{i}^{*}$ are the inverse
of each other. We define $\underline{f}$, $\underline{f}^{*}$,
$\underline{f}^{\underline{s}}$, $\underline{f}^{*\underline{s}}$
in the same manner as in Section~\ref{section:preliminaries}.
\begin{lemma}
 \label{lemma:Afunctions}
 For any $l$-tuple $\underline{m} =(m_{1},\dots, m_{l})\in \Z_{\geq
 0}^{l}$, there exists a non-zero homogeneous polynomial
 $a_{\underline{m}}(\underline{s})$ of $s_{1},\dots, s_{l}$
 satisfying
 \[
  \underline{f}^{\underline{m}}(v)
 \underline{f}^{*\underline{m}}\left(\grad\log
 \underline{f}^{\underline{s}} (v)\right) = a_{\underline{m}}
 (\underline{s}).
 \]
 Here $a_{\underline{m}}(\underline{s})$ is independent of $v$.
\end{lemma}
We call $a_{\underline{m}}(\underline{s})$ the
{\it $a$-function} of $\underline{f}$. The calculation of
$a_{\underline{m}}(\underline{s})$ can be reduced to that of
$\grad\log \underline{f}^{\underline{s}}(v)$, and we have
\begin{equation}
 \label{form:SumOfGradLogF}
 \grad\log \underline{f}^{\underline{s}}(v) =
\sum_{i=1}^{l} s_{i} \cdot \grad\log f_{i}(v).
\end{equation}
Moreover, since
$a_{\underline{m}}(\underline{s})$ is independent of
$v$, it is enough to calculate 
$\grad\log f_{i}(v_{0})\;
(i=1,\dots,l)$ for a suitable generic point $v_{0}$. 

\begin{example}
 \label{exmp:GradLogFEquiorientedCase}
 Before proceeding to the general case, we consider the case
 where $Q$ is an equioriented quiver of type $A_{r}$ as below.
 \[
  Q: \vt{1}\longrightarrow \vt{2}\longrightarrow \vt{3}
 \longrightarrow \cdots \longrightarrow \vt{r}
 \]
 We keep the notation of Example~\ref{exmp:EquiorientedCase}.
 At first, we give explicitly a generic point $v_{0}$ of
 $(GL(\underline{n}), \Rep(Q,\underline{n}))$. For
 $m,n\in\Z_{>0}$, let
 \[
  E(m,n) =
 \left\{
 \begin{array}{ccl}
  \left(E_{m}\, |\, O_{m,n-m}\right) & & \text{if}  \quad m< n \\
  E_{m} & & \text{if}  \quad m = n \\
 {}^{t}\left(E_{n}\, |\, O_{n,m-n}\right) & & \text{if}  \quad m> n \\
 \end{array}.
 \right.
 \]
 Further, for $h\in \Z$ with $0\leq h\leq \min\{m,n\}$, let
 \[
  E(m,n ; h) =
 \left\{
 \begin{array}{ccl}
  \left(
   \begin{array}{c|c}
    E_{h}& O_{h,n-h} \\ \hline
    O_{m-h,h} & O_{m-h,n-h}
   \end{array}
   \right)
  & & \text{if}  \quad h< \min\{m,n\} \\[3pt]
  E(m,n) & & \text{if}  \quad h= \min\{m,n\} \\
 \end{array}.
 \right.
 \]
 Then a direct computation shows that
 \[
  A_{0} =
 \left(E(n_{2},n_{1}),\, E(n_{3},n_{2}),\dots,\,
 E(n_{r},n_{r-1})
 \right)\in \Rep(Q,\underline{n})
 \]
 is a generic point of $(GL(\underline{n}), \Rep(Q,\underline{n}))$.
 (see also Lemma~\ref{lemma:AbeasisGenericPoint} below.)
 Recall that we have assumed that
 $n_{1}=n_{r}< n_{2},\dots, n_{r-1}$, and then 
 \[
  f(v) = \det (X_{r,1}) = \det (X_{r,r-1}\cdots X_{3,2}X_{2,1})
 \]
 is a relative invariant of $(GL(\underline{n}),
 \Rep(Q,\underline{n}))$. We will show that
 \begin{equation}
  \label{form:GradLogFinEquioriented}
  \grad \log f(A_{0}) =
 \left(E(n_{2},n_{1}), \, E(n_{3},n_{2}; n_{1}),\dots,
 \, E(n_{r-1},n_{r-2};n_{1}), \, E(n_{r},n_{r-1})\right).
 \end{equation}
 In general, we have
 $E(m,n)\cdot E(n,l) = E(m,l\,;\,\min\{m,n,l\})$ and thus
 \begin{equation}
  \label{form:ValueOfFv0}
  f(A_{0})= \det \left(E_{n_{1}}\right) = 1.
 \end{equation}
 For $1\leq i\leq n_{k}$ and $1\leq j\leq n_{k-1}$, we denote by
 $x_{ij}^{(k,k-1)}$ the $(i,j)$-th component of $X_{k,k-1}$.
 In view of the definition~\eqref{form:GradLogF},
 it is enough to calculate
 \[
  \frac{\partial f}{\partial x_{ij}^{(k,k-1)}}(A_{0}) \qquad
 (k=2,\dots, r\,;\, i=1,\dots,n_{k}, j=1,\dots, n_{k-1}).
 \]
 For $1\leq s, t\leq n_{1}$, we denote by $y_{st}$ the
 $(s,t)$-th component of $X_{r,1} = X_{r,r-1}\cdots X_{2,1}$.
 Then it follows from the chain rule that
 \begin{equation}
  \label{form:ChainRule}
  \frac{\partial f}{\partial x_{ij}^{(k,k-1)}}(A_{0}) =
 \sum_{1\leq s, t\leq n_{1}} \frac{\partial f}{\partial y_{st}}(A_{0})
 \cdot \frac{\partial y_{st}}{\partial x_{ij}^{(k,k-1)}}(A_{0}).
 \end{equation}
 For a square matrix $R$, we denote by $\Delta_{st}(R)$ the
 $(s,t)$-cofactor of $R$. By definition of the determinant,
 we have
 ${\partial f}/{\partial y_{st}} = \Delta_{st}(X_{r,1})$
 and thus
 \[
  \frac{\partial f}{\partial y_{st}}(A_{0}) =
 \Delta_{st}(X_{r,1})\big|_{v=A_{0}} =
 \Delta_{st}\left(X_{r,1}\big|_{v=A_{0}}\right)=
 \Delta_{st}(E_{n_{1}}) = \delta_{st}
 \]
 where $\delta_{st}$ is the Kronecker-delta symbol.
 On the other hand, 
 \[
  y_{st} = \sum_{\alpha = 1}^{n_{k}}
 \sum_{\beta = 1}^{n_{k-1}} z_{s\alpha}^{(r,k)}\cdot 
 x_{\alpha\beta}^{(k,k-1)} \cdot z_{\beta t}^{(k-1,1)},
 \]
 where $z_{s\alpha}^{(r,k)}$ is the $(s,\alpha)$-component of
 $X_{r,k}=X_{r,r-1}\cdots X_{k+1,k}$ and $z_{\beta t}^{(k-1,1)}$ is
 the $(\beta, t)$-component of $X_{k-1,1} =X_{k-1,k-2}\cdots X_{2,1}$.
 Thus we have
 \[
  \frac{\partial y_{st}}{\partial x_{ij}^{(k,k-1)}} =
 z_{si}^{(r,k)} \cdot z_{jt}^{(k-1,1)}
 \]
 and hence
 \[
  \frac{\partial y_{st}}{\partial x_{ij}^{(k,k-1)}}(A_{0}) =
 z_{si}^{(r,k)}\big|_{v=A_{0}} \cdot z_{jt}^{(k-1,1)}\big|_{v=A_{0}}.
 \]
 Since $X_{rk}\big|_{v=A_{0}} = E(n_{1},n_{k})$ and
 $X_{k-1,1}\big|_{v=A_{0}} = E(n_{k-1},n_{1})$, we see that
 \[
   z_{si}^{(r,k)}\big|_{v=A_{0}} =
 \left\{
 \begin{array}{cl}
  \delta_{si}& (1\leq i\leq n_{1})\\
  0 &  (n_{1}+1\leq i\leq n_{k})
 \end{array},
 \right. \qquad
 z_{jt}^{(k-1,1)}\big|_{v=A_{0}}=
 \left\{
 \begin{array}{cl}
  \delta_{jt}& (1\leq j\leq n_{1})\\
  0 &  (n_{1}+1\leq i\leq n_{k-1})
 \end{array}.
 \right.
 \]
 As a result, we have
 \begin{align*}
  \frac{\partial f}{\partial x_{ij}^{(k,k-1)}}(A_{0}) &=
 \sum_{1\leq s, t\leq n_{1}} \delta_{st}
 \cdot \frac{\partial y_{st}}{\partial x_{ij}^{(k,k-1)}}(A_{0}) 
  = \sum_{s=1}^{n_{1}}
  \frac{\partial y_{ss}}{\partial x_{ij}^{(k,k-1)}}(A_{0}) \\
  &= \sum_{s=1}^{n_{1}}
   z_{si}^{(r,k)}\big|_{v=A_{0}} \cdot
   z_{js}^{(k-1,1)}\big|_{v=A_{0}} \\
  &=
  \left\{
  \begin{array}{ccl}
  \sum_{s=1}^{n_{1}}
   \delta_{si}\cdot \delta_{js}= \delta_{ij} & &
    (1\leq i, j\leq n_{1}) \\
   0 & & (\text{otherwise})\\    
  \end{array}.
  \right.
 \end{align*}
 Together with \eqref{form:ValueOfFv0}, this proves
 \eqref{form:GradLogFinEquioriented}.
\end{example}

In the case of arbitrary orientation,
we use lace diagrams to give
explicit descriptions of generic points of
$(GL(\underline{n}), \Rep(Q,\underline{n}))$.
For a given orientation $Q$ and a given dimension vector
$\underline{n}$, we consider
the lace diagram obtained by simply drawing all possible
horizontal lines between dots of consecutive columns.
We call it the {\it complete lace diagram} for $(GL(\underline{n}),
\Rep(Q,\underline{n}))$.

\begin{example}
 Let us consider $(GL(\underline{n}), \Rep(Q,\underline{n}))$
 discussed in Examples~\ref{exmp:FirstExampleIV}
 and~\ref{exmp:SecondExampleIV}.
 Then the complete lace diagrams are given as in 
 Figure~\ref{figure:CompleteLaceDiagrams}.
 \begin{figure}[htbp]
  \begin{center}
  \begin{picture}(100,90)(0,10)
   \put(-60,80){In the case of}
   \put(-60,65){Example~\ref{exmp:FirstExampleIV}}
   \put(7,7){$\bullet$}
   \put(12, 10){\vector(1,0){20}}
   \put(7,22){$\bullet$}
   \put(12, 25){\vector(1,0){20}}   
   \put(30,7){$\bullet$}
   \put(35,10){\vector(1,0){20}}      
   \put(30,22){$\bullet$}
   \put(35,25){\vector(1,0){20}}         
   \put(30,37){$\bullet$}
   \put(35,40){\vector(1,0){20}}            
   \put(30,52){$\bullet$}
   \put(35,55){\vector(1,0){20}}            
   \put(30,67){$\bullet$}
   \put(35,70){\vector(1,0){20}}            
   \put(53,7){$\bullet$}
   \put(53,22){$\bullet$}
   \put(53,37){$\bullet$}
   \put(53,52){$\bullet$}
   \put(53,67){$\bullet$}
   \put(53,82){$\bullet$}
   \put(58,10){\vector(1,0){20}}
   \put(58,25){\vector(1,0){20}}
   \put(58,40){\vector(1,0){20}}
   \put(58,55){\vector(1,0){20}}
   \put(58,70){\vector(1,0){20}}
   \put(58,85){\vector(1,0){20}}            
   \put(76,7){$\bullet$}
   \put(76,22){$\bullet$}
   \put(76,37){$\bullet$}
   \put(76,52){$\bullet$}
   \put(76,67){$\bullet$}
   \put(76,82){$\bullet$}
   \put(81,10){\vector(1,0){20}}
   \put(81,25){\vector(1,0){20}}   
   \put(100,7){$\bullet$}
   \put(100,22){$\bullet$}   
  \end{picture}
   \qquad\qquad\qquad\qquad
  \begin{picture}(100,105)(0,10)
   \put(-60, 90){In the case of }
   \put(-60,75){Example~\ref{exmp:SecondExampleIV}}
   \put(7,37){$\bullet$}
   \put(7,52){$\bullet$}
   \put(10,40){\vector(1,0){20}}
   \put(10,55){\vector(1,0){20}}   
   \put(30,37){$\bullet$}
   \put(30,52){$\bullet$}
   \put(30,67){$\bullet$}
   \put(30,82){$\bullet$}
   \put(30,97){$\bullet$}
   \put(53,40){\vector(-1,0){20}}            
   \put(53,55){\vector(-1,0){20}}         
   \put(53,70){\vector(-1,0){20}}      
   \put(53,85){\vector(-1,0){20}}   
   \put(53,100){\vector(-1,0){20}}
   \put(53,7){$\bullet$}
   \put(53,22){$\bullet$}
   \put(53,37){$\bullet$}
   \put(53,52){$\bullet$}
   \put(53,67){$\bullet$}
   \put(53,82){$\bullet$}
   \put(53,97){$\bullet$}
   \put(56,10){\vector(1,0){20}}
   \put(56,25){\vector(1,0){20}}
   \put(56,40){\vector(1,0){20}}
   \put(56,55){\vector(1,0){20}}            
   \put(76,7){$\bullet$}
   \put(76,22){$\bullet$}
   \put(76,37){$\bullet$}
   \put(76,52){$\bullet$}
  \put(100,37){$\bullet$}
  \put(100,52){$\bullet$}
  \put(100,40){\vector(-1,0){20}}
  \put(100,55){\vector(-1,0){20}}         
  \end{picture}
\end{center}
  \caption{Complete lace diagrams}
  \label{figure:CompleteLaceDiagrams}
\end{figure}
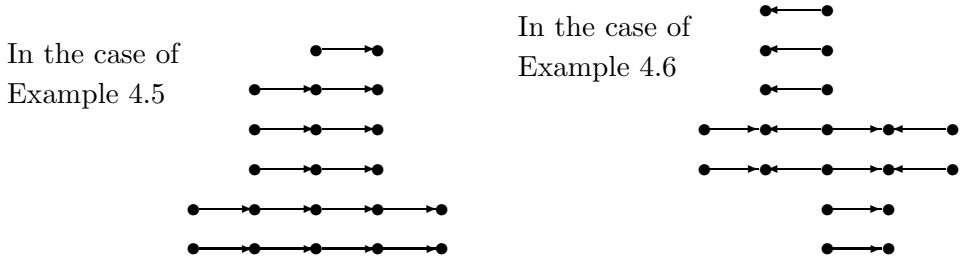
\end{example} 
Then we have the following lemma. 
\begin{lemma}[Abeasis~{\cite[Proposition~3.1]{Abeasis}}]
 \label{lemma:AbeasisGenericPoint}
 The element $A_{0}$ of $\Rep(Q,\underline{n})$
 represented by the complete lace diagram is
 a generic point of $(GL(\underline{n}), \Rep(Q,\underline{n}))$. 
\end{lemma}

%The following lemma, which will be proved in the rest of
%this section, plays a crucial role throughout our method
%to calculation of $b$-functions. 

The matrix representation of $A_{0}$ depends on the choice of
the basis of $\Rep(Q, \underline{n})$.
However, the lemma below says that 
without fixing the basis, it is possible to calculate
the value of $\grad\log f(A_{0})$ from the complete lace diagram.

\begin{lemma}
 \label{lemma:GradLogF}
 Let $A_{0}$ be a generic point of
 $(GL(\underline{n}), \Rep(Q,\underline{n}))$
 given by the complete lace diagram, and
 $f_{(p, q)}(v)$ the irreducible relative invariant given as in
 Section~\ref{section:RelativeInvariants}. 
 Then the value of $\grad\log f_{(p,q)}(A_{0})$ is given by
 the exact lace diagram corresponding to the locally closed orbits
 $\mathcal{O}^{(p,q)}$. 
\end{lemma}

\begin{remark}
 \label{remark:AdditionSuperposition}
   Do {\it not} interpret the lemma above as a result on
	the indecomposable decompositions of $\grad\log
	f_{(p,q)}(A_{0})$. If we give a matrix representation of
	$A_{0}$, then it automatically determines the
	matrix representation of 
	$\grad\log f_{(p,q)}(A_{0})$.
	This enables us to translate the addition in
	\eqref{form:SumOfGradLogF} as the superposition operation
	on lace diagrams.
 We also note that the fact $\grad\log f_{(p,q)}(A_{0})\in
 \mathcal{O}^{(p,q)}$ follows from a previously known result of 
 Gyoja~\cite[Theorem~1.18]{Nonreg}.
\end{remark}

\begin{proof}[Proof of Lemma~\ref{lemma:GradLogF}]
 Let
 \[
 Y_{(p,q)}(A_{0}):\bigoplus_{\tau} L_{\tau}
 \longrightarrow 
 \bigoplus_{\sigma} L_{\sigma} \qquad\;
 \left(\begin{array}{l}
   \text{$\tau$ runs over all the sources of $Q^{(p,q)}$}\\
   \text{$\sigma$ runs over all the sinks of $Q^{(p,q)}$}
 \end{array}\right) 
\]
 be the linear map defined in Section~\ref{section:RepThOfQuiver}.
 Since $f_{(p,q)}(A_{0}) = \det Y_{(p,q)}(A_{0})\neq 0$,
 $Y_{(p,q)}(A_{0})$
 is an isomorphism. Let $N := \sum_{\tau}\dim L_{\tau} =
 \sum_{\tau} \dim L_{\tau}$. Now we choose basis
 $\{\bs{u}_{1},\dots, \bs{u}_{N}\}$
 (resp. $\{\bs{u}_{1}',\dots, \bs{u}_{N}'\}$) of
 $\bigoplus_{\tau} L_{\tau}$ (resp.
 $\bigoplus_{\sigma} L_{\sigma}$). Each 
 $\bs{u}_{i} \; (i=1,\dots, N)$ (resp.  $\bs{u}_{i}'\; (i=1,\dots, N)$)
 can be identified with a dot at
 some source (resp. sink) of the lace diagram associated with
 $Q^{(p,q)}$. 
 By choosing a suitable order of the basis, we may assume that
 for $i=1,\dots, N$, 
 there exists a path from $\bs{u}_{i}$ to $\bs{u}_{i}'$ in the complete
 lace diagram. {\it Via this basis}, we identify $\bigoplus_{\tau}
 L_{\tau}$ and $\bigoplus_{\sigma} L_{\sigma}$ with $\C^{N}$.
 Note that this identification depends on $(p,q)$. 
 Our choice of the basis ensures that $Y_{(p,q)}(A_{0})$ is
 of the form
 \[
  Y_{(p,q)}(A_{0}) = E_{N} + \sum_{(s, t)\in \mathcal{B}} E_{st},
 \]
 where $E_{N}$ is the identity matrix of size $N$,
 $E_{st}$ is the $(s, t)$-matrix unit of size $N$, and
 $\mathcal{B}$ is an index set. Moreover, by looking at the shape of
 the complete lace diagram, we see that $\mathcal{B}$ has the
 following property:
 \begin{equation}
  \label{form:ConditionOnBeta}
  \textrm{
  any two element $(s, t), \, (s',t')\in \mathcal{B}$
  satisfy $s\neq t'$ and $t\neq s'$.}
 \end{equation}

 \begin{example}
  Let us explain the reason for \eqref{form:ConditionOnBeta}
  by examples. 
  In Example~\ref{exmp:FirstExampleIV}, $N=6$ for $(p, q)= (3, 4)$
  and $N=2$ for $(p,q)=(1, 5)$. If we choose the basis  
  $\{\bs{u}_{1},\dots, \bs{u}_{N}\}$
  and $\{\bs{u}_{1}',\dots, \bs{u}_{N}'\}$ as
  Figure~\ref{figure:BasisForFirstExample}, 
  then  we have
  \[
   Y_{(3,4)}(A_{0}) = E_{6}, \qquad Y_{(1,5)}(A_{0}) = E_{2}. 
  \]
  \begin{figure}[htbp]
   \begin{center}
   \begin{picture}(130,100)(0,10)
   \put(20,100){For $(p,q)=(3, 4)$}    
   \put(7,7){$\bullet$}
   \put(12, 10){\vector(1,0){20}}
   \put(7,22){$\bullet$}
   \put(12, 25){\vector(1,0){20}}   
   \put(30,7){$\bullet$}
   \put(35,10){\vector(1,0){20}}      
   \put(30,22){$\bullet$}
   \put(35,25){\vector(1,0){20}}         
   \put(30,37){$\bullet$}
   \put(35,40){\vector(1,0){20}}            
   \put(30,52){$\bullet$}
   \put(35,55){\vector(1,0){20}}            
   \put(30,67){$\bullet$}
   \put(35,70){\vector(1,0){20}}            
   \put(56,9){$\bs{u}_{6}$}
   \put(56,24){$\bs{u}_{5}$}
   \put(56,39){$\bs{u}_{4}$}
   \put(56,54){$\bs{u}_{3}$}
   \put(56,69){$\bs{u}_{2}$}
   \put(56,84){$\bs{u}_{1}$}
   \put(68,10){\vector(1,0){20}}
   \put(68,25){\vector(1,0){20}}
   \put(68,40){\vector(1,0){20}}
   \put(68,55){\vector(1,0){20}}
   \put(68,70){\vector(1,0){20}}
   \put(68,85){\vector(1,0){20}}            
   \put(89,9){$\bs{u}_{6}'$}
   \put(89,24){$\bs{u}_{5}'$}
   \put(89,39){$\bs{u}_{4}'$}
   \put(89,54){$\bs{u}_{3}'$}
   \put(89,69){$\bs{u}_{2}'$}
   \put(89,84){$\bs{u}_{1}'$}
   \put(103,10){\vector(1,0){20}}
   \put(103,25){\vector(1,0){20}}   
   \put(126,7){$\bullet$}
   \put(126,22){$\bullet$}   
  \end{picture}
   \qquad\qquad\qquad
   \begin{picture}(100,100)(0,10)
   \put(20,100){For $(p,q)=(1, 5)$}
   \put(0,9){$\bs{u}_{2}$}
   \put(12, 10){\vector(1,0){20}}
   \put(0,24){$\bs{u}_{1}$}
   \put(12, 25){\vector(1,0){20}}   
   \put(30,7){$\bullet$}
   \put(35,10){\vector(1,0){20}}      
   \put(30,22){$\bullet$}
   \put(35,25){\vector(1,0){20}}         
   \put(30,37){$\bullet$}
   \put(35,40){\vector(1,0){20}}            
   \put(30,52){$\bullet$}
   \put(35,55){\vector(1,0){20}}            
   \put(30,67){$\bullet$}
   \put(35,70){\vector(1,0){20}}            
   \put(53,7){$\bullet$}
   \put(53,22){$\bullet$}
   \put(53,37){$\bullet$}
   \put(53,52){$\bullet$}
   \put(53,67){$\bullet$}
   \put(53,82){$\bullet$}
   \put(58,10){\vector(1,0){20}}
   \put(58,25){\vector(1,0){20}}
   \put(58,40){\vector(1,0){20}}
   \put(58,55){\vector(1,0){20}}
   \put(58,70){\vector(1,0){20}}
   \put(58,85){\vector(1,0){20}}            
   \put(76,7){$\bullet$}
   \put(76,22){$\bullet$}
   \put(76,37){$\bullet$}
   \put(76,52){$\bullet$}
   \put(76,67){$\bullet$}
   \put(76,82){$\bullet$}
   \put(81,10){\vector(1,0){20}}
   \put(81,25){\vector(1,0){20}}   
   \put(102,9){$\bs{u}_{2}'$}
   \put(102,24){$\bs{u}_{1}'$}   
  \end{picture}    
  \end{center}
    \caption{Basis $\{\bs{u}_{i}\}$ and $\{\bs{u}_{i}'\}$
   for the case of Example~\ref{exmp:FirstExampleIV}}
  \label{figure:BasisForFirstExample}
 \end{figure}     

 In Example~\ref{exmp:SecondExampleIV}, $N=9$ for
 $(p, q)= (1, 4), (2, 5)$.
 If we choose the basis  
 $\{\bs{u}_{1},\dots, \bs{u}_{N}\}$
 and $\{\bs{u}_{1}',\dots, \bs{u}_{N}'\}$ as
  Figure~\ref{figure:BasisForSecondExample}    , then  we have
 \[
   Y_{(1,4)}(A_{0}) = E_{9}+E_{16}+E_{27},
  \qquad Y_{(2,5)}(A_{0}) = E_{9}+E_{84}+E_{95}. 
 \]
 \begin{figure}[htbp]
 \begin{center}
  \begin{picture}(130,115)(0,10)
   \put(10,115){For $(p,q)= (1, 4)$}
   \put(-2,39){$\bs{u}_{2}$}
   \put(-2,54){$\bs{u}_{1}$}
   \put(10,40){\vector(1,0){20}}
   \put(10,55){\vector(1,0){20}}   
   \put(31,39){$\bs{u}_{2}'$}
   \put(31,54){$\bs{u}_{1}'$}
   \put(31,69){$\bs{u}_{5}'$}
   \put(31,84){$\bs{u}_{4}'$}
   \put(31,99){$\bs{u}_{3}'$}
   \put(63,40){\vector(-1,0){20}}            
   \put(63,55){\vector(-1,0){20}}         
   \put(63,70){\vector(-1,0){20}}      
   \put(63,85){\vector(-1,0){20}}   
   \put(63,100){\vector(-1,0){20}}
   \put(64,9){$\bs{u}_{9}$}
   \put(64,24){$\bs{u}_{8}$}
   \put(64,39){$\bs{u}_{7}$}
   \put(64,54){$\bs{u}_{6}$}
   \put(64,69){$\bs{u}_{5}$}
   \put(64,84){$\bs{u}_{4}$}
   \put(64,99){$\bs{u}_{3}$}
   \put(77,10){\vector(1,0){20}}
   \put(77,25){\vector(1,0){20}}
   \put(77,40){\vector(1,0){20}}
   \put(77,55){\vector(1,0){20}}            
   \put(98,9){$\bs{u}_{9}'$}
   \put(98,24){$\bs{u}_{8}'$}
   \put(98,39){$\bs{u}_{7}'$}
   \put(98,54){$\bs{u}_{6}'$}
  \put(130,37){$\bullet$}
  \put(130,52){$\bullet$}
  \put(130,40){\vector(-1,0){20}}
  \put(130,55){\vector(-1,0){20}}         
  \end{picture}
  \qquad\qquad
  \begin{picture}(130,115)(0,10)
   \put(10,115){For $(p,q)= (2, 5)$}
   \put(5,37){$\bullet$}
   \put(5,52){$\bullet$}
   \put(10,40){\vector(1,0){20}}
   \put(10,55){\vector(1,0){20}}   
   \put(31,39){$\bs{u}_{5}'$}
   \put(31,54){$\bs{u}_{4}'$}
   \put(31,69){$\bs{u}_{3}'$}
   \put(31,84){$\bs{u}_{2}'$}
   \put(31,99){$\bs{u}_{1}'$}
   \put(63,40){\vector(-1,0){20}}            
   \put(63,55){\vector(-1,0){20}}         
   \put(63,70){\vector(-1,0){20}}      
   \put(63,85){\vector(-1,0){20}}   
   \put(63,100){\vector(-1,0){20}}
   \put(64,9){$\bs{u}_{7}$}
   \put(64,24){$\bs{u}_{6}$}
   \put(64,39){$\bs{u}_{5}$}
   \put(64,54){$\bs{u}_{4}$}
   \put(64,69){$\bs{u}_{3}$}
   \put(64,84){$\bs{u}_{2}$}
   \put(64,99){$\bs{u}_{1}$}
   \put(77,10){\vector(1,0){20}}
   \put(77,25){\vector(1,0){20}}
   \put(77,40){\vector(1,0){20}}
   \put(77,55){\vector(1,0){20}}            
   \put(98,9){$\bs{u}_{7}'$}
   \put(98,24){$\bs{u}_{6}'$}
   \put(98,39){$\bs{u}_{9}'$}
   \put(98,54){$\bs{u}_{8}'$}
  \put(130,39){$\bs{u}_{9}$}
  \put(130,54){$\bs{u}_{8}$}
  \put(130,40){\vector(-1,0){20}}
  \put(130,55){\vector(-1,0){20}}         
  \end{picture}
 \end{center}
 \caption{Basis $\{\bs{u}_{i}\}$ and $\{\bs{u}_{i}'\}$
 for the case of Example~\ref{exmp:SecondExampleIV}}
  \label{figure:BasisForSecondExample}    
 \end{figure}   
 \end{example}

 The property \eqref{form:ConditionOnBeta}
 yields the following three formulas.
 \begin{align}
  \label{form:Ypq1}
  Y_{(p,q)}(A_{0})&=
  \prod_{(s, t)\in \mathcal{B}} \left(E_{N} + E_{st}
  \right), \\[5pt]
    \label{form:Ypq2}
  f_{(p,q)}(A_{0}) &=  1, \\[5pt]
    \label{form:Ypq3}
  {}^{t} \left( Y_{(p,q)}(A_{0})\right)^{-1} &=
  E_{N} - \sum_{(t, s)\in \mathcal{B}} E_{st}.
 \end{align}
 By using these formulas, 
 we generalize the calculation in
 Example~\ref{exmp:GradLogFEquiorientedCase} to the case of
 arbitrary orientation. For $a\in Q_{1} = \{1,\dots, r-1\}$,
 let $e_{i}^{(h(a))} \; (i=1,\dots, n_{h(a)})$ be the basis of
 $L_{h(a)}$ and $e_{j}^{(t(a))} \; (j=1,\dots, n_{t(a)})$
 the basis of $L_{t(a)}$. We take the coordinate system
 with respect to these basis and denote by
 $x_{ij}^{(a)}$ the $(i, j)$-th component of $X_{h(a), t(a)}$.
 Moreover, for $1\leq s, t\leq N$, we denote by
 $y_{st}$ the $(s, t)$-th component of $Y_{(p, q)}(v)$.
 Then the formula~\eqref{form:ChainRule} can be generalized as
 follows.
 \[
  \frac{\partial f_{(p,q)}}{\partial x_{ij}^{(a)}}(A_{0}) =
 \sum_{1\leq s, t\leq N} \frac{\partial f_{(p,q)}}{\partial
 y_{st}} (A_{0}) \cdot 
 \frac{\partial y_{st}}{\partial x_{ij}^{(a)}}(A_{0}).
 \]
 For a matrix $R$, let $\Delta(R)$ be the cofactor matrix. 
 Then \eqref{form:Ypq2} and \eqref{form:Ypq3} imply that
 \begin{align*}
  \left(\frac{\partial f_{(p,q)}}{\partial
 y_{st}} (A_{0})
 \right)_{1\leq s, t\leq N} &= {}^{t}\Delta
 \left(Y_{(p,q)}(A_{0})
 \right) = f_{(p,q)}(A_{0})\cdot {}^{t}
 \left(Y_{(p,q)}(A_{0})
 \right)^{-1} \\
  &=  E_{N} - \sum_{(t, s)\in \mathcal{B}} E_{st},
 \end{align*}
 and thus we see that 
 \[
  \frac{\partial f_{(p,q)}}{\partial
 y_{st}} (A_{0}) =
 \left\{
 \begin{array}{ccl}
  1 & & s=t  \\
  -1 & & (t, s)\in \mathcal{B} \\
  0 & & \textrm{otherwise}
 \end{array}.
 \right.
 \]
 However, if $(t, s)\in \mathcal{B}$,
 then $y_{st} = 0$; the reader is referred to
 \eqref{form:ConditionOnBeta}. Hence
 we may assume that $s=t$, and 
 it is enough to calculate
 $\dfrac{\partial y_{ss}}{\partial x_{ij}^{(a)}}(A_{0})$.
 By generalizing the calculation of
 $\dfrac{\partial y_{st}}{\partial x_{ij}^{(k,k-1)}}(A_{0})$
 in Example~\ref{exmp:GradLogFEquiorientedCase},
 we see that
 \[
  \frac{\partial y_{ss}}{\partial x_{ij}^{(a)}}(A_{0}) =
 \left\{
 \begin{array}{ccl}
  1 & &
   \left(\begin{array}{l}
   \textrm{In the complete lace diagram,} \\
   \textrm{there exists an arrow} \;\;
   e_{j}^{(t(a))}\rightarrow e_{i}^{(h(a))} \\
   \textrm{on the path connecting}\;\;
    \bs{u}_{s} \;\; \textrm{and}\;\; \bs{u}_{s}'
   \end{array}\right)
   \\[15pt]
  0 & & \;\; (\textrm{otherwise})
 \end{array}.
 \right. 
 \]
 We therefore conclude that the diagram representing the
 {\it value} of $\grad\log f_{(p,q)}(A_{0})$ consists only of
 the paths connecting $\bs{u}_{s}$ and $\bs{u}_{s}' \;
 (s= 1,\dots, N)$. This diagram is nothing but the exact lace
 diagram corresponding to $\mathcal{O}^{(p,q)}$. 
 This completes the proof of the lemma.
\end{proof}

As we have mentioned in Remark~\ref{remark:AdditionSuperposition},
the summation
\[
 \grad\log \underline{f}^{\underline{s}} (A_{0}) =
 \sum_{i=1}^{l} s_{i} \cdot \grad\log f_{i}(A_{0})
\]
can be transfered to the superposition of the lace diagrams, and
this ensures the validity of the ``superposition method'' for the
$a$-functions.

\begin{example}
 Let us consider the case of Example~\ref{exmp:FirstExampleIV}.
 By Lemma~\ref{lemma:GradLogF}, 
 $\grad\log \underline{f}^{\underline{s}}(A_{0}) =
 s_{1}\grad\log f_{1}(A_{0}) + s_{2} \grad\log f_{2}(A_{0})$
 is given as in
 Figure~\ref{figure:GradLogFInFirstExample}. 
 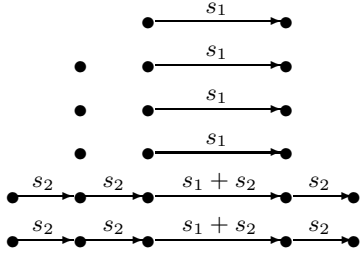
\begin{figure}[htbp] 
\begin{center}
   \begin{picture}(100,90)(0,15)
    \setlength{\unitlength}{1.1pt}
%   \put(-20,80){$s_{2}+\mathcal{F}^{(1,5)}$}   
   \put(7,7){$\bullet$}
   \put(7,22){$\bullet$}
   \put(10,10){\vector(1,0){20}}
   \put(10,25){\vector(1,0){20}}
   \put(16,28){{\footnotesize$s_{2}$}}
   \put(16,13){{\footnotesize$s_{2}$}}
   \put(30,7){$\bullet$}
   \put(30,22){$\bullet$}
   \put(30,37){$\bullet$}
   \put(30,52){$\bullet$}
   \put(30,67){$\bullet$}
   \put(34,10){\vector(1,0){20}}
   \put(34,25){\vector(1,0){20}}
   \put(40,28){{\footnotesize$s_{2}$}}
   \put(40,13){{\footnotesize$s_{2}$}}   
   \put(53,7){$\bullet$}
   \put(53,22){$\bullet$}
   \put(53,37){$\bullet$}
   \put(53,52){$\bullet$}
   \put(53,67){$\bullet$}
   \put(53,82){$\bullet$}
   \put(58,10){\vector(1,0){43}}
   \put(58,25){\vector(1,0){43}}
   \put(58,40){\vector(1,0){43}}
   \put(58,55){\vector(1,0){43}}
   \put(58,70){\vector(1,0){43}}
   \put(58,85){\vector(1,0){43}}
   \put(75,88){{\footnotesize$s_{1}$}}        
   \put(75,73){{\footnotesize$s_{1}$}}    
   \put(75,58){{\footnotesize$s_{1}$}}            
   \put(75,43){{\footnotesize$s_{1}$}}        
   \put(67,28){{\footnotesize$s_{1}+s_{2}$}}
   \put(67,13){{\footnotesize$s_{1}+s_{2}$}}
   \put(100,7){$\bullet$}
   \put(100,22){$\bullet$}
   \put(100,37){$\bullet$}
   \put(100,52){$\bullet$}
   \put(100,67){$\bullet$}
   \put(100,82){$\bullet$}
   \put(105,10){\vector(1,0){20}}
   \put(105,25){\vector(1,0){20}}
   \put(110,28){{\footnotesize$s_{2}$}}
   \put(110,13){{\footnotesize$s_{2}$}}   
   \put(123,7){$\bullet$}
   \put(123,22){$\bullet$}   
  \end{picture}      
\end{center}
 \caption{Diagram expressing the value of 
  $\grad\log \underline{f}^{\underline{s}}(A_{0})$ for 
 Example~\ref{exmp:FirstExampleIV}}
  \label{figure:GradLogFInFirstExample}
 \end{figure}
 If we fix a basis of $\Rep(Q, \underline{n})$, then 
 the arrows $\overset{s_{1}}{\longrightarrow}, \,
 \overset{s_{2}}{\longrightarrow}, \,
 \overset{s_{1}+s_{2}}{\longrightarrow}$ 
 correspond to $s_{1} E_{ij}, \,
 s_{2} E_{i', j'}, \, (s_{1}+s_{2}) E_{i'', j''}$
 with some matrix units $E_{ij},
 E_{i', j'}, E_{i'',j''}$. Then, by the definition of $f_1$
 (recall that $f_{1}$ is constructed as $f_{(3, 4)}$ in 
 Example~\ref{exmp:FirstExampleII}),
 we have
 \[
  a_{(1,0)}(\underline{s}) = f_{1}(A_{0}) f_{1}
 (\grad\log \underline{f}^{\underline{s}} (A_{0})) =
 1\cdot s_{1}^4 (s_{1}+s_{2})^2 = s_{1}^4 (s_{1}+s_{2})^2.
 \]
 Similarly, we have
  \[
  a_{(0,1)}(\underline{s}) = f_{2}(A_{0}) f_{2}
 (\grad\log \underline{f}^{\underline{s}} (A_{0})) =
 1\cdot s_{2}^6 (s_{1}+s_{2})^2 = s_{2}^6 (s_{1}+s_{2})^2,
 \]
 for $f_{2}$ is constructed as $f_{(1, 5)}$ in
 Example~\ref{exmp:FirstExampleII}. Hence we have
 \[
  a_{\underline{m}}(\underline{s}) = a_{(1,0)}(\underline{s})^{m_{1}}
 a_{(0,1)}(\underline{s})^{m_{2}} =
 s_{1}^{4m_{1}} s_{2}^{6m_{2}} (s_{1}+s_{2})^{2(m_{1}+m_{2})}.
 \]
 In general, the $a$-functions determine the $b$-functions except the
 ``constant terms''. In this case, we have
 \begin{align}
  \label{form:ByAfunInFirstExample}
  b_{\underline{m}}(\underline{s})&= \prod_{r=1}^{4}
  [s_{1}+\alpha_{1,r}]_{m_{1}}\times \prod_{r=1}^{6}
  [s_{2}+\alpha_{2,r}]_{m_{2}}\times \prod_{r=1}^{2}
  [s_{1}+s_{2}+\alpha_{3,r}]_{m_{1}+m_{2}}
 \end{align}
 with some $\alpha_{j,r}\in \Q_{>0}$. Here we have employed the
 structure theorem on $b$-functions (see
 M.~Sato~\cite[Theorem~2]{EnAyumi}) and Kashiwara's
 theorem~\cite[Theorem~6.9]{KashiwaraBook}; see also
 Ukai~\cite[Theorem~1.3.5]{Dynkin}.   

 Similarly, in the case of Example~\ref{exmp:SecondExampleIV},
 we see that 
 \[
  a_{\underline{m}}(\underline{s}) = s_{1}^{4m_{1}}
 s_{2}^{4m_{2}} (s_{1}+s_{2})^{5(m_{1}+m_{2})}, 
 \]
 and 
 \begin{align}
  \label{form:ByAfunInSecondExample}
  b_{\underline{m}}(\underline{s})&= \prod_{r=1}^{4}
  [s_{1}+\alpha_{1,r}]_{m_{1}}\times \prod_{r=1}^{4}
  [s_{2}+\alpha_{2,r}]_{m_{2}}\times \prod_{r=1}^{5}
  [s_{1}+s_{2}+\alpha_{3,r}]_{m_{1}+m_{2}}
 \end{align}
 with some $\alpha_{j, r}\in \Q_{>0}$. 
 In any case, the calculation of
 $b_{\underline{m}}(\underline{s})$ is reduced to the
 determination of $\alpha_{j,r}$.  
\end{example}

\section{Proof of Theorem~\ref{theorem:MainTheorem}}
\label{section:Proof}

The determination of  $\alpha_{j, r}$ can be carried out by
localization of $b$-functions (cf. Ukai~\cite[p.~57]{Dynkin},
\cite{SugiGeolocalB}).
We briefly recall the localization of $b$-functions, to
the extent necessary for the calculation in the present paper. 
(the idea of \cite{Dynkin} can be applied to a more general
setting.)
Let $(G, V)$ be a reductive prehomogeneous vector space and
$f\in \C[V]$ a relatively invariant polynomial. Take an arbitrary 
point $v_{1}\in V$. Let $\mathcal{O}_{1} = G\cdot v_{1}$ be the
$G$-orbit through $v_{1}$, and $G_{v_{1}}$ be the
isotropy subgroup of
$G$ at $v_{1}$. Then $G_{v_{1}}$ acts on the tangent space
$T_{v_{1}}(\mathcal{O}_{1}) = \mathfrak{g}\cdot v_{1}$.
Now we {\it assume} that $G_{v_{1}}$ is reductive. 
Then there exists a $G_{v_{1}}$-invariant subspace
$W$ of $V$ such that $V = T_{v_{1}}(\mathcal{O}_{1})\oplus W$.
We call $(G_{v_{1}}, W)$ the {\it slice representation} at
$v_{1}$ (cf. Kac~\cite{KacSlice}, Shmelkin~\cite{Shmelkin}). 
It is easy to see that $(G_{v_{1}}^{\circ}, W)$ is a
reductive prehomogeneous vector space, where $G_{v_{1}}^{\circ}$
is the connected component of $G_{v_{1}}$. 
Moreover,
the function
$f_{v_{1}}$ on $W$ defined by $f_{v_{1}}(w) = f(v_{1}+w) \; (w\in W)$
is a relatively invariant polynomial of $(G_{v_{1}}^{\circ}, W)$.
We denote by $b_{f, \mathcal{O}_{1}}(s)$ the $b$-function of
$f_{v_{1}}$; in effect, if $v_{1}' \in \mathcal{O}_{1}$, then
the $b$-function of $f_{v_{1}'}$ coincides with that of $f_{v_{1}}$. 
Then we have the following lemma.
\begin{lemma}
 \label{lemma:Localization}
 $b_{f, \mathcal{O}_{1}}(s)$ divides $b_{f}(s)$, the $b$-function of $f$. 
\end{lemma}

To apply Lemma~\ref{lemma:Localization}
to our prehomogeneous vector spaces
$(GL(\underline{n}), \Rep(Q,\underline{n}))$,
we recall some facts from representation theory of quivers.
A useful reference for this section is a survey by Brion~\cite{Brion}.
Now, for a moment, we assume that $Q$ is an arbitrary quiver which 
has no oriented cycles. 
For two dimension vectors $\underline{n} = (n_{i})_{i\in Q_{0}},
\underline{m} = (m_{i})_{i\in Q_{0}}$, we consider quiver
representation spaces
\[
 \Rep(Q, \underline{n}) = \bigoplus_{a\in Q_{1}}
 \Hom(L_{t(a)}, L_{h(a)}), \qquad
 \Rep(Q, \underline{m}) = \bigoplus_{a\in Q_{1}}
 \Hom(L_{t(a)}', L_{h(a)}'), 
\]
where $L_{i}$ (resp.~$L_{i}'$) is a vector space of dimension
$n_{i}$ (resp.~$m_{i}$). Take
$A= (A_{h(a), t(a)})_{a\in Q_{1}}\in \Rep(Q, \underline{n})$ and
$B= (B_{h(a), t(a)})_{a\in Q_{1}}\in \Rep(Q, \underline{m})$.
Then we define a map
\begin{equation}
 \label{form:DefOfDAB}
 d_{A, B} : \bigoplus_{i\in Q_{0}} \Hom(L_{i}, L_{i}') \longrightarrow
 \bigoplus_{a\in Q_{1}} \Hom(L_{t(a)}, L_{h(a)}')
\end{equation}
by
\[
 d_{A, B} (\phi) =
 \left(\phi_{h(a)} A_{h(a), t(a)} - B_{h(a), t(a)} \phi_{t(a)}
 \right)_{a\in Q_{1}}
\]
for $\phi = (\phi_{i})_{i\in Q_{0}} \in
\bigoplus_{i\in Q_{0}} \Hom(L_{i}, L_{i}')$. It is known 
(see Ringel~\cite{RingelSpecies},  or
 Brion~\cite[Corollary 1.4.2]{Brion}) that
$\Hom_{Q}(A, B)$ and $\Ext_{Q}(A, B)$, which are defined through
the Ringel resolution, are equal to $\Ker d_{A, B}$ and
$\Coker d_{A, B}$, respectively. That is, 
\[
 \Hom_{Q}(A, B) = \Ker d_{A, B}, \qquad
 \Ext_{Q}(A, B) = \Coker d_{A, B}.  
\]
Then we have an exact sequence
\begin{align*}
 0 &\longrightarrow \Hom_{Q}(A, B) \longrightarrow
 \bigoplus_{i\in Q_{0}} \Hom_{Q}(L_{i}, L_{i}')
 \overset{d_{A,B}}{\longrightarrow}
 \bigoplus_{a\in Q_{1}} \Hom(L_{t(a)}, L_{h(a)}') \\
 &\longrightarrow \Ext_{Q}(A, B) \longrightarrow 0.
\end{align*}
Taking dimensions in the exact sequence above yields
\[
 \dim \Hom_{Q}(A, B) - \dim \Ext_{Q}(A, B) =
 \sum_{i\in Q_{0}} n_{i} m_{i} - \sum_{a\in Q_{1}}
 n_{t(a)} m_{h(a)}.
\]
The Euler form of a quiver $Q$ is the
bilinear form $\left\langle\;\, , \; \right\rangle_{Q}$ on
$\R^{Q_{0}}$ defined by
\begin{equation}
 \label{form:EulerForm}
 \left\langle\underline{n}, \underline{m}\right\rangle_{Q} =
 \sum_{i\in Q_{0}} n_{i} m_{i} - \sum_{a\in Q_{1}} n_{t(a)}
 m_{h(a)}
\end{equation}
for any $\underline{n} = (n_{i})_{i\in Q_{0}}$ and
$\underline{m} = (m_{i})_{i\in Q_{0}}$.
When $A$ (resp.~$B$) is an element of $\Rep(Q, \underline{n})$
(resp.~$\Rep(Q, \underline{m})$), we often write
$\left\langle A, B\right\rangle_{Q}$ for
$\left\langle \underline{n}, \underline{m}\right\rangle_{Q}$.
We thus obtain Ringel's formula
(cf.\ \cite{RingelSpecies})
\begin{equation}
 \label{form:Ringel}
 \left\langle A, B\right\rangle_{Q} = \dim \Hom_{Q}(A, B) - \dim
 \Ext_{Q}(A, B), 
\end{equation}
which will be used later. 
Next let $\underline{n} =\underline{m}$ and $A=B$. We identify
$\bigoplus_{i\in Q_{0}} \Hom(L_{i}, L_{i})$ with the
Lie algebra
$\mathfrak{gl}(\underline{n}) = \bigoplus_{i\in Q_{0}} 
\mathfrak{gl}(n_{i})$ of $GL(\underline{n})$.
Then the map $d_{A, A}$ can be identified with
the differential map at the identity of the orbit map
\[
 GL(\underline{n}) \ni g \longmapsto g\cdot A \in \Rep(Q,
 \underline{n}). 
\]
By definition, $\Hom_{Q}(A, A)$ is isomorphic to the isotropy subalgebra
$\mathfrak{gl}(\underline{n})_{A}$ of $\mathfrak{gl}(\underline{n})$
at $A$, and $\Ext_{Q}(A, A)$ is isomorphic
to the normal space in $\Rep(Q, \underline{n})$ to the orbit
$GL(\underline{n})\cdot A$ at $A$. More precisely, we have
\[
 \Ext_{Q}(A,A) \cong \Rep(Q, \underline{n})/ T_{A}
 \left(GL(\underline{n})\cdot A\right). 
\]

These notions in representation theory of quivers 
are linked to the slice
representations in the following way:
Let $\mathcal{O}_{1}$ be the closed orbit in
$\{A\in \Rep(Q,\underline{n})\, ;\, f(A)\neq 0\}$, where
$f$ is a relative invariant of $(GL(\underline{n}), \Rep(Q,
\underline{n}))$. Take an element $A_{1}$ of $\mathcal{O}_{1}$;
this is a {\it locally semi-simple representation} of $Q$ in the sense
of Shmelkin~\cite{Shmelkin}. 
Then, Matsushima's theorem implies that 
the isotropy subgroup $GL(\underline{n})_{A_{1}}$ is reductive,
and thus there exists a $GL(\underline{n})_{A_{1}}$-invariant
subspace $W_{1}$ satisfying
%\begin{equation}
% \label{form:DefOfWpq}
\[
 \Rep(Q,\underline{n}) = T_{A_{1}}(\mathcal{O}_{1})\oplus
 W_{1}.
\]
%\end{equation}
Let $\mathfrak{gl}(\underline{n})_{A_{1}}$ be 
the isotropy subalgebra of $\mathfrak{gl}(\underline{n})$ at
$A_{1}$. The argument above implies that 
\begin{align*}
 \mathfrak{gl}(\underline{n})_{A_{1}}
 &\cong \Hom_{Q}\left(A_{1}, A_{1}\right), \\
 W_{1} &\cong \Ext_{Q}\left(A_{1}, A_{1} \right).
\end{align*}
Shmelkin~\cite{Shmelkin} showed that the structure of
$(GL(\underline{n})_{A_{1}}, W_{1})$ can be described by using the
{\it local quiver}. Let $A_{1}= \bigoplus_{i=1}^{t} m_{i}^{(1)}
I_{i}^{(1)}$ be the indecomposable decomposition of $A_{1}$,
where $I_{i}^{(1)}$ is an indecomposable representation of $Q$ and
$m_{i}^{(1)}$ is the multiplicity of $I_{i}^{(1)}$ in $A_{1}$. 
By Shmelkin~\cite[Proposition~8]{Shmelkin}, we have
\begin{equation}
 \label{form:HomInLocallySemisimple}
 \dim \Hom_{Q}(I_{i}^{(1)}, I_{j}^{(1)}) = \delta_{ij}, 
\end{equation}
and
by Ringel's formula~\eqref{form:Ringel}, we have
\[
 \delta_{ij} - \langle I_{i}^{(1)}, I_{j}^{(1)}\rangle =
\dim \Ext_{Q}(I_{i}^{(1)}, I_{j}^{(1)})\geq 0.
\]
The local quiver $\Sigma$ is
a quiver with vertices $a_{1},\dots, a_{t}$ corresponding to the
summands $I_{1}^{(1)}, \dots, I_{t}^{(1)}$, and
$\delta_{ij} - \langle I_{i}^{(1)}, I_{j}^{(1)}\rangle$
arrows from $a_{i}$ to $a_{j}$. We set
$\underline{\gamma} = (m_{1}^{(1)}, \dots, m_{t}^{(1)})$.
Then we have the following formula for the slice representation
$(GL(\underline{n})_{A_{1}}, W_{1})$.

\begin{lemma}[Shmelkin~{\cite[formula~(9)]{Shmelkin}}]
 \label{lemma:SchmelkinStrTh}
\begin{align*}
 (GL(\underline{n})_{A_{1}}, W_{1})&\cong
 (GL(\underline{\gamma}), \Rep(\Sigma, \underline{\gamma})) \\
 &\cong
 \left(
 \prod_{i=1}^{t} GL(m_{i}^{(1)}), \;
 \bigoplus_{1\leq i,j\leq t}
 \Ext_{Q}(I_{i}^{(1)}, I_{j}^{(1)})\otimes
 M(m_{j}^{(1)}, m_{i}^{(1)})
 \right).
\end{align*} 
\end{lemma}

Next we recall Schofield's determinantal invariants, since
we need to consider the restrictions of relative invariants to
slice spaces. Take $A\in \Rep(Q, \underline{n})$ and
$B\in \Rep(Q, \underline{m})$ such that
$\left\langle A, B\right\rangle_{Q}=0$.
Then the matrix representing the map $d_{A, B}$ in
\eqref{form:DefOfDAB} is a square matrix, and
\[
 c(A, B):= \det d_{A, B}
\]
is a relative invariant of $(GL(\underline{n})\times
GL(\underline{m}), \, \Rep(Q, \underline{n})\oplus
\Rep(Q, \underline{m}))$. It is easy to see that
$c(A,B)\neq 0$ if and only if $\Hom_{Q}(A, B)=0$, which
is equivalent to $\Ext_{Q}(A, B)=0$. For a
fixed $B$, the restriction of $c$ to
$\Rep(Q, \underline{n})\oplus \{B\}$ gives a relative
invariant $c_{B}$ of $(GL(\underline{n}), \Rep(Q, \underline{n}))$.
Similarly, for a fixed $A$, the restriction of $c$ to
$\{A\}\oplus \Rep(Q, \underline{m})$ gives a relative invariant
$c^{A}$ of $(GL(\underline{m}), \Rep(Q, \underline{m}))$. 
Moreover, we define the {\it right} perpendicular category
$A^{\perp} $ of $A$ to be the full subcategory of representations
$B$ such that $\Hom_{Q}(A, B)=0= \Ext_{Q}(A,B)$. Similarly,
we define the {\it left} perpendicular category ${}^{\perp}A$
as the full subcategory of representations $C$ such that
$\Hom_{Q}(C, A)=0= \Ext_{Q}(C,A)$.
Then we have the following lemma due to
Schofield~\cite{Schofield}.  

\begin{lemma}[\cite{Schofield}]
 Let $Q$ be a quiver with $r$ vertices,  without oriented cycles.
 Assume that $(GL(\underline{n}), \Rep(Q, \underline{n}))$ is
 a prehomogeneous vector space and take a generic point $A_{0}=
 \bigoplus_{i=1}^{p} m_{i}I_{i}$. Here $I_{i}$ is an
 indecomposable representation of $Q$ and $m_{i}$ is the
 multiplicity of $I_{i}$ in $A_{0}$. Then
 $A_{0}^{\perp}$ and ${}^{\perp} A_{0}$ are equivalent to
 categories of representations of $Q^{\perp}$ and ${}^{\perp} Q$,
 respectively, where $Q^{\perp}$ and ${}^{\perp}Q$ are quivers
 with $l:=r-p$ vertices, without oriented cycles.
\end{lemma}

The above lemma implies that both $A_{0}^{\perp}$ and
${}^{\perp}A_{0}$ contain exactly $l=r-p$ simple objects.
Let $T_{1},\dots, T_{l}$ 
be the simple objects in $A_{0}^{\perp}$.

\begin{lemma}[\cite{Schofield}]
 Keep the notation as above. Then $c_{T_{1}},\dots, c_{T_{l}}$
 are the fundamental relative invariants of
 $(GL(\underline{n}), \Rep(Q, \underline{n}))$. 
\end{lemma}

\begin{remark}
 Recall that in the case of quivers of type $A$, the fundamental
 relative invariants are parametrized by the set
 $I_{\underline{n}}(Q)$ defined in
 Section~\ref{section:RelativeInvariants}. 
 The enumeration of 
 $I_{\underline{n}}(Q)$ is, of course, equivalent to the calculation of
 the simple objects $T_{1},\dots, T_{l}$.
 In \cite{ShmelkinJA}, Shmelkin gives an algorithm for the
 calculation of those simple objects, 
 and implemented the algorithm on 
 a computer program TETIVA~\cite{TETIVA}.
 By using TETIVA, one can also calculate the indecomposable decomposition
 $A_{0}= \bigoplus_{i=1}^{p} m_{i}I_{i}$ of a generic point of
 $A_{0}$, and so on. 
\end{remark}

\begin{lemma}
 \label{lemma:Restriction}
 Let $Q$ be a quiver without oriented cycles, 
 $(GL(\underline{n}), \Rep(Q,\underline{n}))$ 
 a triplet arising from quiver representations which is not
 necessarily a prehomogeneous vector space, and
 $f$ a relatively invariant polynomial of
  $(GL(\underline{n}), \Rep(Q,\underline{n}))$.
 Take an element $A_{1}$
 of $\Rep(Q,\underline{n})$ such that $GL(\underline{n})A_{1}$
 is closed in $\{A\in \Rep(Q,\underline{n})\, ;\,
 f(A)\neq 0\}$ and put $W_{1} = \Ext_{Q}(A_{1}, A_{1})$.
 Then the function $f_{A_{1}}$ on $W_{1}$ defined by
 \[
  f_{A_{1}}(w):= f(A_{1}+w)\qquad (w\in W_{1})
 \]
 is a constant function.
\end{lemma}

\begin{proof}
 By Derksen-Weyman~\cite[Theorem~1]{DerksenWeymanJAMS},
 we have $f=c_{B}$ for some
 $B\in \Rep(Q,\underline{m})$. Then, by
 Shmelkin~\cite[Theorem~11]{Shmelkin}, $A_{1}$ is a direct sum of
 simple objects in ${}^{\perp}B$. Since the perpendicular categories
 are closed under direct sums and extensions, all the elements
 in $A_{1}+W_{1} = A_{1}+\Ext_{Q}(A_{1},A_{1})$ belongs to
 ${}^{\perp}B$. By the definition of $c_{B}$ and ${}^{\perp}B$,
 we observe that $c_{B}$ does not vanish on $A_{1}+W_{1}$, 
 and this proves the lemma. 
\end{proof}

Now let us return to our case; we assume that $Q$ is
a quiver of type $A_{r}$.
We keep the notation in Sections~\ref{section:RelativeInvariants}--%
\ref{section:RepThOfQuiver}.
Let $\mathcal{O}^{(p,q)}$ be the  closed orbit in 
$\{A\in \Rep(Q, \underline{n})\, ;\,
f_{(p,q)}(A)\neq 0\}$, and take an element $A^{(p,q)}$ of
$\mathcal{O}^{(p,q)}$. We denote by
$GL(\underline{n})_{A^{(p,q)}}$ the isotropy subgroup of
$GL(\underline{n})$ at $A^{(p,q)}$, and by $W^{(p,q)}$
a $GL(\underline{n})_{A^{(p,q)}}$-invariant subspace of
$\Rep(Q,\underline{n})$
satisfying
\begin{equation}
\label{form:DefOfWpq}
 \Rep(Q, \underline{n}) = T_{A^{(p,q)}}(\mathcal{O}^{(p,q)})
 \oplus W^{(p,q)}. 
\end{equation}
Let
\[
 A^{(p,q)} \cong \bigoplus_{1\leq i\leq j\leq r} m_{ij}^{(p,q)}
 I_{ij}
\]
be the indecomposable decomposition of $A^{(p, q)}$, and set
\[
 \mathcal{C}^{(p,q)} =
 \left\{ [i, j] \;;\; 1\leq i\leq j\leq r, \,
 m_{ij}^{(p,q)}\neq 0
 \right\}.
\]

\begin{lemma}
 \label{lemma:HomAndExt}
%\begin{enumerate}
% \renewcommand{\labelenumi}{(\arabic{enumi})}
% \item For $[i, j], [k, l]\in \mathcal{C}^{(p,q)}$, we have
% \[
% %  \label{form:HomsBetweenIijIkl}  
% \Hom_{Q}(I_{ij}, I_{kl}) \cong
% \left\{
% \begin{array}{cll}
%  \C & & \text{if}\;\; [i, j] = [k, l]\\
%  0 && \text{otherwise}
% \end{array}.
% \right.   
% \]
% \item
 For $[i, j], [k, l]\in \mathcal{C}^{(p,q)}$, we have
 \[
   \Ext_{Q}(I_{ij}, I_{kl}) \cong 
 \left\{
 \begin{array}{ccl}
  \C & &
  \text{if there exists}\,\; a\in Q_{1}\;\,
   \text{such that}\, \; t(a)= j, h(a)=k
    \\
  \C & &
  \text{if there exists}\,\; a\in Q_{1}\;\,
   \text{such that}\, \; t(a)= i, h(a)=l
   \\
  0 & & \text{otherwise}
 \end{array}.
 \right.  
 \]      
%\end{enumerate}  
\end{lemma}

\begin{proof}
% (1) This follows from a general result 
% \cite[Proposition~3.2]{Shmelkin} on locally semi-simple
% representations. 
%
% \noindent (2)
 By considering the shape of exact lace diagrams and 
 the definition~\eqref{form:EulerForm},
 we observe that for $[i, j], [k, l]\in
 \mathcal{C}^{(p,q)}$,
 \[
 \left\langle I_{ij}, \, I_{kl}\right\rangle_{Q} =
 \left\{
 \begin{array}{ccl}
  1& & \text{if} \; [i, j] = [k, l]\\
  -1 & &
  \text{if there exists}\,\; a\in Q_{1}\;\,
   \text{such that}\, \; t(a)= j, h(a)=k \\
  -1 & &
 \text{if there exists}\,\; a\in Q_{1}\;\,
   \text{such that}\, \; t(a)= i, h(a)=l \\
  0 & & \text{otherwise}
 \end{array}.
 \right.
 \]
 Combining this with \eqref{form:Ringel} and
 \eqref{form:HomInLocallySemisimple},  we obtain
 the lemma. 
 We also note that the lemma can be proved by using 
 \cite[Lemma~3]{BR}. 
\end{proof}
By Lemmas~\ref{lemma:SchmelkinStrTh}
and \ref{lemma:HomAndExt}, we obtain the following lemma.

\begin{lemma}
 \label{lemma:SliceRepresentation}
 Let  $\mathcal{O}^{(p,q)}$ be the closed orbit in
 $\{A\in \Rep(Q, \underline{n})\, ;\, f_{(p,q)}(A)\neq 0\}$.
 Take $A^{(p,q)}\in \mathcal{O}^{(p,q)}$ and let
 $A^{(p,q)} \cong \oplus_{i, j} m_{ij}^{(p,q)}
 I_{ij}$ be the indecomposable decomposition of $A^{(p,q)}$.
 Then  the isotropy subgroup $GL(\underline{n})_{A^{(p,q)}}$ and
 the subspace $W^{(p,q)}$ satisfying \eqref{form:DefOfWpq}
 are given as follows:
 \begin{align*}  
  GL(\underline{n})_{A^{(p,q)}}&\cong 
   \prod_{1\leq i\leq j\leq r} GL (m_{ij}^{(p, q)}), \\
  W^{(p,q)} &\cong 
\bigoplus
 \begin{Sb}
  \exists a\in Q_{1} \;
  \text{s.t.}\\
  t(a) = j, \; h(a) = k
 \end{Sb} M(m_{kl}^{(p,q)}, m_{ij}^{(p,q)})  \oplus
 \bigoplus
 \begin{Sb}
  \exists a\in Q_{1} \;
  \text{s.t.}\\
  t(a) = i, \; h(a) = l
 \end{Sb} M(m_{kl}^{(p,q)}, m_{ij}^{(p,q)}).
 \end{align*}
\end{lemma}

\begin{example}
 \label{example:SliceRepresentation}
 Lemma~\ref{lemma:SliceRepresentation} tells us that 
 the structures of slice representations emerge from the
 exact lace diagrams. First let us consider the locally closed
 orbit $\mathcal{O}^{(3, 4)}$ in Example~\ref{exmp:FirstExampleIV}.
 Then, in view of Lemma~\ref{lemma:SliceRepresentation},
 we draw pictures like Figure~\ref{figure:SliceRepn1}
 and observe that
 \begin{align*}
  GL(\underline{n})_{A^{(3,4)}}&\cong GL(2)\times GL(5)\times
  GL(6)\times GL(2), \\
  W^{(3,4)}&\cong M(5,2)\oplus M(6, 5)\oplus M(2, 6).
 \end{align*}
 Here the lines ({\it not} arrows!) in the right picture denote
 the basis of $W^{(3,4)}$. The action of
 $GL(\underline{n})_{A^{(3,4)}}$ on $W^{(3,4)}$ is inherited from
 that of $(GL(\underline{n}), \Rep(Q,\underline{n}))$, and is 
 given by
 \[
  h\cdot w = \left(h_{2}w_{1}^{(3,4)}h_{1}^{-1},
 h_{3} w_{2}^{(3,4)}h_{2}^{-1},
 h_{4}w_{3}^{(3,4)} h_{3}^{-1}\right)
 \]
 for $h=(h_{1},h_{2},h_{3}, h_{4})\in GL(\underline{n})_{A^{(3,4)}}$
 and $w = (w_{1}^{(3,4)},w_{2}^{(3,4)},w_{3}^{(3,4)})\in W^{(3,4)}$. 
\begin{figure}[htbp]
  \begin{center}
  \begin{picture}(100,95)(0,0)
   \put(-40,80){$GL(\underline{n})_{A^{(3,4)}}$}
   \put(7,7){$\bullet$}
   \put(7,22){$\bullet$}
   \put(9,17){\oval(18,35)}
   \put(-7,-15){{\footnotesize $GL(2)$}}
   \put(30,7){$\bullet$}
   \put(30,22){$\bullet$}
   \put(30,37){$\bullet$}
   \put(30,52){$\bullet$}
   \put(30,67){$\bullet$}
   \put(33,40){\oval(18,80)}
   \put(22,-15){{\footnotesize $GL(5)$}}   
   \put(53,7){$\bullet$}
   \put(53,22){$\bullet$}
   \put(53,37){$\bullet$}
   \put(53,52){$\bullet$}
   \put(53,67){$\bullet$}
   \put(53,82){$\bullet$}
   \put(58,10){\vector(1,0){20}}
   \put(58,25){\vector(1,0){20}}
   \put(58,40){\vector(1,0){20}}
   \put(58,55){\vector(1,0){20}}
   \put(58,70){\vector(1,0){20}}
   \put(58,85){\vector(1,0){20}}
   \put(67,47){\oval(35,100)}
   \put(55,-15){{\footnotesize $GL(6)$}}      
   \put(76,7){$\bullet$}
   \put(76,22){$\bullet$}
   \put(76,37){$\bullet$}
   \put(76,52){$\bullet$}
   \put(76,67){$\bullet$}
   \put(76,82){$\bullet$}   
   \put(100,7){$\bullet$}
   \put(100,22){$\bullet$}
   \put(103,17){\oval(18,35)}
   \put(90,-15){{\footnotesize $GL(2)$}}         
  \end{picture}
   \qquad\qquad
   \begin{picture}(100,90)(0,0)
   \put(-20,80){$W^{(3,4)}$}
   \put(7,7){$\bullet$}
   \put(7,22){$\bullet$}
   \qbezier(9,10)(20,10)(30,10)
   \qbezier(10,10)(20,17.5)(30,24)
   \qbezier(10,10)(20,22.5)(30,36)
   \qbezier(10,10)(20,31)(30,52)
   \qbezier(10,10)(20,39)(30,67)
   \qbezier(9,25)(20,25)(30,25)
   \qbezier(9,25)(20,17)(30,10)
   \qbezier(9,25)(20,31)(30,36)
   \qbezier(9,25)(20,38)(30,52)
   \qbezier(9,25)(20,46)(30,67)            
   \put(30,7){$\bullet$}
   \put(30,22){$\bullet$}
   \put(30,37){$\bullet$}
   \put(30,52){$\bullet$}
   \put(30,67){$\bullet$}
   \qbezier(33,10)(43,10)(53,10)
   \qbezier(33,10)(43,17)(53,24)
   \qbezier(33,10)(43,25)(53,39)
   \qbezier(33,10)(43,32)(53,54)
   \qbezier(33,10)(43,39)(53,69)
   \qbezier(33,10)(43,47)(53,84)     
   \qbezier(33,25)(43,17)(53,10)
   \qbezier(33,25)(43,25)(53,25)
   \qbezier(33,25)(43,32)(53,39)
   \qbezier(33,25)(43,39)(53,54)
   \qbezier(33,25)(43,47)(53,69)
   \qbezier(33,25)(43,54)(53,84)    
   \qbezier(33,40)(43,25)(53,10)
   \qbezier(33,40)(43,32)(53,25)
   \qbezier(33,40)(43,40)(53,40)
   \qbezier(33,40)(43,47)(53,54)
   \qbezier(33,40)(43,55)(53,69)
   \qbezier(33,40)(43,64)(53,84)     
   \qbezier(33,55)(43,32)(53,10)
   \qbezier(33,55)(43,40)(53,25)
   \qbezier(33,55)(43,47)(53,40)
   \qbezier(33,55)(43,55)(53,55)
   \qbezier(33,55)(43,62)(53,69)
   \qbezier(33,55)(43,60)(53,84)         
   \qbezier(33,70)(43,40)(53,10)
   \qbezier(33,70)(43,46)(53,25)
   \qbezier(33,70)(43,55)(53,40)
   \qbezier(33,70)(43,63)(53,55)
   \qbezier(33,70)(43,70)(53,70)
   \qbezier(33,70)(43,77)(53,84)         
   \put(53,7){$\bullet$}
   \put(53,22){$\bullet$}
   \put(53,37){$\bullet$}
   \put(53,52){$\bullet$}
   \put(53,67){$\bullet$}
   \put(53,82){$\bullet$}
   \put(58,10){\vector(1,0){20}}
   \put(58,25){\vector(1,0){20}}
   \put(58,40){\vector(1,0){20}}
   \put(58,55){\vector(1,0){20}}
   \put(58,70){\vector(1,0){20}}
   \put(58,85){\vector(1,0){20}}            
   \put(76,7){$\bullet$}
   \put(76,22){$\bullet$}
   \put(76,37){$\bullet$}
   \put(76,52){$\bullet$}
   \put(76,67){$\bullet$}
   \put(76,82){$\bullet$}
   \qbezier(80,10)(90,10)(100,10)
   \qbezier(80,10)(90,17)(100,25)    
   \qbezier(80,25)(90,17)(100,10)
   \qbezier(80,25)(90,25)(100,25)
   \qbezier(80,40)(90,25)(100,10)
   \qbezier(80,40)(90,32)(100,25)    
   \qbezier(80,55)(90,32)(100,10)
   \qbezier(80,55)(90,40)(100,25)    
   \qbezier(80,70)(90,40)(100,10)
   \qbezier(80,70)(90,48)(100,25)    
   \qbezier(80,85)(90,48)(100,10)
   \qbezier(80,85)(90,56)(100,25)    
   \put(100,7){$\bullet$}
   \put(100,22){$\bullet$}
  \put(-3,-15){{\footnotesize $M(5,2)$}}
  \put(35,-15){{\footnotesize $M(6,5)$}}
  \put(75,-15){{\footnotesize $M(2,6)$}}        
  \end{picture}
 \end{center} 
 \caption{Slice representation associated with
 $\mathcal{O}^{(3,4)}$ in Example~\ref{exmp:FirstExampleIV}}
 \label{figure:SliceRepn1}
\end{figure}
  
 Second we consider the locally closed
 orbit $\mathcal{O}^{(1, 5)}$ in Example~\ref{exmp:FirstExampleIV}.
 Then, 
 we draw pictures like Figure~\ref{figure:SliceRepn2}
 and observe that
 \begin{align*}
  GL(\underline{n})_{A^{(1,5)}}&\cong GL(2)\times GL(3)\times
  GL(4)\times GL(4), \\
  W^{(1,5)}&\cong M(4,3)\oplus M(4, 4).
 \end{align*}
 The action of
 $GL(\underline{n})_{A^{(1,5)}}$ on $W^{(1,5)}$ is given by
 \[
  h\cdot w = \left(h_{3}w_{1}^{(1,5)}h_{2}^{-1},
 h_{4} w_{2}^{(1,5)}h_{3}^{-1}\right)
 \]
 for $h=(h_{1},h_{2},h_{3}, h_{4})\in GL(\underline{n})_{A^{(1,5)}}$
 and $w = (w_{1}^{(1,5)},w_{2}^{(1,5)})\in W^{(1,5)}$.  
 \begin{figure}[htbp]
  \begin{center}
  \begin{picture}(100,100)(0,0)
   \put(-45,90){$GL(\underline{n})_{A^{(1,5)}}$}
   \put(56,16){\oval(120,30)}
   \put(33,55){\oval(13,42)}
   \put(55.5,61){\oval(13,56)}
   \put(78.3,61){\oval(13,56)}
   \put(-35,15){{\footnotesize $GL(2)$}}
   \put(-7,50){{\footnotesize $GL(3)$}}
   \put(38,95){{\footnotesize $GL(4)$}}
   \put(70,95){{\footnotesize $GL(4)$}}   
   \put(7,7){$\bullet$}
   \put(7,22){$\bullet$}
   \put(10,10){\vector(1,0){20}}
   \put(10,25){\vector(1,0){20}}
   \put(30,7){$\bullet$}
   \put(30,22){$\bullet$}
   \put(30,37){$\bullet$}
   \put(30,52){$\bullet$}
   \put(30,67){$\bullet$}
   \put(34,10){\vector(1,0){20}}
   \put(34,25){\vector(1,0){20}}
   \put(53,7){$\bullet$}
   \put(53,22){$\bullet$}
   \put(53,37){$\bullet$}
   \put(53,52){$\bullet$}
   \put(53,67){$\bullet$}
   \put(53,82){$\bullet$}
   \put(58,10){\vector(1,0){20}}
   \put(58,25){\vector(1,0){20}}
   \put(76,7){$\bullet$}
   \put(76,22){$\bullet$}
   \put(76,37){$\bullet$}
   \put(76,52){$\bullet$}
   \put(76,67){$\bullet$}
   \put(76,82){$\bullet$}
   \put(82,10){\vector(1,0){20}}
   \put(82,25){\vector(1,0){20}}
   \put(100,7){$\bullet$}
   \put(100,22){$\bullet$}   
  \end{picture}
   \qquad\qquad
   \begin{picture}(100,100)(-10,0)
   \put(-30,90){$W^{(1,5)}$}
   \put(20,93){{\footnotesize $M(4,3)$}}
   \put(60,93){{\footnotesize $M(4,4)$}}      
   \put(7,7){$\bullet$}
   \put(7,22){$\bullet$}
   \put(10,10){\vector(1,0){20}}
   \put(10,25){\vector(1,0){20}}
   \put(30,7){$\bullet$}
   \put(30,22){$\bullet$}
   \put(30,37){$\bullet$}
   \put(30,52){$\bullet$}
   \put(30,67){$\bullet$}
   \qbezier(33,40)(43,40)(53,40)
   \qbezier(33,40)(43,47)(53,55)
   \qbezier(33,40)(43,55)(53,70)
   \qbezier(33,40)(43,62)(53,85)
   \qbezier(33,55)(43,47)(53,40)
   \qbezier(33,55)(43,55)(53,55)
   \qbezier(33,55)(43,62)(53,70)
   \qbezier(33,55)(43,70)(53,85)    
   \qbezier(33,70)(43,55)(53,40)
   \qbezier(33,70)(43,62)(53,55)
   \qbezier(33,70)(43,70)(53,70)
   \qbezier(33,70)(43,76.5)(53,85)            
   \qbezier(56,40)(66,40)(76,40)
   \qbezier(56,40)(66,47)(76,55)
   \qbezier(56,40)(66,55)(76,70)
   \qbezier(56,40)(66,62)(76,85)   
   \qbezier(56,55)(66,47)(76,40)
   \qbezier(56,55)(66,55)(76,55)
   \qbezier(56,55)(66,62)(76,70)
   \qbezier(56,55)(66,70)(76,85)            
   \qbezier(56,70)(66,55)(76,40)
   \qbezier(56,70)(66,62)(76,55)
   \qbezier(56,70)(66,70)(76,70)
   \qbezier(56,70)(66,76.5)(76,85)    
   \qbezier(56,85)(66,62)(76,40)
   \qbezier(56,85)(66,70)(76,55)
   \qbezier(56,85)(66,77)(76,70)
   \qbezier(56,85)(66,85)(76,85)    
   \put(34,10){\vector(1,0){20}}
   \put(34,25){\vector(1,0){20}}
   \put(53,7){$\bullet$}
   \put(53,22){$\bullet$}
   \put(53,37){$\bullet$}
   \put(53,52){$\bullet$}
   \put(53,67){$\bullet$}
   \put(53,82){$\bullet$}
   \put(58,10){\vector(1,0){20}}
   \put(58,25){\vector(1,0){20}}
   \put(76,7){$\bullet$}
   \put(76,22){$\bullet$}
   \put(76,37){$\bullet$}
   \put(76,52){$\bullet$}
   \put(76,67){$\bullet$}
   \put(76,82){$\bullet$}
   \put(82,10){\vector(1,0){20}}
   \put(82,25){\vector(1,0){20}}
   \put(100,7){$\bullet$}
   \put(100,22){$\bullet$}   
  \end{picture} 
  \end{center}
 \caption{Slice representation associated with
 $\mathcal{O}^{(1,5)}$ in Example~\ref{exmp:FirstExampleIV}}
 \label{figure:SliceRepn2}  
 \end{figure}
 
 Finally, we consider the locally closed
 orbit $\mathcal{O}^{(1, 4)}$ in Example~\ref{exmp:SecondExampleIV}.
 Then, 
 we draw pictures like Figure~\ref{figure:SliceRepn3}
 and observe that
 \begin{align*}
  GL(\underline{n})_{A^{(1,4)}}&\cong GL(2)\times GL(3)\times
  GL(4)\times GL(2), \\
  W^{(1,4)}&\cong M(2,4)\oplus M(4, 2).
 \end{align*}
 The action of
 $GL(\underline{n})_{A^{(1,4)}}$ on $W^{(1,4)}$ is given by
 \[
  h\cdot w = \left(h_{1}w_{1}^{(1,4)}h_{3}^{-1},
 h_{3} w_{2}^{(1,4)}h_{4}^{-1}\right)
 \]
 for $h=(h_{1},h_{2},h_{3}, h_{4})\in GL(\underline{n})_{A^{(1,4)}}$
 and $w = (w_{1}^{(1,4)},w_{2}^{(1,4)})\in W^{(1,4)}$.   
 \begin{figure}[htbp]
  \begin{center}
  \begin{picture}(130,130)(0,0)
   \put(-50,100){$GL(\underline{n})_{A^{(1,4)}}$}
   \put(-30,45){{\footnotesize $GL(2)$}}
   \put(20,10){{\footnotesize $GL(4)$}}
   \put(30,120){{\footnotesize $GL(3)$}}
   \put(115,45){{\footnotesize $GL(2)$}}   
   \put(20,47){\oval(35,30)}
   \put(43,86){\oval(35,50)}
   \put(67.5,32){\oval(35,66)}
   \put(102.5,47){\oval(13,30)}
   \put(7,37){$\bullet$}
   \put(7,52){$\bullet$}
   \put(10,40){\vector(1,0){20}}
   \put(10,55){\vector(1,0){20}}   
   \put(30,37){$\bullet$}
   \put(30,52){$\bullet$}
   \put(30,67){$\bullet$}
   \put(30,82){$\bullet$}
   \put(30,97){$\bullet$}
   \put(53,70){\vector(-1,0){20}}      
   \put(53,85){\vector(-1,0){20}}   
   \put(53,100){\vector(-1,0){20}}
   \put(53,7){$\bullet$}
   \put(53,22){$\bullet$}
   \put(53,37){$\bullet$}
   \put(53,52){$\bullet$}
   \put(53,67){$\bullet$}
   \put(53,82){$\bullet$}
   \put(53,97){$\bullet$}
   \put(56,10){\vector(1,0){20}}
   \put(56,25){\vector(1,0){20}}
   \put(56,40){\vector(1,0){20}}
   \put(56,55){\vector(1,0){20}}            
   \put(76,7){$\bullet$}
   \put(76,22){$\bullet$}
   \put(76,37){$\bullet$}
   \put(76,52){$\bullet$}
  \put(100,37){$\bullet$}
  \put(100,52){$\bullet$}
  \end{picture}
  \qquad\qquad
     \begin{picture}(100,105)(0,0)
   \put(-30,100){$W^{(1,4)}$}
   \put(25,-7){{\footnotesize $M(2,4)$}}
   \put(80,-7){{\footnotesize $M(4,2)$}}            
   \put(7,37){$\bullet$}
   \put(7,52){$\bullet$}
   \put(10,40){\vector(1,0){20}}
   \put(10,55){\vector(1,0){20}}   
   \put(30,37){$\bullet$}
   \put(30,52){$\bullet$}
   \put(30,67){$\bullet$}
   \put(30,82){$\bullet$}
   \put(30,97){$\bullet$}
   \put(53,70){\vector(-1,0){20}}      
   \put(53,85){\vector(-1,0){20}}   
   \put(53,100){\vector(-1,0){20}}
   \qbezier(33,55)(43,55)(53,55)
   \qbezier(33,55)(43,47)(53,40)
   \qbezier(33,55)(43,40)(53,25)
   \qbezier(33,55)(43,33)(53,10)            
   \qbezier(33,40)(43,47)(53,55)      
   \qbezier(33,40)(43,40)(53,40)
   \qbezier(33,40)(43,33)(53,25)
   \qbezier(33,40)(43,26)(53,10)            
   \put(53,7){$\bullet$}
   \put(53,22){$\bullet$}
   \put(53,37){$\bullet$}
   \put(53,52){$\bullet$}
   \put(53,67){$\bullet$}
   \put(53,82){$\bullet$}
   \put(53,97){$\bullet$}
   \put(56,10){\vector(1,0){20}}
   \put(56,25){\vector(1,0){20}}
   \put(56,40){\vector(1,0){20}}
   \put(56,55){\vector(1,0){20}}            
   \put(76,7){$\bullet$}
   \put(76,22){$\bullet$}
   \put(76,37){$\bullet$}
   \put(76,52){$\bullet$}
   \qbezier(80,55)(90,55)(100,55)
   \qbezier(80,55)(90,47)(100,40)
   \qbezier(80,40)(90,47)(100,55)
   \qbezier(80,40)(90,40)(100,40)
   \qbezier(80,25)(90,40)(100,55)
   \qbezier(80,25)(90,33)(100,40)
   \qbezier(80,10)(90,33)(100,55)
   \qbezier(80,10)(90,26)(100,40)      
  \put(100,37){$\bullet$}
  \put(100,52){$\bullet$}
  \end{picture}   
  \end{center}
 \caption{Slice representation associated with
 $\mathcal{O}^{(1,4)}$ in Example~\ref{exmp:SecondExampleIV}}
 \label{figure:SliceRepn3}    
 \end{figure}
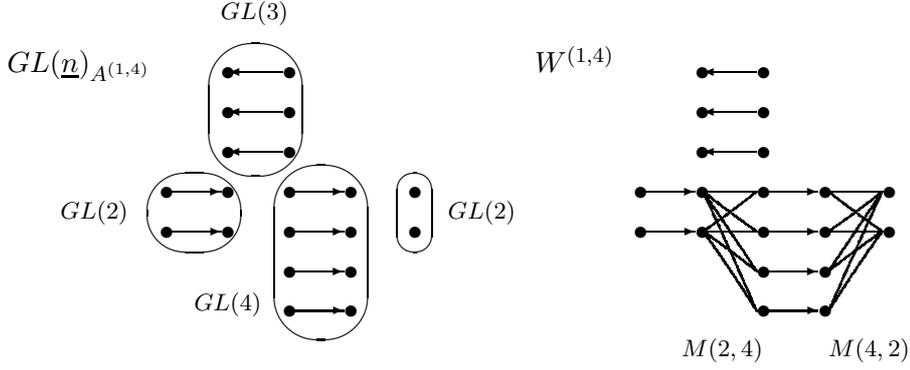
\end{example}

%In the final step to the proof of Theorem~\ref{theorem:MainTheorem}, 
%we need the following lemma. 
%
%\begin{lemma}
% \label{lemma:RestIsConst}
% For $w\in W^{(p,q)}$, we have $f_{(p,q)}(A^{(p,q)}+w) = 1$.
% That is, the restriction of $f_{(p,q)}$ to $W^{(p,q)}$ is
% constant. 
%\end{lemma}
%
%\begin{proof}
% For the case of the equioriented quivers, the lemma is trivial
% from the construction of the relative invariants. For the case of
% the non-equioriented quivers,
% the proof is similar to that of \eqref{form:Ypq2}.
% We also note that the lemma can be proved by using 
% \cite[Theorem~2.6 (4)]{SugiGeolocalB}.
%\end{proof}

Now 
we complete the calculation of \eqref{form:BfunOfFirstExample}
and \eqref{form:BfunOfSecondExample}, which is a prototype
of the proof of Theorem~\ref{theorem:MainTheorem}.

\begin{example}[proof of  \eqref{form:BfunOfFirstExample}]
 \label{example:ProofOfFirstExample}
 We keep the notation of Example~\ref{example:SliceRepresentation}.
 First we consider the slice representation
 $(GL(\underline{n})_{A^{(3,4)}}, W^{(3,4)})$.
 By Lemma~\ref{lemma:Restriction}, the restriction of
 $f_{1} = f_{(3,4)}$ to $W^{(3,4)}$ is constant.
 On the other hand, by the construction of $f_{2} = f_{(1,5)}$, 
 the restriction $f_{2}'$ of $f_{2}$ to $W^{(3,4)}$ is given by
 \[
  f_{2}'(w) = \det \left( w_{3}^{(3,4)} w_{2}^{(3,4)} w_{1}^{(3,4)}
 \right)
 \]
 for $w = (w_{1}^{(3,4)}, w_{2}^{(3,4)}, w_{3}^{(3,4)}) \in W^{(3,4)}$. 
 If we employ Notation~\ref{notation:RIasDetSymbol}, it follows that 
 \[
  f_{2}'(w)= \det\left(\ovt{2}\longrightarrow \ovt{5}\longrightarrow
 \ovt{6}\longrightarrow \ovt{2}
 \right).
 \]
 By Theorem~\ref{thm:Bfun}, we have
 \[
  b\left(\ovt{2}\longrightarrow \ovt{5}\longrightarrow
 \ovt{6}\longrightarrow \ovt{2}
 \right) = (s+1)(s+2)(s+4)(s+5)^2 (s+6).
 \]
 Then we apply Lemma~\ref{lemma:Localization} to
 $f^{\underline{m}} = f_{1}^{m_{1}} f_{2}^{m_{2}}$ and
 $\mathcal{O}^{(3,4)}$. Since $b_{f^{\underline{m}}}(s) =
 b_{\underline{m}}(\underline{m}s)$, 
 we see that 
 $b_{\underline{m}}(\underline{s})$ is divisible by
 \begin{equation}
  \label{form:FirstLocalBfun}
  [s_{2}+1]_{m_{2}} [s_{2}+2]_{m_{2}}[s_{2}+4]_{m_{2}}
 [s_{2}+5]_{m_{2}}^2 [s_{2}+6]_{m_{2}}.
 \end{equation}  
 This calculation can be interpreted graphically in the following way.
 Let us compare two exact lace diagrams in
 Figure~\ref{fig:LDinFirstExample}, and pick up the arrows of the
 diagram of $\mathcal{O}^{(1,5)}$ (right) which do not overlap with the
 arrows of the diagram of $\mathcal{O}^{(3,4)}$ (left). Transfer
 those arrows to the diagram representing the slice representation
 associated with $\mathcal{O}^{(3,4)}$ (see
 Figure~\ref{figure:SliceRepn1}). 
 Then we observe that these arrows form a smaller exact lace diagram as
 shown in Figure~\ref{figure:FactorsInSliceRepn}.
 The local $b$-function~\eqref{form:FirstLocalBfun} corresponds to
 this smaller exact lace diagram. 
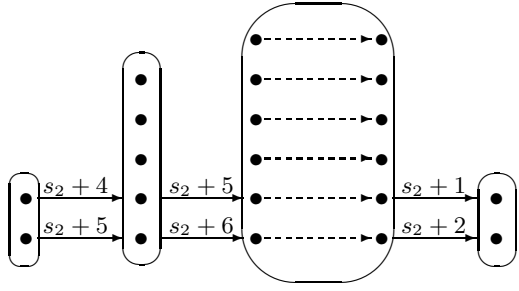
\begin{figure}[htbp]
  \begin{center}
  \begin{picture}(200,95)(0,0)
   \put(7,7){$\bullet$}
   \put(7,22){$\bullet$}
   \put(14,10){\vector(1,0){32}}
   \put(14,25){\vector(1,0){32}}
   \put(16,12){{\footnotesize $s_{2}+5$}}
   \put(16,27){{\footnotesize $s_{2}+4$}}   
   \put(9,17){\oval(11,35)}   
   \put(50,7){$\bullet$}
   \put(50,22){$\bullet$}
   \put(50,37){$\bullet$}
   \put(50,52){$\bullet$}
   \put(50,67){$\bullet$}
   \put(60,10){\vector(1,0){31}}
   \put(60,25){\vector(1,0){31}}
   \put(63,12){{\footnotesize $s_{2}+6$}}
   \put(63,27){{\footnotesize $s_{2}+5$}}      
   \put(53,40){\oval(14,80)}   
   \put(93,7){$\bullet$}
   \put(93,22){$\bullet$}
   \put(93,37){$\bullet$}
   \put(93,52){$\bullet$}
   \put(93,67){$\bullet$}
   \put(93,82){$\bullet$}
   \put(119,46){\oval(57,105)}   
%   \put(58,10){\vector(1,0){20}}
%   \put(58,25){\vector(1,0){20}}
%   \put(58,40){\vector(1,0){20}}
%   \put(58,55){\vector(1,0){20}}
%   \put(58,70){\vector(1,0){20}}
%   \put(58,85){\vector(1,0){20}}
   \multiput(99,10)(4,0){9}{\line(1,0){2}}
   \multiput(99,25)(4,0){9}{\line(1,0){2}}
   \multiput(99,40)(4,0){9}{\line(1,0){2}}
   \multiput(99,55)(4,0){9}{\line(1,0){2}}
   \multiput(99,70)(4,0){9}{\line(1,0){2}}
   \multiput(99,85)(4,0){9}{\line(1,0){2}}
   \put(137,10){\vector(1,0){2}}
   \put(137,25){\vector(1,0){2}}
   \put(137,40){\vector(1,0){2}}
   \put(137,55){\vector(1,0){2}}
   \put(137,70){\vector(1,0){2}}
   \put(137,85){\vector(1,0){2}}
   \put(140,7){$\bullet$}
   \put(140,22){$\bullet$}
   \put(140,37){$\bullet$}
   \put(140,52){$\bullet$}
   \put(140,67){$\bullet$}
   \put(140,82){$\bullet$}
   \put(147,10){\vector(1,0){31}}
   \put(147,25){\vector(1,0){31}}
   \put(150,12){{\footnotesize $s_{2}+2$}}
   \put(150,27){{\footnotesize $s_{2}+1$}}      
   \put(183,7){$\bullet$}
   \put(183,22){$\bullet$}
   \put(186,17){\oval(14,35)}   
  \end{picture}
 \end{center} 
 \caption{Smaller exact lace diagram}
 \label{figure:FactorsInSliceRepn}
\end{figure}
 
 Second we consider the slice representation
 $(GL(\underline{n})_{A^{(1,5)}}, W^{(1,5)})$. The restriction
 of $f_{2} = f_{(1,5)}$ to $W^{(1,5)}$
 is constant, and the restriction $f_{1}'$ of
 $f_{1}= f_{(3,4)}$ to $W^{(1,5)}$ is given by
 \[
  f_{1}'(w) = \det \left( w_{2}^{(1,5)}
 \right) 
 \]
 for $w= (w_{1}^{(1,5)}, w_{2}^{(1,5)}) \in W^{(1,5)}$.
 It follows that 
 $f_{1}'(w)= \det \left(
 \ovt{4}\longrightarrow \ovt{4}
 \right)$.
 Since $b\left(\ovt{4}\longrightarrow \ovt{4}\right)
 = (s+1)(s+2)(s+3)(s+4)$, we observe that
 $b_{\underline{m}}(\underline{s})$ is divisible by
 \begin{equation}
  \label{form:SecondLocalBfun}
  [s_{1}+1]_{m_{1}}[s_{1}+2]_{m_{1}}
   [s_{1}+3]_{m_{1}}[s_{1}+4]_{m_{1}}.
 \end{equation}
 Finally, let $\underline{m}= (1, 0)$ and $\underline{s} = (s, 0)$
 in \eqref{form:ByAfunInFirstExample}. Then it follows from 
 $b_{(1,0)}(s,0) = b_{f_{1}}(s) = (s+1)\cdots (s+6)$ that 
 $\{\alpha_{3,1}, \alpha_{3,2}\} = \{5, 6\}$. 
 Now all of 
 $\alpha_{j, r}$ in \eqref{form:ByAfunInFirstExample} have
 been determined and this completes the proof of
 \eqref{form:BfunOfFirstExample}, i.e.,
 \begin{align*}
    b_{\underline{m}}(\underline{s})&=
  [s_{1}+1]_{m_{1}} [s_{1}+2]_{m_{1}} [s_{1}+3]_{m_{1}}
  [s_{1}+4]_{m_{1}} \\
  &\quad \times
    [s_{2}+1]_{m_{2}}[s_{2}+2]_{m_{2}} [s_{2}+4]_{m_{2}}
  [s_{2}+5]_{m_{2}}^{2} [s_{2}+ 6]_{m_{2}}\\
  \nonumber
  &\quad \times
  [s_{1}+s_{2}+ 5]_{m_{1}+m_{2}} [s_{1}+s_{2}+ 6]_{m_{1}+m_{2}}.
 \end{align*}
\end{example}

\begin{remark}
 In fact, the argument in
 Example~\ref{example:ProofOfFirstExample} is redundant;
 it is enough to determine one of the two local
 $b$-functions \eqref{form:FirstLocalBfun} and
 \eqref{form:SecondLocalBfun}.
% It is intended to show both of Figures~\ref{figure:SliceRepn1}
% and \ref{figure:SliceRepn2}.
\end{remark}

\begin{example}[proof of  \eqref{form:BfunOfSecondExample}]
 We consider the slice representation
 $(GL(\underline{n})_{A^{(1,4)}}, W^{(1,4)})$.
 By Lemma~\ref{lemma:Restriction}, the restriction $f_{1}'$ of
 $f_{1} = f_{(1,4)}$ to $W^{(1,4)}$ is constant, 
 and by the construction of $f_{2} = f_{(2,5)}$,
 the restriction $f_{2}'$ of $f_{2}$ to $W^{(1,4)}$ is given by
 \[
 f_{2}'(w) = \det
 \left(w_{1}^{(1,4)} w_{2}^{(1,4)}
 \right) 
 \]
 for $w = (w_{1}^{(1,4)}, w_{2}^{(1,4)})\in W^{(1,4)}$.
 We have 
 $f_{2}'(w) = \det\left(\ovt{2}\longleftarrow \ovt{4}\longleftarrow
 \ovt{2} \right)$, and since
 $b\left(\ovt{2}\longleftarrow \ovt{4}\longleftarrow
 \ovt{2} \right) = (s+1)(s+2)(s+3)(s+4)$, we see that
 $b_{\underline{m}}(\underline{s})$ is divisible by
 \[
  [s_{2}+1]_{m_{2}}[s_{2}+2]_{m_{2}}[s_{2}+3]_{m_{2}}[s_{2}+4]_{m_{2}}.
 \]
 Then it follows from
 \[
  b_{(0, 1)}(0, s) = b_{f_{2}}(s) =
 (s+1)(s+2)(s+3)^2 (s+4)^2 (s+5)(s+6)(s+7)
 \]
 that $\{\alpha_{3,1},\dots, \alpha_{3,5}\}= \{3,4,5,6,7\}$ in
 \eqref{form:ByAfunInSecondExample}. We thus obtain
 \eqref{form:BfunOfFirstExample}, i.e.,
 \begin{align*}
   b_{\underline{m}}(\underline{s})&=
  [s_{1}+1]_{m_{1}}[s_{1}+2]_{m_{1}}[s_{1}+4]_{m_{1}}
  [s_{1}+5]_{m_{1}} \\
  &\quad\times
  [s_{2}+1]_{m_{2}}[s_{2}+2]_{m_{2}}[s_{2}+3]_{m_{2}}
  [s_{2}+4]_{m_{2}} \\
  \nonumber
  &\quad \times
  [s_{1}+s_{2}+ 3]_{m_{1}+m_{2}}
  [s_{1}+s_{2}+ 4]_{m_{1}+m_{2}}
  [s_{1}+s_{2}+ 5]_{m_{1}+m_{2}}\\
  \nonumber
  &\quad\times
  [s_{1}+s_{2}+ 6]_{m_{1}+m_{2}}
  [s_{1}+s_{2}+ 7]_{m_{1}+m_{2}}.  
 \end{align*}
\end{example} 

%\begin{remark}
% In fact, the argument in
% Example~\ref{example:ProofOfFirstExample} is redundant;
% it is enough to determine one of the two local
% $b$-functions \eqref{form:FirstLocalBfun} and
% \eqref{form:SecondLocalBfun}.
% It is intended to show both of Figures~\ref{figure:SliceRepn1}
% and \ref{figure:SliceRepn2}.
%\end{remark}

Now we are in a position to finish the proof of
our main theorem. 

\begin{proof}[Proof of Theorem~\ref{theorem:MainTheorem}] 
 {\bf Step (I)}. 
 First we consider the case of $l=2$, i.e.,
 we assume that there exist two fundamental relative invariants
 $f_{1}$ and $f_{2}$.
% It is shown in Section~\ref{section:RepThOfQuiver} that
% there exists a one-to-one correspondence between the arrows
% in the exact lace diagrams and the factors of $b$-functions,
% and in Section~\ref{section:Afunctions}, we have proved that 
% the ``superposition method'' is valid for $a$-functions.
 As shown in Section~\ref{section:Afunctions}, the $b$-function
 $b_{\underline{m}}(\underline{s})$ of $\underline{f} =
 (f_{1}, f_{2})$ is, in general, of the form
 \[
  b_{\underline{m}}(\underline{s}) =
 \prod_{r=1}^{\mu_{1}} [s_{1}+\alpha_{1, r}]_{m_{1}}
 \times
  \prod_{r=1}^{\mu_{2}} [s_{2}+\alpha_{2, r}]_{m_{2}}
 \times
 \prod_{r=1}^{\mu_{3}} [s_{1}+s_{2}+\alpha_{3, r}]_{m_{1}+m_{2}}
 \]
 with some $\alpha_{j, r}\in \Q_{>0}$.
 Thus it remains to determine $\alpha_{j,r}$. 
 By using $b_{(1,0)}(s, 0) = b_{f_{1}}(s)$,
 $b_{(0,1)}(0,s) = b_{f_{2}}(s)$ and 
 Theorem~\ref{thm:Bfun}, we obtain 
 $\{\alpha_{1,1},\dots, \alpha_{1,\mu_{1}}, \,
 \alpha_{3, 1},\dots, \alpha_{3,\mu_{3}}\}$ and
 $\{\alpha_{2,1},\dots, \alpha_{2,\mu_{2}}, \,
 \alpha_{3, 1},\dots, \alpha_{3,\mu_{3}}\}$
 {\it as  sets}. We will show that the arrangement
 such as
 Figures~\ref{fig:SuperpositionInFirstExample}
 and~\ref{fig:SuperpositionInSecondExample} gives the
 correct answer to our problem.
 For $i=1, 2$, let $\mathcal{O}^{(i)}$ be the  closed orbit
 in $\{f_{i}\neq 0\}$ and
 $(GL(\underline{n})_{A^{(i)}}, W^{(i)})$ the slice representation
 associated with $\mathcal{O}^{(i)}$. Let $f_{2}'$ be the restriction of
 $f_{2}$ to $W^{(1)}$ and
 \[
  b_{f_{2}'}(s) = \prod_{r=1}^{\mu_{2}'} (s+\beta_{r})
 \]
 be the $b$-function of $f_{2}'$. Then, by applying 
 Lemma~\ref{lemma:Localization} to $f^{\underline{m}} =
 f_{1}^{m_{1}}f_{2}^{m_{2}}$ and $\mathcal{O}^{(1)}$, we see that 
 $b_{\underline{m}}(\underline{s})$ is divisible by
 \[
  \prod_{r=1}^{\mu_{2}'} [s_{2}+\beta_{r}]_{m_{2}},
 \]
 and hence $\{\beta_{1},\dots, \beta_{\mu_{2}'}\}\subset
 \{\alpha_{2,1},\dots, \alpha_{2,\mu_{2}}\}$. 
 On the other hand,
 let $D^{(i)}\; (i=1,2)$ be the exact lace diagram representing
 $\mathcal{O}^{(i)}$. 
 Then, by the construction of the exact lace diagrams, 
 the arrows of $D^{(2)}$ which
 do not overlap with the arrows of $D^{(1)}$ form a smaller exact lace
 diagram, which is contained in the diagram representing
 the slice representation $(GL(\underline{n})_{A^{(1)}}, W^{(1)})$.
 The arrows which corresponds to $s_{2}+\beta_{1},\dots,
 s_{2}+\beta_{\mu_{2}'}$ also form an exact lace diagram
 in the diagram of $(GL(\underline{n})_{A^{(1)}}, W^{(1)})$. 
 This observation proves that  
 $\mu_{2} = \mu_{2}'$ and
 $\{\alpha_{2, 1},\dots, \alpha_{2, \mu_{2}}\} =
 \{\beta_{1},\dots, \beta_{\mu_{2}}\}$. Since 
 $\{\alpha_{2,1},\dots, \alpha_{2,\mu_{2}}, \,
 \alpha_{3, 1},\dots, \alpha_{3,\mu_{3}}\}$ is calculated as a set,
 we obtain $\{\alpha_{3, 1},\dots, \alpha_{3,\mu_{3}}\}$, and
 also  $\{\alpha_{1, 1},\dots, \alpha_{1,\mu_{1}}\}$.
 This means that our arrangement is correct for the case of 
 $l=2$. 

 \bigskip
 
 \noindent
 {\bf Step (II)}. For general $l$, we use an induction on $l$.
 Let $b_{\underline{m}}(\underline{s})$ be the $b$-function of
 $\underline{f}=(f_{1},\dots, f_{l})$. We decompose
 $b_{\underline{m}}(\underline{s})$ as
 \begin{equation}
  \label{form:WhenL>2}
  b_{\underline{m}}(\underline{s}) = c_{\underline{m}}(\underline{s})
 \cdot \prod_{r=1}^{\mu}
 [s_{1}+\cdots + s_{l}+\alpha_{r}]_{m_{1}+\cdots + m_{l}},
 \end{equation}
 where $c_{\underline{m}}(\underline{s})$ does not contain the
 factor of the form $s_{1}+\cdots + s_{l}+\alpha$.
 Let $\widetilde{D}$ be the diagram obtained by superposing all the
 exact lace diagrams $D^{(i)}$ representing the locally closed orbit
 $\mathcal{O}^{(i)} \; (i=1,\dots, l)$. Fix an index $k$
 ($k=1,\dots, l$). Then the arrows of $\widetilde{D}$ which do not
 overlap with the arrows of $D^{(k)}$ corresponds to the $b$-function
 $b_{\underline{m}}^{(k)}(\underline{s})$ of the slice representation
 $(GL(\underline{n})_{A^{(k)}}, W^{(k)})$. By induction hypothesis,
 $b_{\underline{m}}^{(k)}(\underline{s})$ can be calculated by the
 superposition method. By repeating this for $k=1,\dots, l$,
 we obtain $c_{\underline{m}}(\underline{s})$.
 For $\varepsilon_{i} = (0,\dots, \overset{i}{1},\dots, 0)$, we have
 $b_{\varepsilon_{i}}(\varepsilon_{i}s) = b_{f_{i}}(s)$, and
 by using this with Theorem~\ref{thm:Bfun}, we obtain
 $\{\alpha_{1},\dots, \alpha_{\mu}\}$ in \eqref{form:WhenL>2}.
 This completes the proof of
 Theorem~\ref{theorem:MainTheorem}.
\end{proof}

We conclude the present paper by giving an example with $l=3$. 

\begin{example}
 Let us consider the following quiver of type $A_{7}$:
 \[
  Q:\quad \vt{1}\lra \vt{2}\lla \vt{3}\lra \vt{4}\lra
 \vt{5}\lra \vt{6}\lla \vt{7}
 \]
 We put $\underline{n} = (1,3,5,4,4,3,1)$. Then
 $(GL(\underline{n}), \Rep(Q,\underline{n}))$ has 3 fundamental
 relative invariants $f_{1}:= f_{(1, 5)}, f_{2} := f_{(3,4)},
 f_{3}:= f_{(2,7)}$, which are given
 explicitly as follows:
 \begin{align*}
  f_{1}(v) &= \det
 \begin{pmatrix}
  X_{2,1} & X_{2, 3} \\
  O & X_{6, 5}X_{5, 4}X_{4,3}
 \end{pmatrix}, \\
 f_{2}(v) &= \det X_{5, 4},  \\
 f_{3}(v) &=
 \det
 \begin{pmatrix}
  X_{2, 3}& O \\
  X_{6, 5}X_{5, 4}X_{4, 3} & X_{6, 7}
 \end{pmatrix}  
 \end{align*}
 for $v = \left(
 X_{2, 1}, \, X_{2, 3}, \, X_{4, 3}, \,
 X_{5, 4}, \, X_{6, 5}, \, X_{6, 7}
 \right)\in \Rep(Q, \underline{n})$.
 We superpose the exact lace diagrams corresponding to the
 locally closed orbits,
 and the resultant diagram is given as 
 Figure~\ref{figure:ExampleL=3}.
 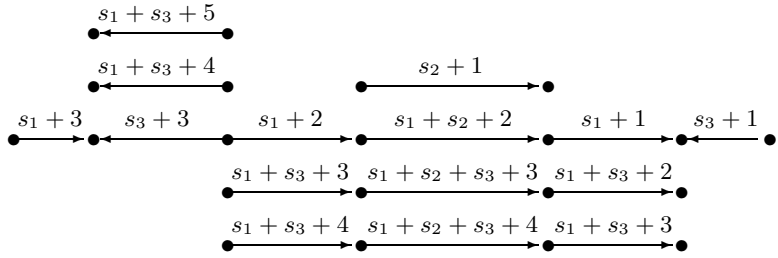
\begin{figure}[htbp]
  \begin{center}
  \begin{picture}(300,100)(0,0)
  \put(0,40){$\bullet$}
  \put(3,43){\vector(1,0){26}}
  \put(4,48){{\footnotesize$s_{1}+3$}}   
  \put(30,40){$\bullet$}
  \put(30,60){$\bullet$}
  \put(30,80){$\bullet$}
  \put(81,43){\vector(-1,0){46}}      
  \put(81,63){\vector(-1,0){46}}      
  \put(81,83){\vector(-1,0){46}}
  \put(44,48){{\footnotesize$s_{3}+3$}}
  \put(34,68){{\footnotesize$s_{1}+s_{3}+4$}}
  \put(34,88){{\footnotesize$s_{1}+s_{3}+5$}}  
  \put(80,0){$\bullet$}           
  \put(80,20){$\bullet$}        
  \put(80,40){$\bullet$}     
  \put(80,60){$\bullet$}  
  \put(80,80){$\bullet$}
  \put(83,3){\vector(1,0){46}}         
  \put(83,23){\vector(1,0){46}}      
  \put(83,43){\vector(1,0){46}}
  \put(94,48){{\footnotesize$s_{1}+2$}}
  \put(84,28){{\footnotesize$s_{1}+s_{3}+3$}}
  \put(84,8){{\footnotesize$s_{1}+s_{3}+4$}}         
  \put(130,0){$\bullet$}           
  \put(130,20){$\bullet$}        
  \put(130,40){$\bullet$}     
  \put(130,60){$\bullet$}
  \put(133,3){\vector(1,0){66}}
  \put(133,23){\vector(1,0){66}}
  \put(133,43){\vector(1,0){66}}
  \put(133,63){\vector(1,0){66}}
  \put(154,68){{\footnotesize$s_{2}+1$}}
  \put(145,48){{\footnotesize$s_{1}+s_{2}+2$}}
  \put(135,28){{\footnotesize$s_{1}+s_{2}+s_{3}+3$}}
  \put(135,8){{\footnotesize$s_{1}+s_{2}+s_{3}+4$}} 
  \put(200,0){$\bullet$}           
  \put(200,20){$\bullet$}        
  \put(200,40){$\bullet$}     
  \put(200,60){$\bullet$}
  \put(203,3){\vector(1,0){46}}
  \put(203,23){\vector(1,0){46}}
  \put(203,43){\vector(1,0){46}}
  \put(215,48){{\footnotesize$s_{1}+1$}}
  \put(205,28){{\footnotesize$s_{1}+s_{3}+2$}}
  \put(205,8){{\footnotesize$s_{1}+s_{3}+3$}} 
  \put(250,0){$\bullet$}           
  \put(250,20){$\bullet$}        
  \put(250,40){$\bullet$}
  \put(283,40){$\bullet$}
  \put(281,43){\vector(-1,0){26}}
  \put(257,48){{\footnotesize$s_{3}+1$}}   
  \end{picture}
  \end{center}
  \caption{Example with $l=3$}
  \label{figure:ExampleL=3}
 \end{figure}   
 Hence we see that the $b$-function $b_{\underline{m}}
 (\underline{s})$ of $\underline{f}= (f_{1},f_{2},f_{3})$ is
 given by
 \begin{align*}
  b_{\underline{m}}(\underline{s})&=
  [s_{1}+1]_{m_{1}}[s_{1}+2]_{m_{1}}[s_{1}+3]_{m_{1}}
  [s_{2}+1]_{m_{2}}[s_{3}+1]_{m_{3}}[s_{3}+3]_{m_{3}}  \\
  & \quad \times
  [s_{1}+s_{2}+2]_{m_{1}+m_{2}} [s_{1}+s_{3}+ 2]_{m_{1}+m_{3}}
  [s_{1}+s_{3}+3]_{m_{1}+m_{3}}^2 \\
  &\quad \times
  [s_{1}+s_{3}+ 4]_{m_{1}+m_{3}}^2
  [s_{1}+s_{3}+ 5]_{m_{1}+m_{3}} \\
  &\quad \times
  [s_{1}+s_{2}+s_{3}+3]_{m_{1}+m_{2}+m_{3}}
  [s_{1}+s_{2}+s_{3}+4]_{m_{1}+m_{2}+m_{3}}.
 \end{align*}
 
\end{example}

\address{
Kazunari Sugiyama \\
Department of Mathematics \\
Chiba Institute of Technology\\
2-1-1, Shibazono, Narashino, Chiba, 275-0023 
Japan 
}
{skazu@sky.it-chiba.ac.jp}

\end{document}